\documentclass[11pt,reqno]{amsart}
\usepackage{hyperref}

\allowdisplaybreaks[4]
\usepackage{amssymb}
\usepackage{mathrsfs}
\usepackage{amsthm}
\usepackage{amsmath}
\numberwithin{equation}{section}
\usepackage{bm}
\usepackage{bbm}
\usepackage{graphicx}
\usepackage{booktabs}
\usepackage{float}
\usepackage{subfig}
\usepackage{geometry}
\usepackage{color}
\theoremstyle{plain}
\newtheorem{lemma}{Lemma}[section]
\newtheorem{thm}[lemma]{Theorem}

\newtheorem{coro}[lemma]{Corollary}
\newtheorem{remark}[lemma]{Remark}
\newtheorem{prop}[lemma]{Proposition}
\newtheorem{assumption}{Assumption}

\begin{document}
\title[Central limit theorem for full discretization]{Central limit theorem for full discretization of parabolic SPDE}
\author{Chuchu Chen, Tonghe Dang, Jialin Hong, Tau Zhou}
\address{LSEC, ICMSEC,  Academy of Mathematics and Systems Science, Chinese Academy of Sciences, Beijing 100190, China,
\and 
School of Mathematical Sciences, University of Chinese Academy of Sciences, Beijing 100049, China}
\email{chenchuchu@lsec.cc.ac.cn; dangth@lsec.cc.ac.cn; hjl@lsec.cc.ac.cn;
zt@lsec.cc.ac.cn}
\thanks{This work is funded by the National key R$\&$D Program of China under Grant (No. 2020YFA0713701), National Natural Science Foundation of China (No. 11971470,
No. 11871068, No. 12031020, No. 12022118), and by Youth Innovation Promotion Association CAS}
\begin{abstract}
 
 In order to characterize the fluctuation between the ergodic limit and the time-averaging estimator of a full discretization in a quantitative way, we
establish a central limit theorem for the full discretization of the parabolic stochastic partial differential equation. The theorem shows that the normalized 
time-averaging estimator converges to a normal distribution with the variance being the same as that of the continuous case, 
where the scale used for the normalization  corresponds to the temporal strong convergence order of the considered full discretization.
A key ingredient in the proof is to extract an appropriate 
martingale difference series sum from the normalized time-averaging estimator so that the convergence to the normal distribution of such a sum and the convergence to zero in probability
 of the remainder are well balanced. The main novelty of our method to balance the convergence lies in proposing an appropriately modified Poisson equation so as to possess the space-independent regularity estimates.
 As a byproduct, the full discretization is shown to fulfill the weak law of large numbers, namely,
 the time-averaging estimator converges to the ergodic limit in probability. 
\end{abstract}
\maketitle
\section{Introduction}
This paper focuses on 
characterizing quantitatively the fluctuation between the ergodic limit and the time-averaging estimator of a full discretization, namely, establishing
the central limit theorem (CLT) for the full discretization of the parabolic stochastic partial differential equation (SPDE).
Concerning the following parabolic SPDE 
driven by additive noise
\begin{align}\label{spde}
\mathrm{d}X(t)=AX(t)\mathrm{d}t+F(X(t))\mathrm{d}t+\mathrm{d}W(t), \quad X(0)=X_0\in H,
\end{align}
where $H:=L^2((0,1);\mathbb{R})$ endowed with the usual inner product $\langle\cdot,\cdot\rangle$ and the norm $\|\cdot\|,$ and $A:\mathrm{Dom}(A)\subset H\rightarrow H$ is the Laplacian with homogeneous Dirichlet boundary conditions. The stochastic process $\{W(t)\}_{t\ge 0}$ is a generalized $Q$-Wiener process on a filtered probability space $(\Omega,\mathcal{F},\left\{\mathcal{F}_t\right\}_{t\ge 0},\mathbb{P})$. Precise assumptions of the stochastic process $\{W(t)\}_{t\ge 0}$ are given in Assumption \ref{assump_Q}. The drift coefficient $F$ is the Nemytskii operator satisfying a dissipation condition, which ensures that 
 \eqref{spde} possesses a unique ergodic invariant measure $\pi$.
 One of the fundamental
problems in the classical probability theory, which is closely related to the invariant measure $\pi$, is the asymptotic behavior of the functional $\int_0^Th(X(t))\mathrm{d}t$ as $T\rightarrow \infty$ for $h:H\rightarrow \mathbb{R}$ in a 
class of reference functions.
For \eqref{spde}, it is proved in \cite{clt12} that under some conditions,
the time average $\frac{1}{T}\int_0^Th(X(t))\mathrm{d}t$ converges to the ergodic limit $\pi(h)$ in probability and the normalized fluctuation around $\pi(h)$ can be described by a centered Gaussian random variable,
 i.e. the exact solution fulfills the weak law of large numbers (LLNs) $$\lim_{T\rightarrow\infty}\frac{1}{T}\int_0^Th(X(t))\mathrm{d}t= \pi(h) \text{ in probability,}$$ 
and 
the CLT $$\lim_{T\rightarrow\infty}\frac{1}{\sqrt{T}}\int_0^T[h(X(t))-\pi(h)]\mathrm{d}t=\mathcal{N}(0,\tilde{\sigma}^2)\text{ in distribution.}$$
The variance $\tilde{\sigma}^2$ can be written as $\tilde{\sigma}^2=\pi(\|Q^{\frac{1}{2}}D\phi(\cdot)\|^2),$ where $\phi$ is a solution of the associated Poisson's equation
$
 \mathcal{L}\phi(x)=h(x)-\pi(h)
$
with $\mathcal{L}$ being the generator of \eqref{spde}. See Proposition \ref{CLTexact} for the details. 

In general, ergodic limits of SPDEs can not be solved exactly, thus numerical approximations are proposed and studied 
(see e.g. \cite{B14, BK17, CHW17, czh20, CHS21, hwbook, HWZ17} and references therein). 
One common approach is constructing the time-averaging estimator $m^{-1}\sum_{k=0}^{m-1}h(\mathbb{X}(t_k))$ of a numerical discretization $\{\mathbb{X}(t_k)\}_{k\ge 0}$ of \eqref{spde} with $m>0$ being an integer. Under some conditions, this time-averaging estimator converges to $\pi(h)$ in the mean sense with a certain order with respect to mesh sizes and $m$. 
 It is interesting to further study  
such a convergence in probability and 
characterize the fluctuation
 in a quantitative way for a numerical discretization of \eqref{spde}, namely, 
to solve the following questions:
\begin{itemize}
\item[(\romannumeral1)] 
Does the time-averaging estimator converge to the ergodic limit in probability? Namely, dose the weak LLNs hold?
\item[(\romannumeral2)] Does the time-averaging estimator under certain normalization converge to some normal distribution? Namely, does the CLT hold?
\item[(\romannumeral3)] If (\romannumeral2) holds, what is the 
relation of the variances of the normal distributions corresponding to the numerical discretization and the exact solution?
\end{itemize}

For finite dimensional stochastic differential equations, much progress has been made on the above questions, see e.g. \cite{highorder, cltsde, siam10, AAP12} and references therein.
 For instance, in \cite{siam10}, error estimates (both in the mean square sense and in the a.s.-sense) for time-averaging estimator of the generalized Euler--Maruyama (EM) method on the torus are obtained, so
 both the weak and strong LLNs are implied. Authors in \cite{AAP12} establish the CLT for the Euler scheme with decreasing step of a wide class of Brownian ergodic diffusions.
In \cite{highorder}, sufficient conditions for approximating the invariant measure with high order of accuracy for a numerical method are
proposed and the order is proved in the a.s.-sense, which implies the strong LLNs. 
Recently, 
authors in \cite{cltsde} prove the CLT and self-normalized Cram\'er-type moderate deviation of the time-averaging estimator for EM method. 
However,
to our knowledge,
less is known for the LLNs and the CLT of numerical discretizations of infinite dimensional systems.
The main purpose of this paper is to study the weak LLNs and the CLT for a full discretization of \eqref{spde}.

To this end, we first apply the spectral Galerkin method in spatial direction and the exponential integrator in temporal direction to \eqref{spde} to obtain the full discretization. Under the dissipation assumption, this full discretization admits a unique invariant measure. Based on the time-independent a priori estimates of the full discretization, the time-independent strong convergence order of this full discretization is obtained under the usual correlation assumption $\|(-A)^{\frac{\beta-1}{2}}Q^{\frac{1}{2}}\|_{\mathcal{L}_2(H)}<\infty\,(0<\beta\leq 1)$, 
i.e., order $\beta$ in spatial direction and  $\frac{\beta}{2}$ in temporal direction. With this full discretization, we construct the time-averaging estimator $\Pi^{N}_{\tau}(h):=m^{-1}\sum_{k=0}^{m-1}h(X^N_k)$ with $\{X^N_k\}_{k\in\mathbb{N}},N$ and $\tau$ being the numerical solutions, the dimension of the spectral Galerkin projection space and the temporal step size of the full discretization, respectively. The main contribution of this paper is to prove the following CLT: For $h\in \mathcal{C}_b^4(H;\mathbb{R}),$
$$\tau^{-\frac{\beta}{2}}\big[\Pi^N_{\tau}(h)-\pi(h)\big]\rightarrow \mathcal{N}\big(0,\pi(\|Q^{\frac{1}{2}}D\phi(\cdot)\|^2)\big)$$ in distribution as $\tau$ goes to zero under a coupling condition for $\tau$ and $N$ and the assumption  $m:=\lfloor\tau^{-1-\beta}\rfloor$.
 We remark that 
  the scale $\tau^{\frac{\beta}{2}}$ in the normalized time-averaging estimator $\tau^{-\frac{\beta}{2}}\big[\Pi^N_{\tau}(h)-\pi(h)\big]$ corresponds to the temporal strong convergence order of the full discretization, and the variance is the same as that for the case of the exact solution. We refer to Theorem \ref{main} for details. As a byproduct, the weak LLNs holds, i.e.,
$[\Pi^N_{\tau}(h)-\pi(h)]\rightarrow 0$
in probability as $\tau$ goes to zero.

In order to prove the desired CLT, we need to extract an appropriate 
martingale difference series sum from the normalized time-averaging estimator so that the convergence to the normal distribution of such a sum and the convergence to zero in probability
 of the remainder are well balanced.
 The main novelty for proving the CLT lies in proposing the modified Poisson equation
$
 \mathcal{L}\phi_{\tau,\epsilon}(x)=\tau^{\frac{\beta}{2}-\epsilon}\phi_{\tau,\epsilon}(x)+h(x)-\pi(h)
$, which depends on a real parameter  $\epsilon\in(0,\frac{\beta}{2})$. This modified Poisson's equation has some features as follows.
First, the martingale difference series sum $\mathcal{M}^N_{\tau}$ is extracted by using this modified Poisson's equation and is shown to converge to $\mathcal{N}\big(0,\pi(\|Q^{\frac{1}{2}}D\phi(\cdot)\|^2)\big)$ in distribution as $\tau$ goes to zero for any fixed choice $\epsilon\in(0,\frac{\beta}{2}).$ Second, this modified Poisson's equation is an approximation of the associated Poisson equation and is shown to possess both the space-dependent and the space-independent regularity estimates. Third, by fully utilizing these regularity estimates, we can give a small range of choices for $\epsilon$ so that the remainder after the extraction converges to zero in probability as $\tau$ goes to zero.
We would like to mention that a crucial step to show the convergence in probability for the remainder is establishing sufficient conditions such that 
exponential functionals of Brownian motions are bounded in the mean sense, by making the best of the independent increments property of terms with Brownian motion and properties of the conditional expectation.
Through selecting a proper parameter $\epsilon$ in an overlapped interval, the desired CLT is finally obtained.

The outline of this paper is as follows. In the next section, some preliminaries and the full discretization of \eqref{spde} are given, and three Poisson-type equations as well as their properties are introduced. We also present the CLT and the weak LLNs for the exact solution of \eqref{spde}. In Section \ref{mainresult}, we give the main result of this article and show the strategy for proving the main theorem.
In Section \ref{sec3}, we prove properties of Poisson-type equations.
Section \ref{sec4} is devoted to proofs of propositions listed in Section \ref{mainresult}. The proofs of the a priori estimates and the convergence order of the full discretization  are given in the appendix. 
\section{Preliminaries}
In this section, we first give assumptions on $A,\, F,\,W$ and the initial datum. Then we introduce the full discretization, 
and give some properties, including the existence and uniqueness of the invariant measure, the time-independent a priori estimates as well as the convergence order of the full discretization. Three Poisson-type equations as well as their properties are introduced.
At the end of this section, we present the CLT and the weak LLNs for the exact solution of \eqref{spde}. 

Throughout this article, we use $C$ to denote an unspecified positive and finite constant which may not be the same in each occurrence. More specific constants which are dependent on certain parameters $a,b,\ldots$ are numbered as $C(a,b,\ldots).$ Let the parameter $\gamma$ be a sufficiently small positive number which may vary from one place to another.
In the sequel, denote by $\overset{d}{\longrightarrow}$ and $\overset{\mathbb{P}}{\longrightarrow}$ the convergence in distribution and in probability, respectively.

\subsection{Settings}
Let $\mathcal{O}:=(0,1),$ and $H^s(\mathcal{O}),$ $H^s_0(\mathcal{O})$ with $s>0$ be the usual Sobolev spaces.
Then the domain of the operator $A$ is ${\rm Dom}(A):=H^2(\mathcal{O})\cap H^1_0(\mathcal{O})$, and there is a sequence of real numbers $\lambda_i=\pi^2 i^2, i\in\mathbb{N}_+$ and an orthonormal basis $\left\{e_i(x)=\sqrt{2}\sin(i\pi x),x\in\mathcal{O}\right\}_{i\in\mathbb{N}_+}$ of $H$ such that $-Ae_i=\lambda_ie_i.$ It is known that $-A$ is a positive, self-adjoint and densely defined operator, and $A$ generates the analytic semigroup $\{S(t):=e^{tA},t\ge 0\}$ on $H$. For $s\in\mathbb{R},$ define the Hilbert space $\dot{H}^{s}:=\mathrm{Dom}((-A)^{\frac{s}{2}}),$ equipped with the inner product $\langle\cdot,\cdot\rangle_{s}:=\langle (-A)^{\frac{s}{2}}\cdot,(-A)^{\frac{s}{2}}\cdot\rangle$ and the norm $\|\cdot\|_{s}:=\langle\cdot,\cdot\rangle^{\frac{1}{2}}_{s}$. For $s_2>s_1\ge 0,$ $\|\cdot\|_{s_1}\leq  \|\cdot\|_{s_2}.$ Denote by $\mathcal{C}^k_b:=\mathcal{C}^k_b(H;\mathbb{R})$ the bounded functions that are $k$th continuously differentiable with bounded derivatives up to order $k$, and $\|h\|_{k,\infty}:=\sup_{0\leq j\leq k}\{\|h\|_j\}$ with $\|h\|_0=\sup_{x\in H}|h(x)|,\;\|h\|_1=\sup_{x\in H}\|Dh(x)\|$ and so on. Furthermore, $\|\cdot\|_{\mathcal{L}(H)},\|\cdot\|_{\mathcal{L}_2(H)}$ and $\mathrm{Tr(\cdot)}$ denote the usual operator norm, Hilbert--Schmidt operator norm and the trace of an operator, respectively.

It is well-known that there is a positive constant $C$ such that
\begin{align}
\|(-A)^{\nu}S(t)\|_{\mathcal{L}(H)}&\leq Ct^{-\nu}e^{-\frac{\lambda_1}{2}t},\quad t>0,\;\nu\ge 0,\label{semigroup1}\\
\|(-A)^{-\rho}(S(t)-S(s))\|_{\mathcal{L}(H)}&\leq C(t-s)^{\rho}e^{-\frac{\lambda_1}{2}s},\quad 0\leq s\le t,\;\rho\in[0,1],\label{semigroup2}
\end{align}
and
\begin{align}\label{semigroup3}
\int_s^t\|(-A)^{\frac{\rho}{2}}S(t-r)\|_{\mathcal{L}(H)}^2\,\mathrm{d}r\leq C(t-s)^{1-\rho},\quad 0\leq s\leq t,\;\rho\in[0,1].
\end{align}

We make the following assumptions on the drift coefficient $F$, the stochastic process $\{W(t)\}_{t\ge 0},$ and the initial datum $X_0$.
\begin{assumption}\label{assumpF_1}
Assume that for any $x,y\in H,$
\begin{align*}
\langle x-y,F(x)-F(y)\rangle \leq K\|x-y\|^2 
\end{align*}
with $K<\lambda_1$ and 
$\|DF(y)\|\leq L_F$
 with $L_{F}>0.$
\end{assumption}

\begin{assumption}\label{assump_Q}
Assume that $\{W(t)\}_{t\ge 0}$ is a generalized $Q$-Wiener process on a filtered probability space $(\Omega,\mathcal{F},\{\mathcal{F}_t\}_{t\ge 0},\mathbb{P})$, which can be represented as 
\begin{align*}
W(t):=\sum_{j=1}^{\infty}Q^{\frac{1}{2}}e_j\beta_j(t),
\end{align*}
where $\{\beta_j(t),t\ge 0\}_{j\in\mathbb{N}_+}$ is a sequence of independent real-valued standard Brownian motions. Here, $Q\in\mathcal{L}(H)$ is self-adjoint, positive definite and invertible,
which commutes with $A$ and satisfies the following two conditions:
\begin{itemize}
\item
[(\romannumeral1)] 
$
\|(-A)^{\frac{\beta-1}{2}}Q^{\frac{1}{2}}\|_{\mathcal{L}_2(H)}<\infty
$
for some $\beta\in(0,1]$ and $\|Q^{-\frac{1}{2}}(-A)^{-\frac{1}{2}}\|_{\mathcal{L}(H)}<\infty$.
\item
[(\romannumeral2)] $\mathrm{Range}((-A)^{-\frac{\varepsilon}{2}})\subset\mathrm{Range}(Q^{\frac{1}{2}})$ with some $\varepsilon<1.$
\end{itemize}
\end{assumption}

For instance, $Q=(-A)^{\kappa},$ $-1<\kappa<\frac{1}{2}-\beta$ satisfies Assumption \ref{assump_Q}. Then by choosing a suitable $\kappa,$ two important cases are included, namely, the trace-class noise case for $\beta=1$ and the space-time white noise case (i.e., $Q=I$) for $\beta<\frac{1}{2}.$ We remark that the condition (\romannumeral2) in Assumption \ref{assump_Q} is needed to have the regularizing effect of the Markov semigroup associated to \eqref{spde} (see \cite[Hypothesis $4.1$]{pde01}). 
\begin{assumption}\label{assumpX_0}
Assume that the $\mathcal{F}_0$-measurable initial datum $X_0$ satisfies that $\mathbb{E}e^{\vartheta\|X_0\|^2}<\infty$ for some $\vartheta>0$, and $X_0\in L^{\vartheta_1}(\Omega;\dot{H}^{\beta})$ for $\vartheta_1\ge 2$ being sufficiently large.\end{assumption}

\subsection{Full discretization}
In this subsection, we first present the considered semi-discretization and full discretization for \eqref{spde}, which based on the spectral Galerkin method in spatial direction and the exponential integrator in temporal direction. Then we give the existence and uniqueness of the invariant measures and the a priori estimates of numerical solutions. In the sequel, let $\tau\in (0,\frac{1}{2}),\,t_k:=k\tau,\,k\in \mathbb{N}$ and $\lfloor t\rfloor_{\tau}:=\max\{t_k:t_k\leq t\},$ i.e., the maximum grid point before time $t$.

Applying the spectral Galerkin method to \eqref{spde}, we get the semi-discretization
\begin{align}\label{semidiscrete}
\mathrm{d}X^N(t)=A^NX^N(t)\mathrm{d}t+F^N(X^N(t))\mathrm{d}t+P^N\mathrm{d}W(t),\quad X^N(0)=P^NX_0,
\end{align}
where $P^N: H\rightarrow H_N:=\mathrm{span}\{e_1,\ldots,e_N\}$ is the spectral Galerkin projection defined by $P^Nx=\sum_{i=1}^N\langle x,e_i\rangle e_i$ for $x\in H$,  $A^N=AP^N$ and $F^N=P^NF.$ 
We discretize \eqref{semidiscrete} in temporal direction by the exponential integrator to get the full discretization
\begin{align}\label{fulldiscrete}
X^N_{k+1}=S^N(\tau)X^N_k+\tau S^N(\tau)F^N(X^N_k)+S^N(\tau) P^N\Delta W_{k},\quad X^N_0=P^NX_0,
\end{align}
where $\{S^N(t):=e^{t A^N},t\ge 0\}$ is the semigroup generated by $A^N$, and $\Delta W_{k}=\sum_{j=1}^{\infty}Q^{\frac{1}{2}}e_j\Delta_{k}\beta_j$ with $\Delta_{k}\beta_j=\beta_j(t_{k+1})-\beta_j(t_k),k\in\mathbb{N}.$ The continuous version of the full discretization is
\begin{align}\label{contiversion}
X^{N,\tau}_t=S^N(t)X^N_0+\int_0^tS^N(t-\lfloor s\rfloor_{\tau})F^N(X^{N,\tau}_{\lfloor s\rfloor_{\tau}})\mathrm{d}s+\int_0^tS^N(t-\lfloor s\rfloor_{\tau})P^N\mathrm{d}W(s),
\end{align}
which solves
\begin{align}\label{differential}
\mathrm{d}X^{N,\tau}_t=A^NX^{N,\tau}_t\mathrm{d}t+S^N(t-\lfloor t\rfloor_{\tau})F^N(X^{N,\tau}_{\lfloor t\rfloor_{\tau}})\mathrm{d}t+S^N(t-\lfloor t\rfloor_{\tau})P^N\mathrm{d}W(t),\; X^{N,\tau}_0=X^N_0.
\end{align}

 The well-posedness as well as the existence and uniqueness of invariant measures of the exact solution of \eqref{spde} and the semi-discrete numerical solution of \eqref{semidiscrete} are well-known. For our convenience, we list them without proofs, and refer to \cite[Section $4$]{B14}, \cite[Section $4$]{BK17} for more details. In the sequel, we denote by $X(t,X_0)$ and $X^N(t,X^N_0)$ the solutions of \eqref{spde} and \eqref{semidiscrete} at time $t$ with initial data being $X_0$ and $X^N_0,$ respectively. Sometimes, we abbreviate them as $X(t)$ and $X^N(t)$ when 
the initial data are not emphasized. Denote $\bar{\pi}(\psi):=\int \psi(x)\bar{\pi}(\mathrm{d}x)$ the integral of $\psi$ with respect to the probability measure $\bar{\pi}.$
\begin{prop}\label{propexact}
Let Assumptions \ref{assumpF_1}, \ref{assump_Q} and \ref{assumpX_0} hold. Then \eqref{spde} has a unique mild solution given by
\begin{align*}
X(t)=S(t)X_0+\int_0^tS(t-s)F(X(s))\mathrm{d}s+\int_0^tS(t-s)\mathrm{d}W(s)\quad a.s.,\; t\ge 0,
\end{align*}
and it admits a unique
invariant measure $\pi.$ Moreover, 
for $p\ge2$ and $\psi\in\mathcal{C}_b,\; t\ge 0,$
\begin{align}
&\sup_{t\ge 0}\mathbb{E}\|X(t,X_0)\|^p\leq C(\|(-A)^{\frac{\beta-1}{2}}Q^{\frac{1}{2}}\|_{\mathcal{L}_2(H)},K,p)(1+\mathbb{E}\|X_0\|^p),\notag\\
&\mathbb{E}\|X(t,X_0)-X(t,\bar{X}_0)\|^p\leq C(p)\mathbb{E}\|X_0-\bar{X}_0\|^pe^{-\frac{1}{2}(\lambda_1-K)pt},\notag\\
&|\mathbb{E}\psi(X(t,X_0))-\pi(\psi)|\leq C\|\psi\|_{\infty}(1+\mathbb{E}\|X_0\|^2)e^{-\frac{1}{2}(\lambda_1-K)t},\label{exactmeas}
\end{align}
 and $\pi (\|\cdot\|^p)<\infty,$ where $X_0$ and $\bar{X}_0$ are two different initial data. 
\end{prop}
\begin{prop}\label{propsemi}
Let Assumptions \ref{assumpF_1}, \ref{assump_Q} and \ref{assumpX_0} hold. Then \eqref{semidiscrete} 
has a unique mild solution given by
\begin{align*}
X^N(t)=S^N(t)X^N_0+\int_0^tS^N(t-s)F^N(X^N(s))\mathrm{d}s+\int_0^tS^N(t-s)P^N\mathrm{d}W(s)\quad a.s.,\; t\ge 0.
\end{align*}
Moreover, for $p\ge 2,$
\begin{align*}
\sup_{t\ge0}\mathbb{E}\|X^N(t,X^N_0)\|^p\leq C(\|(-A)^{\frac{\beta-1}{2}}Q^{\frac{1}{2}}\|_{\mathcal{L}_2(H)},K,p)(1+\mathbb{E}\|X^N_0\|^p)
\end{align*}
and
\begin{align}\label{exp1}
\mathbb{E}\|X^N(t,X^N_0)-X^N(t,\bar{X}^N_0)\|^{p}\leq C(p)\mathbb{E}\|X^N_0-\bar{X}^N_0\|^{p}e^{-\frac{1}{2}(\lambda_1-K)pt},\;t\ge 0,
\end{align}
where $X^N_0$ and $\bar{X}^N_0$ are two different initial data. 
Further, \eqref{semidiscrete} admits a unique invariant measure $\pi^N$ and 
for $\psi\in\mathcal{C}_b,\; t\ge0,$
\begin{align}\label{semipi}
|\mathbb{E}\psi(X^N(t,X^N_0))-\pi^N(\psi)|\leq C\|\psi\|_{\infty}(1+\mathbb{E}\|X^N_0\|^2)e^{-\frac{1}{2}(\lambda_1-K)t},
\end{align}
and $\pi^N(\|\cdot\|^p)<\infty.$
\end{prop}

Below we give the existence and uniqueness of the invariant measure of \eqref{fulldiscrete}, and show the time-independent a priori estimates and the convergence order of the full discretization. To make the article self-contained, we list the proofs, which are postponed to the appendix.
\begin{prop}\label{propfull}
Let Assumptions \ref{assumpF_1}, \ref{assump_Q} and \ref{assumpX_0}   hold. 
 For $p\ge 1$ and $\tau\leq \frac{\lambda_1-K}{4L_F^2},$
 \begin{align*}
&\sup_{t\ge 0}\mathbb{E}\|X^{N,\tau}_{t}\|^{2p}_{\beta}\leq C(\|(-A)^{\frac{\beta-1}{2}}Q^{\frac{1}{2}}\|_{\mathcal{L}_2(H)},K,p)(1+\mathbb{E}\|X^N_0\|^{2p}_{\beta}),\\
&\mathbb{E}\|X^{N,\tau}_t-X^{N,\tau}_s\|^2\leq C(\|X^N_0\|_{L^2(\Omega;\dot{H}^{\beta})},\|(-A)^{\frac{\beta-1}{2}}Q^{\frac{1}{2}}\|_{\mathcal{L}_2(H)})(t-s)^{\beta},0\leq s\leq t ,\\
&\mathbb{E}\|X^N_k-\bar{X}^N_k\|^{2p}\leq \mathbb{E}\|X^N_0-\bar{X}^N_0\|^{2p}e^{-(\lambda_1-K)pt_k},\;k\in\mathbb{N},
\end{align*}
where $X^N_k$ and $\bar{X}^N_k$ are solutions of \eqref{fulldiscrete} with initial data $X^N_0$ and $\bar{X}^N_0$, respectively.
  Further, \eqref{fulldiscrete} admits a unique invariant measure $\pi^N_{\tau}$
 and for $\psi\in \mathcal{C}^1_b,\,k\in\mathbb{N},$
\begin{align}\label{full1}
|\mathbb{E}\psi(X^N_k)-\pi^N_{\tau}(\psi)|\leq C\|\psi\|_{1,\infty}(1+\mathbb{E}\|X^N_0\|^2)e^{-\frac{1}{2}(\lambda_1-K)t_k},
\end{align}
and $\pi^N_{\tau}(\|\cdot\|^p)<\infty.$
\end{prop}
\begin{remark}\label{rem1}
In Propositions \ref{propexact} \ref{propsemi} and \ref{propfull}, the conditions $\psi\in\mathcal{C}_b$ in \eqref{exactmeas} and \eqref{semipi} and $\psi\in \mathcal{C}^1_b$ in \eqref{full1} can both be changed to  $\psi\in \mathcal{C}^1_q:=\{\psi:H\rightarrow \mathbb{R}\big|\|D\psi(x)\|\leq C(1+\|x\|^q),x\in H\}$ with $q\ge 1.$ In this setting, similar arguments lead to 
$$|\mathbb{E}\psi(X(t,X_0))-\pi (\psi)|\leq C(1+\mathbb{E}\|X_0\|^{2q+2})^{\frac{1}{2}}e^{-\frac{1}{2}(\lambda_1-K)t},$$
$$|\mathbb{E}\psi(X^N(t,X^N_0))-\pi ^N(\psi)|\leq C(1+\mathbb{E}\|X^N_0\|^{2q+2})^{\frac{1}{2}}e^{-\frac{1}{2}(\lambda_1-K)t}$$ for $t\ge 0$
and
$$|\mathbb{E}\psi(X^N_k)-\pi ^N_{\tau}(\psi)|\leq C(1+\mathbb{E}\|X^N_0\|^{2q+2})^{\frac{1}{2}}e^{-\frac{1}{2}(\lambda_1-K)t_k}$$ for $k\in \mathbb{N}$ 
since $\sup_{t\ge 0}\mathbb{E}\|X(t,X_0)\|^{2q}<C(1+\mathbb{E}\|X_0\|^{2q}),\;
\sup_{t\ge 0}\mathbb{E}\|X^N(t,X^N_0)\|^		{2q}<C(1+\mathbb{E}\|X^N_0\|^{2q})$ and $\sup_{k\in\mathbb{N}}\mathbb{E}\|X^N_k\|^{2q}<C(1+\mathbb{E}\|X^N_0\|^{2q}).$ 
\end{remark}

Similarly to \cite[Theorem $3.4$]{wxj}, the following time-independent convergence order of the semi-discretization can be obtained based on the exponential decay in Proposition \ref{propsemi}. The proof is omitted.

\begin{prop}\label{propcon1}
Let Assumptions \ref{assumpF_1}, \ref{assump_Q} and \ref{assumpX_0} hold. Then
\begin{align*}
\sup_{t\ge 0}\mathbb{E}\|X(t)-X^N(t)\|^2\leq C(\|(-A)^{\frac{\beta-1}{2}}Q^{\frac{1}{2}}\|_{\mathcal{L}_2(H)},K)(1+\mathbb{E}\|X_0\|^2_{\beta})N^{-2\beta}.
\end{align*}
\end{prop}

Now we show the convergence order of the full discretization, whose proof is postponed to the appendix.
\begin{prop}\label{propcon}
Let Assumptions \ref{assumpF_1}, \ref{assump_Q} and \ref{assumpX_0} hold. Then \begin{align*}
\sup_{k\in\mathbb{N}}\mathbb{E}\|X^N_{k}-X^N(t_k)\|^2\leq C(\|(-A)^{\frac{\beta-1}{2}}Q^{\frac{1}{2}}\|_{\mathcal{L}_2(H)},K)(1+\mathbb{E}\|X_0\|_{\beta}^2)\tau^{\beta}.
\end{align*}
\end{prop}

\begin{remark}\label{remark2}
By Propositions \ref{propexact}, \ref{propsemi}, \ref{propfull}, \ref{propcon1} and \ref{propcon}, we have $$|\pi^N(\psi)-\pi(\psi)|\leq CN^{-\beta},\quad |\pi^N_{\tau}(\psi)-\pi^N(\psi)|\leq C\tau^{\frac{\beta}{2}},\quad\psi\in\mathcal{C}^1_q.$$ In fact, 
\begin{align*}
&~|\pi^N(\psi)-\pi(\psi)|=|\mathbb{E}\psi(X^N(t,T^N_0))-\mathbb{E}\psi(X(t,\tilde{T}_0))|\\
\leq&~ |\mathbb{E}\psi(X^N(t,T^N_0))-\mathbb{E}\psi(X^N(t,X^N_0))|
+|\mathbb{E}\psi(X^N(t,X^N_0))
-\mathbb{E}\psi(X(t,X^N_0))|\\
&+|\mathbb{E}\psi(X(t,X^N_0))-\mathbb{E}\psi(X(t,\tilde{T}_0))|\\
\leq &~C\mathbb{E}\big[(1+\|X^N(t,T^N_0)\|^q+\|X^N(t,X^N_0)\|^q)\|X^N(t,T^N_0)-X^N(t,X^N_0)\|\big]\\
&+C\mathbb{E}\big[(1+\|X^N(t,X^N_0)\|^q+\|X(t,X^N_0)\|^q)\|X^N(t,X^N_0)-X(t,X^N_0)\|\big]\\
&+C\mathbb{E}\big[(1+\|X(t,X^N_0)\|^q+\|X(t,\tilde{T}_0)\|^q)\|X(t,X^N_0)-X(t,\tilde{T}_0)\|\big]\\
\leq&~ Ce^{-Ct}+CN^{-\beta},
\end{align*}
where the law of $T^N_0$ and $\tilde{T}_0$ are $\pi^N$ and $\pi$ respectively. Letting $t\rightarrow \infty$ leads to the desired result. The estimate of $|\pi^N_{\tau}(\psi)-\pi^N(\psi)|$ can be proved similarly. We remark that it is shown in \cite{czh20} that by the approach of weak error analysis, for $\psi\in\mathcal{C}^2_b,$ the convergence order of $|\pi^N(\psi)-\pi(\psi)|$ and $|\pi^N_{\tau}(\psi)-\pi^N(\psi)|$ can be improved to $2\beta-$ and $\beta-$, respectively.
\end{remark}

\subsection{Poisson-type equations}\label{poissontype}
In this subsection, we introduce three Poisson-type equations, including the Poisson's equation and the modified Poisson's equation associated to \eqref{spde}, and the Poisson's equation associated to the semi-discretization \eqref{semidiscrete}. 

To proceed, we introduce the generators of \eqref{spde} and the semi-discretization \eqref{semidiscrete}.
 The generator $\mathcal{L}$ of \eqref{spde} is defined as
\begin{align}\label{generator}
\mathcal{L}\varphi(x)=\langle Ax+F(x),D\varphi(x)\rangle+\frac{1}{2}\mathrm{Tr}(QD^2\varphi(x)),\quad \varphi\in\mathcal{C}^2_b,
\end{align}
and the generator $\mathcal{L}^{N}$ of \eqref{semidiscrete} is defined as
\begin{align}
&\mathcal{L}^N\varphi(x)=\langle A^Nx+F^N(x),D\varphi(x)\rangle+\frac{1}{2}\mathrm{Tr}(P^NQD^2\varphi(x)),\;\varphi\in\mathcal{C}^2_b(H_N;\mathbb{R}).\label{L^N}
\end{align}
For a bounded Borel measurable function $h:H\rightarrow \mathbb{R}$, we can define three Poisson-type equations as follows, which are frequently used in the sequel.
\begin{itemize}
\item[(\romannumeral1)] Poisson's equation associated to \eqref{spde}:
\begin{align}\label{poisson2}
\mathcal{L}\phi(x)=h(x)-\pi(h),\quad x\in H.
\end{align} 
Let
\begin{align}\label{phiexp}
\phi(x):=-\int_0^{\infty}[\mathbb{E}h(X(t,x))-\pi(h)]\mathrm{d}t.
\end{align}
Similarly to \cite[Lemma $4.3$]{B12}, it can be verified that $\phi$ is a solution of the associated Poisson's equation.

\item[(\romannumeral2)] Modified Poisson's equation associated to \eqref{spde}:
\begin{align}\label{poisson}
 \mathcal{L}\phi_{\tau,\epsilon}(x)=\tau^{\frac{\beta}{2}-\epsilon}\phi_{\tau,\epsilon}(x)+h(x)-\pi(h), \quad x\in H,
 \end{align}
 where $0<\epsilon<\frac{\beta}{2}$ is a parameter to be determined. 
 We deduce from \cite[Chapter 5]{pde01} that the unique solution $\phi_{\tau,\epsilon}$ can be represented as 
\begin{align}\label{phiinte}
\phi_{\tau,\epsilon}(x)=-\int_{0}^{\infty}\exp\{-\tau^{\frac{\beta}{2}-\epsilon} t\}\big[\mathbb{E}h(X(t,x))-\pi(h)\big]\mathrm{d}t.
\end{align}

\item[(\romannumeral3)] Poisson's equation associated to \eqref{semidiscrete}: 
\begin{align}\label{GammaNtau}
\mathcal{L}^N\Gamma^N_{\tau,\epsilon}(x)=\phi_{\tau,\epsilon}(x)-\pi^N(\phi_{\tau,\epsilon}),\;x\in H_N,
\end{align}
 where $\phi_{\tau,\epsilon}$ is the solution of \eqref{poisson}.
 It can be verified that $$\Gamma^N_{\tau,\epsilon}(x):=-\int_0^{\infty}\big[\mathbb{E}\phi_{\tau,\epsilon}(X^N(t,x))-\pi^N(\phi_{\tau,\epsilon})\big]\mathrm{d}t$$ is a solution of \eqref{GammaNtau}.
\end{itemize}

We need the following assumption on the coefficient $F$, which plays a key role in gaining some regularity estimates of Poisson-type equations. 
\begin{assumption}\label{assumpF_3}
Assume that there exist $0\leq \sigma<1,\;C(\sigma)>0$ such that for $x,l,k,y,z\in H,$ 
\begin{align}
&\|(-A)^{-\sigma}D^2F(x)(l,k)\|\leq C(\sigma)\|l\|\|k\|,\label{D2F}\\
&\|(-A)^{-\sigma}D^3F(x)(l,k,y)\|\leq C(\sigma)\|l\|\|k\|\|y\|,\label{D3F}\\
&\|(-A)^{-\sigma}D^4F(x)(l,k,y,z)\|\leq C(\sigma)\|l\|\|k\|\|y\|\|z\|.\label{D4F}
\end{align}
\end{assumption}

Throughout this article, we assume that $h\in\mathcal{C}_b^4.$ Below, we show some properties, including regularity estimates and asymptotic property of Poisson-type equations. For the proofs, see Section \ref{sec3}.
\begin{lemma}\label{lemmaphi}
For any $\delta_i\in[0,1)$ with $\sum_{i=1}^4\delta_i<1$, there exists $C:=C(\delta_1,\delta_2,\delta_3,\delta_4)>0$ independent of $\tau,\epsilon$ such that the following estimates hold.\\
$(\mathrm{\romannumeral1})$ $($space-dependent regularity estimates of $\phi_{\tau,\epsilon}$$)$ Let Assumptions \ref{assumpF_1}, \ref{assump_Q} and \ref{assumpF_3} hold. Then for $x\in H,$
\begin{align}
&|\phi_{\tau,\epsilon}(x)|\leq C\|h\|_{\infty}(1+\|x\|^2),\quad \|(-A)^{\delta_1}D\phi_{\tau,\epsilon}(x)\|\leq C\|h\|_{1,\infty}(1+\|x\|^2),\label{phi1}\\
&\|(-A)^{\delta_1}D^2\phi_{\tau,\epsilon}(x)(-A)^{\delta_2}\|_{\mathcal{L}(H)}\leq C\|h\|_{2,\infty}(1+\|x\|^2),\label{phi2}\\
&\|(-A)^{\delta_1}D^3\phi_{\tau,\epsilon}(x)((-A)^{\delta_2}\cdot,(-A)^{\delta_3}\cdot)\|_{\mathcal{L}(H\times H;H)}\leq C\|h\|_{3,\infty}(1+\|x\|^2),\label{phi3}\\
&\|(-A)^{\delta_1}D^4\phi_{\tau,\epsilon}(x)((-A)^{\delta_2}\cdot,(-A)^{\delta_3}\cdot,(-A)^{\delta_4}\cdot)\|_{\mathcal{L}(H\times H\times H;H)}\leq C\|h\|_{4,\infty}(1+\|x\|^2).\;\label{phi4}
\end{align}
$(\mathrm{\romannumeral2})$ $($space-dependent regularity estimates of $\phi$$)$ Let Assumptions \ref{assumpF_1} and \ref{assump_Q} hold. Then for $x\in H,$
\begin{align}\label{phiexact}
&|\phi(x)|\leq C\|h\|_{\infty}(1+\|x\|^2),\quad\|(-A)^{\delta_1}D\phi(x)\|\leq C\|h\|_{1,\infty}(1+\|x\|^2).
\end{align}
If in addition Assumption \ref{assumpF_3} \eqref{D2F} holds, then for $x\in H,$
\begin{align*}
\|(-A)^{\delta_1}D^2\phi(x)(-A)^{\delta_2}\|_{\mathcal{L}(H)}\leq C\|h\|_{2,\infty}(1+\|x\|^2).
 \end{align*}
 $(\mathrm{\romannumeral3})$ $($space-dependent regularity estimates of $\Gamma^N_{\tau,\epsilon}$$)$ Let Assumptions \ref{assumpF_1} and \ref{assump_Q} hold. Then for $x\in H_N,$
 \begin{align*}
|\Gamma^N_{\tau,\epsilon}(x)|\leq C(1+\|x\|^3),\quad \|(-A)^{\delta_1}D\Gamma^N_{\tau,\epsilon}(x)\|\leq C(1+\|x\|^3).
\end{align*}
If in addition Assumption \ref{assumpF_3} \eqref{D2F} holds, then for $x\in H_N,$
\begin{align*}
\|(-A)^{\delta_1}D^2\Gamma^N_{\tau,\epsilon}(x)(-A)^{\delta_2}\|_{\mathcal{L}(H)}\leq C(1+\|x\|^3).
\end{align*}
\end{lemma}

\begin{lemma}\label{lemmaDDphi}
Let Assumptions \ref{assumpF_1}, \ref{assump_Q} and \ref{assumpF_3} hold. For any $\delta_i\in[0,1)$ with $\sum_{i=1}^3\delta_i<1$, there exists $C:=C(\delta_1,\delta_2,\delta_3)>0$ independent of $\tau,\epsilon$ such that for $x\in H,$
\begin{itemize}
\item
[(\romannumeral1)] $($space-independent regularity estimates of $\phi_{\tau,\epsilon}$$)$
\begin{align}
&\|(-A)^{\delta_1}D\phi_{\tau,\epsilon}(x)\|\leq C\tau^{-\frac{\beta}{2}+\epsilon},\label{phi5}\\
&\|(-A)^{\delta_1}D^2\phi_{\tau,\epsilon}(x)(-A)^{\delta_2}\|_{\mathcal{L}(H)}\leq C\tau^{-\frac{\beta}{2}+\epsilon},\label{phi6}\\
&\|(-A)^{\delta_1}D^3\phi_{\tau,\epsilon}(x)((-A)^{\delta_2}\cdot,(-A)^{\delta_3}\cdot)\|_{\mathcal{L}(H\times H;H)}\leq C\tau^{-\frac{\beta}{2}+\epsilon};\label{phi7}
\end{align}
\item 
[(\romannumeral2)] $($asymptotic relationship between $\phi$ and $\phi_{\tau,\epsilon}$$)$\\
$\int_H\|(-A)^{\delta_1}(D\phi_{\tau,\epsilon}(x)-D\phi(x))\|^2\pi(\mathrm{d}x)\rightarrow 0$ as $\tau\rightarrow0$ for $\epsilon\in(0,\frac{\beta}{2})$. 
\end{itemize}
\end{lemma}
\subsection{Central limit theorem of the exact solution}It is proved in \cite{clt12} that the CLT and the weak LLNs for the exact solution of \eqref{spde} hold, i.e.,
the time average $\frac{1}{T}\int_0^Th(X(t))\mathrm{d}t$ converges to the ergodic limit $\pi(h)$ in probability and the normalized fluctuation around $\pi(h)$ can be described by a centered Gaussian random variable. Since \cite{clt12} considers the case of more general stochastic processes, here we give a direct proof for the case of \eqref{spde} for readers' convenience.

\begin{prop}\label{CLTexact}
Let Assumptions \ref{assumpF_1}, \ref{assump_Q} and \ref{assumpX_0} hold. Then for $h\in\mathcal{C}^1_b,$ we have
\begin{align}\label{disexact}
\frac{1}{\sqrt{T}}\int_0^T[h(X(t))-\pi(h)]\mathrm{d}t\overset{d}{\longrightarrow} \mathcal{N}\big(0,\pi(\|Q^{\frac{1}{2}}D\phi(\cdot)\|^2)\big)\text{ as }T\rightarrow \infty,
\end{align}
 where $\mathcal{N}\big(0,\pi(\|Q^{\frac{1}{2}}D\phi(\cdot)\|^2)\big)$ is the normal distribution with $\phi$ being the solution \eqref{phiexp} of \eqref{poisson2}.
\end{prop}
\begin{proof}
We fist show that $\pi(\|Q^{\frac{1}{2}}D\phi(\cdot)\|^2)$ is well-defined.
By Lemma \ref{lemmaphi} (\romannumeral2), 
we have 
\begin{align}\label{estimatef_1}
\|Q^{\frac{1}{2}}D\phi(x)\|^2&\leq \|Q\|_{\mathcal{L}(H)}\|D\phi(x)\|^2\leq C\|h\|^2_{1,\infty}(1+\|x\|^4).
\end{align}
As a consequence, $\pi(\|Q^{\frac{1}{2}}D\phi(\cdot)\|^2)<\infty.$

By Poisson's equation \eqref{poisson2} and It$\mathrm{\hat{o}}$'s formula, we get
\begin{align*}
&\quad\frac{1}{\sqrt{T}}\int_0^T[h(X(t))-\pi(h)]\mathrm{d}t=\frac{1}{\sqrt{T}}\int_0^T\mathcal{L}\phi(X(t))\mathrm{d}t\\
&=\frac{1}{\sqrt{T}}\int_0^T\mathrm{d}\phi(X(t))-\frac{1}{\sqrt{T}}\int_0^T\langle D\phi(X(t)),\mathrm{d}W(t)\rangle\\
&=\frac{1}{\sqrt{T}}[\phi(X(T))-\phi(X(0))]-\frac{1}{\sqrt{T}}\int_0^T\langle D\phi(X(t)),\mathrm{d}W(t)\rangle.
\end{align*}
Applying Ch\'ebyshev's inequality and Lemma \ref{lemmaphi} (\romannumeral2) leads to that for each fixed $a>0,$
\begin{align*}
&\quad\mathbb{P}\Big(\frac{1}{\sqrt{T}}[\phi(X(T))-\phi(X(0))]>a\Big)\leq a^{-2}\mathbb{E}\Big|\frac{1}{\sqrt{T}}[\phi(X(T))-\phi(X(0))]\Big|^2\\
&\leq \frac{C}{a^2T}\sup_{t\ge0}\mathbb{E}|\phi(X(t))|^2\leq \frac{C}{a^2T}(1+\sup_{t\ge 0}\mathbb{E}\|X(t)\|^4)\to 0\text{ as }T\to\infty.
\end{align*}
By the It\^o isometry, we have
\begin{align*}
&\quad\mathbb{E}\Big|\frac{1}{\sqrt{T}}\int_0^T\langle D\phi(X(t)),\mathrm{d}W(t)\rangle\Big|^2=\frac{1}{T}\int_0^T\mathbb{E}\|Q^{\frac{1}{2}}D\phi(X(t))\|^2\mathrm{d}t\\
&= \frac{1}{T}\int_0^T\big[\mathbb{E}\|Q^{\frac{1}{2}}D\phi(X(t))\|^2-\pi(\|Q^{\frac{1}{2}}D\phi(\cdot)\|^2)\big]\mathrm{d}t+\pi(\|Q^{\frac{1}{2}}D\phi(\cdot)\|^2).
\end{align*}
 Remark \ref{rem1} and Lemma \ref{lemmaphi} (\romannumeral2) imply that
\begin{align*}
&\quad\Big|\frac{1}{T}\int_0^T\big[\mathbb{E}\|Q^{\frac{1}{2}}D\phi(X(t))\|^2-\pi(\|Q^{\frac{1}{2}}D\phi(\cdot)\|^2)\big]\mathrm{d}t\Big|
&\leq \frac{C}{T}\int_0^Te^{-\frac{1}{2}(\lambda_1-K)t}\mathrm{d}t\to0\text{ as }T\to\infty,
\end{align*}
which gives $\mathbb{E}\Big|\frac{1}{\sqrt{T}}\int_0^T\langle D\phi(X(t)),\mathrm{d}W(t)\rangle\Big|^2\to\pi(\|Q^{\frac{1}{2}}D\phi(\cdot)\|^2)\text{ as }T\to\infty.$
Hence the stochastic integral $\frac{1}{\sqrt{T}}\int_0^T\langle D\phi(X(t)),\mathrm{d}W(t)\rangle$ converges to $\mathcal{N}\big(0,\pi(\|Q^{\frac{1}{2}}D\phi(\cdot)\|^2)\big)$ in distribution as $T\to\infty.$ Applying Slutsky's theorem (see e.g., \cite[Theorem $9.1.6$]{measure}) yields \eqref{disexact}. The proof is finished.
\end{proof}

\begin{coro}
Under conditions in Proposition \ref{CLTexact}, the weak LLNs holds, i.e. $$\frac{1}{T}\int_0^Th(X(t))\mathrm{d}t\overset{\mathbb{P}}\longrightarrow\pi(h)\text{ as }T\to\infty.$$
\end{coro}
\begin{proof}
Similarly to the proof of Proposition \ref{CLTexact}, we have
\begin{align*}
\frac{1}{T}\int_0^T[h(X(t))-\pi(h)]\mathrm{d}t=\frac{1}{T}[\phi(X(T))-\phi(X(0))]-\frac{1}{T}\int_0^T\langle D\phi(X(t)),\mathrm{d}W(t)\rangle.
\end{align*}
By Lemma \ref{lemmaphi} (\romannumeral2), the It\^o isometry and Remark \ref{rem1}, it can be proved that both the second moments of $\frac{1}{T}[\phi(X(T))-\phi(X(0))]$ and $\frac{1}{T}\int_0^T\langle D\phi(X(t)),\mathrm{d}W(t)\rangle$ converge to zero as $T\to\infty.$ Hence applying Ch\'ebyshev's inequality leads to the desired weak LLNs.
\end{proof}
\section{Main result and idea of its proof}\label{mainresult}
In this section, we give the main result of this article and show the strategy for proving the main theorem, based on the decomposition of the normalized time-averaging estimator
$\tau^{-\frac{\beta}{2}}[\Pi^N_{\tau}(h)-\pi(h)]$.

\subsection{Central limit theorem of full discretization}
For simplicity, assume that $\tau^{-1-\beta}$ is an integer and set $m:=\tau^{-1-\beta}.$
Define the empirical measure $\Pi^N_{\tau}:=m^{-1}\sum_{k=0}^{m-1}\delta_{X^N_k}$ with $\delta_{x}$ being the Dirac measure at $x\in H_N$. Then for a Borel measurable function $h,$ we have $\Pi^N_{\tau}(h)=m^{-1}\sum_{k=0}^{m-1}h(X^N_k)$, which is called the time-averaging estimator. Similarly to \cite{BK17}, the time-averaging estimator can be shown to converge to $\pi(h)$ in the mean sense with a certain order. Here, we are interested in studying the fluctuation of the term $[\Pi^N_{\tau}(h)-\pi(h)]$. 

With these settings in mind, we can state the main result of our study. 
\begin{thm}\label{main}
Let Assumptions \ref{assumpF_1}, \ref{assump_Q}, \ref{assumpX_0} and \ref{assumpF_3} hold. Assume that $N=\lfloor \tau^{-\alpha}\rfloor$ $($i.e., the integer part of $\tau^{-\alpha}$$)$ with $\alpha\in(\frac{1}{4},\frac{1}{2})$. Then for $h\in \mathcal{C}^4_b$, we have
\begin{align}\label{mainre}
\tau^{-\frac{\beta}{2}}\big[\Pi^N_{\tau}(h)-\pi(h)\big]\overset{d}{\longrightarrow} \mathcal{N}\big(0,\pi(\|Q^{\frac{1}{2}}D\phi(\cdot)\|^2)\big)\text{ as }\tau\rightarrow 0.
\end{align}
\end{thm}

We remark that the function $f:=\|Q^{\frac{1}{2}}D\phi(\cdot)\|^2$ in Theorem \ref{main} belongs to $\mathcal{C}^1_4$, i.e., $\|Df(x)\|\leq C(1+\|x\|^4),\,x\in H$.
By Lemma \ref{lemmaphi} (\romannumeral2), 
we have 
\begin{align}
\|Df(x)\|&=\sup_{\|l\|=1}|Df(x)l|= \sup_{\|l\|=1}2|\langle Q^{\frac{1}{2}}D\phi(x),Q^{\frac{1}{2}}D^2\phi(x)l\rangle|\notag\\
&\leq C\|D\phi(x)\|\|D^2\phi(x)\|_{\mathcal{L}(H)}\leq C\|h\|^2_{2,\infty}(1+\|x\|^4).\label{estimatef_2}
\end{align}

  \subsection{Idea of the proof}\label{proofmain}

We first decompose the term $\tau^{-\frac{\beta}{2}}[\Pi^N_{\tau}(h)-\pi(h)]$ by using the modified Poisson's equation \eqref{poisson}. 

 Define $\Psi^N_{\tau}(h):=m^{-1}\sum_{k=0}^{m-1}\phi_{\tau,\epsilon}(X^N_k)$ with $\phi_{\tau,\epsilon}$ being the solution \eqref{phiinte} of the modified Poisson's equation \eqref{poisson}.
  It is clear that
\begin{align*}
&\Pi^N_{\tau}(h)-\pi(h)=m^{-1}\sum_{k=0}^{m-1}(h(X^N_k)-\pi(h))=m^{-1}\sum_{k=0}^{m-1}\Big(\mathcal{L}\phi_{\tau,\epsilon}(X^N_k)-\tau^{\frac{\beta}{2}-\epsilon}\phi_{\tau,\epsilon}(X^N_k)\Big)\\
=&~\tau^{\beta}\sum_{k=0}^{m-1}\Big[\mathcal{L}\phi_{\tau,\epsilon}(X^N_k)\tau-(\phi_{\tau,\epsilon}(X^N_{k+1})-\phi_{\tau,\epsilon}(X^N_k))\Big]
+\tau^{\beta}\sum_{k=0}^{m-1}\big[\phi_{\tau,\epsilon}(X^N_{k+1})-\phi_{\tau,\epsilon}(X^N_k)\big]-\tau^{\frac{\beta}{2}-\epsilon}\Psi^N_{\tau}(h).
\end{align*}
The Taylor expansion implies that
\begin{align}\label{taylor}
&\phi_{\tau,\epsilon}(X^N_{k+1})-\phi_{\tau,\epsilon}(X^N_k)\notag\\
=&~\big\langle X^N_{k+1}-X^N_k,D\phi_{\tau,\epsilon}(X^N_k)\big\rangle+\frac{1}{2}\big\langle X^N_{k+1}-X^N_k,D^2\phi_{\tau,\epsilon}(X^N_k)(X^N_{k+1}-X^N_k)\big\rangle+\mathcal{R}^N_{k,\phi_{\tau,\epsilon}}
\end{align}
with 
\begin{align*}
\mathcal{R}^N_{k,\phi_{\tau,\epsilon}}=\frac{1}{2}\int_0^1(1-t)^2D^3\phi_{\tau,\epsilon}\big(X^N_k+t(X^N_{k+1}-X^N_k)\big)\big(X^N_{k+1}-X^N_k,X^N_{k+1}-X^N_k,X^N_{k+1}-X^N_k\big)\mathrm{d}t.
\end{align*}
 By the differential form \eqref{differential} of the full discretization,
we get
\begin{align*}
X^N_{k+1}-X^N_k=\int_{t_k}^{t_{k+1}}A^NX^{N,\tau}_t\mathrm{d}t+\int_{t_k}^{t_{k+1}} S^N(t-t_k)F^N(X^N_k)\mathrm{d}t+\int_{t_k}^{t_{k+1}} S^N(t-t_k)P^N\mathrm{d}W(t).
\end{align*}
Thus, \eqref{generator} and \eqref{taylor} lead to
\begin{align*}
&\mathcal{L}\phi_{\tau,\epsilon}(X^N_k)\tau-(\phi_{\tau,\epsilon}(X^N_{k+1})-\phi_{\tau,\epsilon}(X^N_k))\\
=&~\Big\langle \int_{t_k}^{t_{k+1}}A^N (X^N_k-X^{N,\tau}_t)\mathrm{d}t,D\phi_{\tau,\epsilon}(X^N_k)\Big\rangle+
\Big\langle \int_{t_k}^{t_{k+1}} (\mathrm{Id}-S^N(t-t_k))F(X^N_k)\mathrm{d}t,D\phi_{\tau,\epsilon}(X^N_k)\Big\rangle\\
& -\Big\langle \int_{t_k}^{t_{k+1}} S^N(t-t_k)P^N \mathrm{d}W(t),D\phi_{\tau,\epsilon}(X^N_k)\Big\rangle
+\frac{1}{2}\mathrm{Tr}(QD^2\phi_{\tau,\epsilon}(X^N_k))\tau\\
&-\frac{1}{2}\big\langle X^N_{k+1}-X^N_k,D^2\phi_{\tau,\epsilon}(X^N_k)(X^N_{k+1}-X^N_k)\big\rangle-\mathcal{R}^N_{k,\phi_{\tau,\epsilon}}.
\end{align*}
Hence, we get
\begin{align}\label{decompose}
\tau^{-\frac{\beta}{2}}\big[\Pi^N_{\tau}(h)-\pi(h)\big]
&=\mathcal{M}^N_{\tau}+\mathcal{R}^N_{\tau}+\tilde{\mathcal{R}}^N_{\tau},
\end{align}
where
\begin{align*}
\mathcal{M}^N_{\tau}:=-\tau^{\frac{\beta}{2}}\sum_{k=0}^{m-1}\Big\langle \int_{t_k}^{t_{k+1}} S^N(t-t_k)P^N\mathrm{d}W(t),D\phi_{\tau,\epsilon}(X^N_k)\Big\rangle,\quad 
\end{align*}
is the martingale difference series sum, and $\mathcal{R}^N_{\tau}+\tilde{\mathcal{R}}^N_{\tau}$ is the remainder with
\begin{align*}
\tilde{\mathcal{R}}^N_{\tau}:=-\tau^{-\epsilon}\Psi^N_{\tau}(h)
\end{align*}
and
\begin{align*}
\mathcal{R}^N_{\tau}
&:=\tau^{\frac{\beta}{2}}\big[\phi_{\tau,\epsilon}(X^N_m)-\phi_{\tau,\epsilon}(X^N_0)\big]+\tau^{\frac{\beta}{2}}\sum_{k=0}^{m-1}\Big\langle\int_{t_k}^{t_{k+1}} (\mathrm{Id}-S^N(t-t_k))F(X^N_k)\mathrm{d}t,D\phi_{\tau,\epsilon}(X^N_k)\Big\rangle\\
&\quad+\frac{1}{2}\tau^{\frac{\beta}{2}}\sum_{k=0}^{m-1}\Big[\mathrm{Tr}(QD^2\phi_{\tau,\epsilon}(X^N_k))\tau
-\big\langle X^N_{k+1}-X^N_k,D^2\phi_{\tau,\epsilon}(X^N_k)(X^N_{k+1}-X^N_k)\big\rangle\Big]\\
&\quad-\sum_{k=0}^{m-1}\tau^{\frac{\beta}{2}}\mathcal{R}^N_{k,\phi_{\tau,\epsilon}}+\tau^{\frac{\beta}{2}}\sum_{k=0}^{m-1}\Big\langle \int_{t_k}^{t_{k+1}}A^N (X^N_k-X^{N,\tau}_t)\mathrm{d}t,D\phi_{\tau,\epsilon}(X^N_k)\Big\rangle
=:\sum_{i=1}^{5}\mathcal{R}^N_{\tau, i}.
\end{align*}

Based on the decomposition \eqref{decompose}, our strategy to prove  Theorem \ref{main} follows from the following three steps. We first prove that $\mathcal{M}^N_{\tau}\overset{d}{\longrightarrow}  \mathcal{N}(0,\pi(f))\text{ as }\tau\rightarrow 0$. Then we show that $\mathcal{R}^N_{\tau}\overset{\mathbb{P}}{\longrightarrow}  0$ as $\tau\rightarrow 0$. The last step is to show that $\tilde{\mathcal{R}}^N_{\tau}\overset{\mathbb{P}}{\longrightarrow}  0$ as $\tau\rightarrow 0$. The above three steps,
 together with Slutsky's Theorem (see e.g., \cite[Theorem $9.1.6$]{measure}) imply 
$\tau^{-\frac{\beta}{2}}\big[\Pi^N_{\tau}(h)-\pi(h)\big]\overset{d}{\longrightarrow} \mathcal{N}(0,\pi(f))$ as $\tau\rightarrow 0$. 
Precisely, 
 the proof of the Theorem \ref{main} relies on the following three propositions, whose proofs are postponed to the next section.
\begin{prop}\label{prop2}
Under conditions in Theorem \ref{main},
 for any $\epsilon\in (0,\frac{\beta}{2}),$
\begin{align}
\mathcal{M}^N_{\tau}\overset{d}{\longrightarrow}  \mathcal{N}(0,\pi(f))\text{ as }\tau\rightarrow 0.
\end{align}
\end{prop}

\begin{prop}\label{prop3}
Under conditions in Theorem \ref{main},
 for any $\epsilon\in (\max\{0,\frac{\alpha-1+\beta}{2}\},\frac{\beta}{2}),$
$$\mathcal{R}^N_{\tau}\overset{\mathbb{P}}{\longrightarrow} 0\text{ as }\tau\rightarrow 0.$$
\end{prop}

\begin{prop}\label{prop4}
Under conditions in Theorem \ref{main}, for any $\epsilon\in (0,\alpha\beta)$,
\begin{align*}
\tilde{\mathcal{R}}^N_{\tau}\overset{\mathbb{P}}{\longrightarrow} 0\text{ as }\tau\rightarrow 0.
\end{align*}
\end{prop}

With the above three propositions, we can give the proof of Theorem \ref{main}.
\begin{proof}[Proof of Theorem \ref{main}]
The proof is a combination of Propositions  \ref{prop2}, \ref{prop3} and \ref{prop4}. 
Since $\frac{\alpha-1+\beta}{2}<\frac{\frac{1}{2}-1+\beta}{2}\leq \frac{1}{4}\beta<\alpha\beta<\frac{\beta}{2}$, we can take $\epsilon\in(\max\{0,\frac{\alpha-1+\beta}{2}\},\alpha\beta)$
such that Propositions  \ref{prop2}, \ref{prop3} and \ref{prop4} hold simultaneously. Hence, applying Slutsky's theorem finishes
  the proof.\end{proof}

At the end of this subsection, we remark that based on Propositions \ref{prop3} and \ref{prop4} and the application of Ch\'ebyshev's inequality, the following weak LLNs can be obtained.
\begin{coro}\label{coroprob}
Under conditions in Theorem \ref{main},
 the weak LLNs holds, i.e., \begin{align*}
 \big[\Pi^N_{\tau}(h)-\pi(h)\big]\overset{\mathbb{P}}{\longrightarrow}0\text{ as }\tau\rightarrow0.
 \end{align*}
\end{coro}
\begin{proof}
The proof is based on the decomposition \eqref{decompose}.
 Since \begin{align*}
 &\mathbb{E}|\tau^{\frac{\beta}{2}}\mathcal{M}^N_{\tau}|^2\leq \tau^{2\beta}\mathbb{E}\Big| \int_0^{t_m}\big\langle (-A)^{\frac{1-\beta}{2}}D\phi_{\tau,\epsilon}(X^{N,\tau}_{\lfloor t\rfloor_{\tau}}),S^N(t-\lfloor t\rfloor_{\tau})P^N(-A)^{\frac{\beta-1}{2}}\mathrm{d}W(t)\big\rangle\Big|^2\\
\leq &~C\tau^{2\beta}\int_0^{t_m}\|(-A)^{\frac{\beta-1}{2}}Q^{\frac{1}{2}}\|^2_{\mathcal{L}_2(H)}\sup_{k\ge 0}\mathbb{E}\|(-A)^{\frac{1-\beta}{2}}D\phi_{\tau,\epsilon}(X^N_k)\|^2\mathrm{d}t\leq C\tau^{\beta}\rightarrow0\text{ as }\tau\rightarrow0,
\end{align*}
we deduce from Ch\'ebyshev's inequality that
$\mathbb{P}(|\tau^{\frac{\beta}{2}}\mathcal{M}^N_{\tau}|>a)\rightarrow0 \text{ as }\tau\rightarrow0.$
Moreover, by Proposition \ref{prop3}, we have
\begin{align*}
\mathbb{P}(|\tau^{\frac{\beta}{2}}\mathcal{R}^N_{\tau}|>a)=\mathbb{P}(|\mathcal{R}^N_{\tau}|>\tau^{-\frac{\beta}{2}}a)\rightarrow0 \text{ as }\tau\rightarrow0.
\end{align*}
Similarly, applying Proposition \ref{prop4} leads to $\mathbb{P}(|\tau^{\frac{\beta}{2}}\tilde{\mathcal{R}}^N_{\tau}|>a)\rightarrow0\text{ as }\tau\rightarrow0.$ The proof is finished by the additive property of the convergence in probability.
\end{proof}

  \section{Proofs of properties of Poisson-type equations}
\label{sec3}

In this section, 
we prove Lemmas \ref{lemmaphi} and \ref{lemmaDDphi} about properties of Poisson-type equations introduced in subsection \ref{poissontype}.
As a prerequisite of proving Lemma \ref{lemmaphi}, we first give estimates of $u(t,x):=\mathbb{E}h(X(t,x))$.

\begin{lemma}
For $l,k,y,s\in H,$ we have for $\delta_i\in[0,1)$ with $\sum_{i=1}^4\delta_i<1$ and $t>0,$
\begin{align}
&|Du(t,x)l|\leq C\|h\|_{1,\infty}e^{-Ct}\big(1+t^{-\delta_1}\big)(1+\|x\|^2)\|l\|_{-2\delta_1},\label{Dutge0}\\
&|D^2u(t,x)(l,k)|\leq C\|h\|_{2,\infty}e^{-Ct}\big(1+t^{-\sum_{i=1}^2\delta_i}\big)(1+\|x\|^2)\|l\|_{-2\delta_1}\|k\|_{-2\delta_2},\label{D2utge0}\\
&|D^3u(t,x)(l,k,y)|\leq C\|h\|_{3,\infty}e^{-Ct}\big(1+t^{-\sum_{i=1}^3\delta_i}\big)(1+\|x\|^2)\|l\|_{-2\delta_2}\|k\|_{-2\delta_3}\|y\|_{-2\delta_1},\label{ttge1}\\
&|D^4u(t,x)(l,k,y,s)|\leq C\|h\|_{4,\infty}e^{-Ct}\big(1+t^{-\sum_{i=1}^4\delta_i}\big)(1+\|x\|^2)\|l\|_{-2\delta_2}\|k\|_{-2\delta_3}\|y\|_{-2\delta_1}\|s\|_{-2\delta_4}.\label{D4utge0}
\end{align}
\end{lemma}
\begin{proof}
For the proof of \eqref{Dutge0} and \eqref{D2utge0}, we refer to \cite[Propositions $5.1, 5.2$]{B14} for the details. We only show the proof \eqref{ttge1} since \eqref{D4utge0} can be proved similarly. 
By the chain's rule and \cite[(21)]{B14}, we get for $l,k,y\in H,$
\begin{align}
D^3u(t,x)(l,k,y)&=\mathbb{E}\Big[D^3h(X(t,x))(\eta^{l,x}(t),\eta^{k,x}(t),\eta^{y,x}(t))+D^2h(X(t,x))(\zeta^{l,y,x}(t),\eta^{k,x}(t))\nonumber\\
&\quad\quad+D^2h(X(t,x))(\eta^{l,x}(t),\zeta^{k,y,x}(t))+D^2h(X(t,x))(\zeta^{l,k,x}(t),\eta^{y,x}(t))\nonumber\\
&\quad\quad+Dh(X(t,x))\xi^{l,k,y,x}(t)\Big],\;t>0,\;x\in H,\label{u1}
\end{align}
where $\eta^{l,x}$ is the solution of
\begin{align*}
\frac{\mathrm{d}\eta^{l,x}(t)}{\mathrm{d}t}=A\eta^{l,x}(t)+DF(X(t,x))\eta^{l,x}(t),\quad \eta^{l,x}(0)=l,
\end{align*}
 $\zeta^{l,k,x}$ is the solution of
\begin{align*}
\frac{\mathrm{d}\zeta^{l,k,x}(t)}{\mathrm{d}t}=A\zeta^{l,k,x}(t)+DF(X(t,x))\zeta^{l,k,x}(t)+D^2F(X(t,x))(\eta^{l,x}(t),\eta^{k,x}(t)),\quad \zeta^{l,k,x}(0)=0,
\end{align*}
and
$\xi^{l,k,y,x}$ is the solution of
\begin{align*}
\frac{\mathrm{d}\xi^{l,k,y,x}(t)}{\mathrm{d}t}&=A\xi^{l,k,y,x}(t)+DF(X(t,x))\xi^{l,k,y,x}(t)+D^2F(X(t,x))(\zeta^{l,k,x}(t),\eta^{y,x}(t))\\
&\quad+D^2F(X(t,x))(\zeta^{l,y,x}(t),\eta^{k,x}(t))+D^2F(X(t,x))(\eta^{l,x}(t),\zeta^{k,y,x}(t))\\
&\quad+D^3F(X(t,x))(\eta^{l,x}(t),\eta^{k,x}(t),\eta^{y,x}(t)),\quad
 \xi^{l,k,y,x}(0)=0.
\end{align*}
In the following, we split the proof into three steps.

\textit{Step $1$:}
We are going to show that for $0<t\leq 1$,
\begin{align}
&\|\xi^{l,k,y,x}(t)\|\leq C(\delta_1,\delta_2,\delta_3)t^{-\sigma}\|l\|_{-2\delta_2}\|k\|_{-2\delta_3}\|y\|_{-2\delta_1},\label{xi1}\\
&\|\xi^{l,k,y,x}(t)\|\leq C(\delta_1,\delta_2,\delta_3)t^{-(\delta_1+\delta_2+\delta_3)}\|l\|_{-2\delta_2}\|k\|_{-2\delta_3}\|y\|_{-2\delta_1},\label{xi2}
\end{align}
where $\sigma$ is defined in Assumption \ref{assumpF_3}.

The similar argument as the proof of \cite[Eq. (24)]{B14} can lead to that for $0<t\leq 1,$
\begin{equation}
\begin{split}\label{u2}
&\|\eta^{l,x}(t)\|\leq Ct^{-\delta_1}\|l\|_{-2\delta_1},\;\delta_1<1,\\
&\|\zeta^{l,k,x}(t)\|\leq Ct^{-(\delta_1+\delta_2)}\|l\|_{-2\delta_1}\|k\|_{-2\delta_2},\;\delta_1+\delta_2<1.
\end{split}
\end{equation}
In fact, by \eqref{semigroup1}, \eqref{D2F} and Assumption \ref{assumpF_1},
\begin{align*}
&\|\zeta^{l,k,x}(t)\|\leq L_F\int_0^t\|\zeta^{l,k,x}(r)\|\mathrm{d}r+\int_0^tC(t-r)^{-\sigma}\|\eta^{l,x}(r)\|\|\eta^{k,x}(r)\|\mathrm{d}r\\
\leq &~L_F\int_0^t\|\zeta^{l,k,x}(r)\|\mathrm{d}r+\int_0^1Ct^{1-\sigma-(\delta_1+\delta_2)}(1-\theta)^{-\sigma}\theta^{-(\delta_1+\delta_2)}\mathrm{d}\theta\|l\|_{-2\delta_1}\|k\|_{-2\delta_2}.
\end{align*}
Since $t^{1-\sigma-(\delta_1+\delta_2)}\leq t^{-(\delta_1+\delta_2)}$ for $0<t\leq 1,$ by applying Gr{\"o}nwall's inequality, we have $\|\zeta^{l,k,x}(t)\|\leq Ct^{-(\delta_1+\delta_2)}\|l\|_{-2\delta_1}\|k\|_{-2\delta_2}.$
Eq. \eqref{u2}, together with Assumption \ref{assumpF_1}, \eqref{D2F} and \eqref{D3F} yields that for $0<t\leq 1,$
\begin{align*}
&\|\xi^{l,k,y,x}(t)\|\\
\leq&~ L_F\int_0^t\|\xi^{l,k,y,x}(r)\|\mathrm{d}r+C\int_0^t(t-r)^{-\sigma}\big(\|\eta^{y,x}(r)\|\|\zeta^{l,k,x}(r)\|+\|\zeta^{l,y,x}(r)\|\|\eta^{k,x}(r)\|\\
&+\|\eta^{l,x}(r)\|\|\zeta^{k,y,x}(r)\|\big)\mathrm{d}r+C\int_0^t(t-r)^{-\sigma}\|\eta^{l,x}(r)\|\|\eta^{k,x}(r)\|\|\eta^{y,x}(r)\|\mathrm{d}r\\
\leq&~ L_F\int_0^t\|\xi^{l,k,y,x}(r)\|\mathrm{d}r+C\int_0^t(t-r)^{-\sigma}r^{-(\delta_1+\delta_2+\delta_3)}\|l\|_{-2\delta_2}\|k\|_{-2\delta_3}\|y\|_{-2\delta_1}\mathrm{d}r\\
\leq&~L_F\int_0^t\|\xi^{l,k,y,x}(r)\|\mathrm{d}r+C\int_0^1t^{1-\sigma-(\delta_1+\delta_2+\delta_3)}(1-\theta)^{-\sigma}\theta^{-(\delta_1+\delta_2+\delta_3)}\mathrm{d}\theta\|l\|_{-2\delta_2}\|k\|_{-2\delta_3}\|y\|_{-2\delta_1}.
\end{align*}
 Note that $t^{1-\sigma-(\delta_1+\delta_2+\delta_3)}\leq t^{-\sigma}$ and $t^{1-\sigma-(\delta_1+\delta_2+\delta_3)}\leq t^{-(\delta_1+\delta_2+\delta_3)}$ for $0<t\leq 1.$ Therefore, applying Gr{\"o}nwall's inequality leads to \eqref{xi1} and \eqref{xi2}, respectively.

\textit{Step $2$:} We are going to verify 
\begin{align}\label{step2}
|D^3u(t,x)(l,k,y)|\leq C\|h\|_{\infty}e^{-C(t-1)}(1+\|x\|^2)\|l\|\|k\|\|y\|,\;t\ge 1.
\end{align}
The Bismut--Elworthy–Li formula (see e.g. \cite[Proposition $4.4.3$]{pde01}) states that for any $\Phi\in\mathcal{C}^3(H;\mathbb{R})$ with $|\Phi(x)|\leq M(\Phi)(1+\|x\|^2),M(\Phi)>0,$ we can calculate the derivatives of $v(s,x):=\mathbb{E}\Phi(X(s,x))$ with $s>0$ up to order $3$ with respect to $x$.
The first derivative can be represented as
$$
D v(s,x) l=\frac{1}{s} \mathbb{E}\int_{0}^{s}\langle Q^{-\frac{1}{2}} \eta^{l,x}(r), \mathrm{d}  \tilde{W}(r)\rangle \Phi(X(s, x)),
$$
where $\tilde{W}=\sum_{j=1}^{\infty}e_j\beta_j$ is the cylindrical Wiener process.
 The Markov property leads to $v(s,x)=\mathbb{E}v(\frac{s}{2},X(\frac{s}{2},x))$, which can imply
\begin{align}\label{dv}
D v(s,x)l=\frac{2}{s} \mathbb{E}\int_{0}^{\frac{s}{2}}\langle Q^{-\frac{1}{2}} \eta^{l,x}(r), \mathrm{d} \tilde{W}(r)\rangle v\Big(\frac{s}{2}, X\big(\frac{s}{2}, x\big)\Big).
\end{align}
Similarly, we get formulas of the second and third derivatives as follows:
\begin{align}\label{dv2}
D^{2} v(s,x) (l, k)=&~ \frac{2}{s} \mathbb{E}\int_{0}^{\frac{s}{2}}\langle Q^{-\frac{1}{2}} \zeta^{l, k,x}(r), \mathrm{d}  \tilde{W}(r)\rangle v\Big(\frac{s}{2}, X\big(\frac{s}{2}, x\big)\Big) \notag\\
&+\frac{2}{s} \mathbb{E}\int_{0}^{\frac{s}{2}}\langle Q^{-\frac{1}{2}} \eta^{l,x}(r), \mathrm{d}  \tilde{W}(r)\rangle D v\Big(\frac{s}{2}, X\big(\frac{s}{2}, x\big)\Big) \eta^{k,x}\Big(\frac{s}{2}\Big),
\end{align}
and
\begin{align}\label{bdg1}
D^3v(s,x)(l,k,y)=&~\frac{2}{s}\mathbb{E}\int_0^{\frac{s}{2}}\langle Q^{-\frac{1}{2}}\xi^{l,k,y,x}(r),\mathrm{d}\tilde{W}(r)\rangle v\Big(\frac{s}{2},X\big(\frac{s}{2},x\big)\Big)\nonumber\\
&+\frac{2}{s}\mathbb{E}\int_0^{\frac{s}{2}}\langle Q^{-\frac{1}{2}}\zeta^{l,k,x}(r),\mathrm{d}\tilde{W}(r)\rangle Dv\Big(\frac{s}{2},X\big(\frac{s}{2},x\big)\Big)\eta^{y,x}\Big(\frac{s}{2}\Big)\nonumber\\
&+\frac{2}{s}\mathbb{E}\int_0^{\frac{s}{2}}\langle Q^{-\frac{1}{2}}\zeta^{l,y,x}(r),\mathrm{d}\tilde{W}(r)\rangle Dv\Big(\frac{s}{2},X\big(\frac{s}{2},x\big)\Big)\eta^{k,x}\Big(\frac{s}{2}\Big)\nonumber\\
&+\frac{2}{s}\mathbb{E}\int_0^{\frac{s}{2}}\langle Q^{-\frac{1}{2}}\eta^{l,x}(r),\mathrm{d}\tilde{W}(r)\rangle D^2v\Big(\frac{s}{2},X\big(\frac{s}{2},x\big)\Big)\Big(\eta^{k,x}\Big(\frac{s}{2}\Big),\eta^{y,x}\Big(\frac{s}{2}\Big)\Big)\nonumber\\
&+\frac{2}{s}\mathbb{E}\int_0^{\frac{s}{2}}\langle Q^{-\frac{1}{2}}\eta^{l,x}(r),\mathrm{d}\tilde{W}(r)\rangle Dv\Big(\frac{s}{2},X\big(\frac{s}{2},x\big)\Big)\zeta^{k,y,x}\Big(\frac{s}{2}\Big).
\end{align}

Similar arguments as those in \cite[Lemma $6$]{CHS21} yield that for $0<s\leq 1,$
\begin{align*}
&\int_0^s\|(-A)^{\frac{1}{2}}\eta^{l,x}(r)\|^2\mathrm{d}r\leq C\|l\|^2,\quad \int_0^s\|(-A)^{\frac{1}{2}}\zeta^{l,k,x}(r)\|^2\mathrm{d}r\leq C\|l\|^2\|k\|^2,\\
&\int_0^s\|(-A)^{\frac{1}{2}}\xi^{l,k,y,x}(r)\|^2\mathrm{d}r\leq C\|l\|^2\|k\|^2\|y\|^2.
\end{align*}
With this in hand, applying the Burkholder--Davis--Gundy inequality to \eqref{dv}, we get for $0<s\le1,$
\begin{align*}
|Dv(s,x)l|&\leq \frac{2}{s}\Big(\mathbb{E}\Big|\int_0^{\frac{s}{2}}\langle Q^{-\frac{1}{2}}(-A)^{-\frac{1}{2}}(-A)^{\frac{1}{2}}\eta^{l,x}(r),\mathrm{d}\tilde{W}(r)\rangle\Big|^2\Big)^{\frac{1}{2}}\Big(\mathbb{E}\Big|v\Big(\frac{s}{2},X\big(\frac{s}{2},x\big)\Big)\Big|^2\Big)^{\frac{1}{2}}\\
&\leq \frac{C}{s}\Big(\mathbb{E}\int_0^{\frac{s}{2}}\|Q^{-\frac{1}{2}}(-A)^{-\frac{1}{2}}\|_{\mathcal{L}(H)}^2\|(-A)^{\frac{1}{2}}\eta^{l,x}(r)\|^2\mathrm{d}r\Big)^{\frac{1}{2}} M(\Phi)(1+\|x\|^2)\\
&\leq Cs^{-1}M(\Phi)(1+\|x\|^2)\|l\|
\end{align*}
under the assumption $\|Q^{-\frac{1}{2}}(-A)^{-\frac{1}{2}}\|_{\mathcal{L}(H)}<\infty$.
Similarly, applying the Burkholder--Davis--Gundy inequality again to 
 \eqref{dv2} and \eqref{bdg1} leads to that for $0<s\le 1,$
\begin{align}
&|D^2v(s,x)(l,k)|\leq Cs^{-2}M(\Phi)(1+\|x\|^2)\|l\|\|k\|,\notag\\
&|D^3v(s,x)(l,k,y)|\leq Cs^{-3}M(\Phi)(1+\|x\|^2)\|l\|\|k\|\|y\|.\label{d3phi}
\end{align}

The Markov property yields $u(t,x)=\mathbb{E}u(t-1,X(1,x))$. Note that Proposition \ref{propexact} implies
 $|u(t-1,x)-\pi(h)|\leq C\|h\|_{\infty}e^{-C(t-1)}(1+\|x\|^2)$ for $t\ge 1.$ Hence for $t\ge 1,$ we choose
$\Phi_t(x)=u(t-1,x)-\pi(h),$ then
$u(t,x)=\mathbb{E}\Phi_t(X(1,x))+\pi(h)$ with $M(\Phi_t)\leq C\|h\|_{\infty}e^{-C(t-1)}.$ Then, combining \eqref{d3phi} with $s=1$, we get \eqref{step2}.

\textit{Step $3$:} We are going to finish the proof of \eqref{ttge1}.
Combining $u(t,x)=\mathbb{E}u(t-1,X(1,x))$ with the chain's rule, we obtain
\begin{align}
&D^3u(t,x)(l,k,y)=\mathbb{E}[D^3[u(t-1,X(1,x))](l,k,y)]\nonumber\\
=&~\mathbb{E}\Big[D^3u(t-1,X(1,x))(\eta^{l,x}(1),\eta^{k,x}(1),\eta^{y,x}(1))+D^2u(t-1,X(1,x))(\zeta^{l,y,x}(1),\eta^{k,x}(1))\nonumber\\&
\quad+D^2u(t-1,X(1,x))(\eta^{l,x}(1),\zeta^{k,y,x}(1))
+D^2u(t-1,X(1,x))(\zeta^{l,k,x}(1),\eta^{y,x}(1))\nonumber\\
&\quad+Du(t-1,X(1,x))\xi^{l,k,y,x}(1)\Big].\label{d3u2}
\end{align}
Inequalities \eqref{Dutge0}, \eqref{D2utge0} and \eqref{step2} imply that
\begin{align*}
&\mathbb{E}|Du(t-1,X(1,x))\xi^{l,k,y,x}(1)|\leq C\|h\|_{\infty}e^{-C(t-1)}\mathbb{E}\big[(1+\|X(1,x)\|^2)\|\xi^{l,k,y,x}(1)\|\big],\\
&\mathbb{E}|D^2u(t-1,X(1,x))(\eta^{l,x}(1),\zeta^{k,y,x}(1))|
\leq C\|h\|_{\infty}e^{-C(t-1)}\mathbb{E}\big[(1+\|X(1,x)\|^2)\|\eta^{l,x}(1)\|\|\zeta^{k,y,x}(1)\|\big]
\end{align*}
and
\begin{align*}
&~\mathbb{E}|D^3u(t-1,X(1,x))(\eta^{l,x}(1),\eta^{k,x}(1),\eta^{y,x}(1))|\\
\leq &~C\|h\|_{\infty}e^{-C(t-1)}\mathbb{E}\big[(1+\|X(1,x)\|^2)\|\eta^{l,x}(1)\|\|\eta^{k,x}(1)\|\|\eta^{y,x}(1)\|\big].
\end{align*}
 Combining \eqref{xi2} with \eqref{u2}, we get for $t\ge 1,$
\begin{align*}
&\mathbb{E}|Du(t-1,X(1,x))\xi^{l,k,y,x}(1)|
\leq C\|h\|_{\infty}e^{-C(t-1)}(1+\|x\|^2)\|l\|_{-2\delta_2}\|k\|_{-2\delta_3}\|y\|_{-2\delta_1},\\
&\mathbb{E}|D^2u(t-1,X(1,x))(\eta^{l,x}(1),\zeta^{k,y,x}(1))|
\leq C\|h\|_{\infty}e^{-C(t-1)}(1+\|x\|^2)\|l\|_{-2\delta_2}\|k\|_{-2\delta_3}\|y\|_{-2\delta_1}
\end{align*}
and
\begin{align*}
&~\mathbb{E}|D^3u(t-1,X(1,x))(\eta^{l,x}(1),\eta^{k,x}(1),\eta^{y,x}(1))|\\
\leq &~C\|h\|_{\infty}e^{-C(t-1)}(1+\|x\|^2)\|l\|_{-2\delta_2}\|k\|_{-2\delta_3}\|y\|_{-2\delta_1}.
\end{align*}
The other terms of \eqref{d3u2} can be estimated similarly.
Hence, we obtain
\begin{align}\label{tge1}
|D^3u(t,x)(l,k,y)|\leq C\|h\|_{\infty}e^{-C(t-1)}(1+\|x\|^2)\|l\|_{-2\delta_2}\|k\|_{-2\delta_3}\|y\|_{-2\delta_1},\;t\ge 1.
\end{align}
Moreover, by $\eqref{u1}$ for $0<t\leq 1$, \eqref{xi2} and \eqref{d3u2}, we have for $\sum_{i=1}^3\delta_i<1,$
\begin{align}\label{tleq1}
|D^3u(t,x)(l,k,y)|\leq C\|h\|_{3,\infty}t^{-(\delta_1+\delta_2+\delta_3)}\|l\|_{-2\delta_2}\|k\|_{-2\delta_3}\|y\|_{-2\delta_1},\;0<t\leq 1.
\end{align}
The proof is completed.
\end{proof}

Now we show the proof of Lemma \ref{lemmaphi}. 
We only show the detailed proof for regularity estimates of $\phi_{\tau,\epsilon}$ and give the essential discussions for proofs of $\phi$ and $\Gamma^N_{\tau,\epsilon}.$
\begin{proof}[Proof of Lemma \ref{lemmaphi}.]
(\romannumeral1)
 \eqref{phi1} and \eqref{phi2} can be proved by applying 
  \cite[(16)]{B14}, \eqref{Dutge0} and \eqref{D2utge0}.
To be specific, by \eqref{exactmeas}, we get
\begin{align*}
|\phi_{\tau,\epsilon}(x)|\leq C\|h\|_{\infty}(1+\|x\|^2)\int_0^{\infty}e^{-\tau^{\frac{\beta}{2}-\epsilon}t}e^{-Ct}\mathrm{d}t\leq C\|h\|_{\infty}(1+\|x\|^2).
\end{align*}
By \eqref{Dutge0}, we obtain for $l\in H,$
\begin{align*}
&|(-A)^{\delta_1}D\phi_{\tau,\epsilon}(x)l|\leq \int_0^{\infty}|(-A)^{\delta_1}Du(t,x)l|\mathrm{d}t\\
\leq&~ \int_0^{\infty}C\|h\|_{1,\infty}(1+\|x\|^2)(1+t^{-\delta_1})e^{-Ct}\|l\|\mathrm{d}t
\leq C\|h\|_{1,\infty}(1+\|x\|^2)\|l\|.
\end{align*}
By \eqref{D2utge0}, we obtain for $l,k\in H,$
\begin{align*}
&|\langle(-A)^{\delta_1}D^2\phi_{\tau,\epsilon}(-A)^{\delta_2}l,k\rangle|\\
\leq&\int_0^{\infty}|\langle(-A)^{\delta_1}D^2u(t,x)(-A)^{\delta_2}l,k\rangle|\mathrm{d}t\\
\leq& \int_0^{\infty}C\|h\|_{2,\infty}(1+\|x\|^2)(1+t^{-(\delta_1+\delta_2)})e^{-Ct}\|l\|\|k\|\mathrm{d}t\\
\leq &~C\|h\|_{2,\infty}(1+\|x\|^2)\|l\|\|k\|.
\end{align*}

Combining \eqref{ttge1}, \eqref{phi3} can be obtained by
\begin{align*}
&\big|\big\langle(-A)^{\delta_1}D^3\phi_{\tau,\epsilon}(x)((-A)^{\delta_2}l,(-A)^{\delta_3}k),y\big\rangle\big|\\
\leq&~\int_0^{\infty}\big|e^{-\tau^{\frac{\beta}{2}-\epsilon}t}\big\langle(-A)^{\delta_1}D^3u(t,x)((-A)^{\delta_2}l,(-A)^{\delta_3}k),y\big\rangle\big|\mathrm{d}t\\
\leq &~C\|h\|_{3,\infty}(1+\|x\|^2)\|l\|\|k\|\|y\|
\end{align*}
and taking supremum with respect to $\|l\|,\|k\|,\|y\|\leq 1.$ The proof of \eqref{phi3} is finished.
The proof of \eqref{phi4} is similar to that of \eqref{phi3} based on \eqref{D4utge0}. We omit it.

(\romannumeral 2)
The proof of (\romannumeral 2) is similar to that of (\romannumeral1). We omit it. 

(\romannumeral 3)
The proof of (\romannumeral 3) is similar to that of (\romannumeral 1),
 the difference lies in the nonhomogeneous source terms of \eqref{poisson2} and \eqref{GammaNtau}. In (\romannumeral 3), the nonhomogeneous source term $\phi_{\tau,\epsilon}$ of \eqref{GammaNtau} satisfies $$\max\{\|D\phi_{\tau,\epsilon}(x)\|,\|D^2\phi_{\tau,\epsilon}(x)\|_{\mathcal{L}(H)}\}\leq C(1+\|x\|^2),$$
 while in (\romannumeral1) $h\in\mathcal{C}^4_b.$
 Noting Remark \ref{rem1},
the results here can be proved similarly as those in (\romannumeral1).

The proof is completed.
\end{proof}

We are in a position to give the proof of Lemma \ref{lemmaDDphi}, which is devoted to presenting some space-independent regularity estimates of the modified Poisson's equation \eqref{poisson} and the asymptotic relationship between $\phi$ and the solution 
$\phi_{\tau,\epsilon}$ of the modified Poisson's equation \eqref{poisson}.

\begin{proof}[Proof of Lemma \ref{lemmaDDphi}.]
(\romannumeral1)
We only show the proof of \eqref{phi7} since \eqref{phi5} and \eqref{phi6} can be proved similarly. 

For $t\ge1,$ combining \eqref{d3u2} and the fact that $u\in\mathcal{C}^4_b$ (see e.g.  \cite[Proposition $4.4.1$]{pde01}), we get
\begin{align*}
|D^3u(t,x)(l,k,y)|\leq &~C\mathbb{E}\big(\|\eta^{l,x}(1)\|\|\eta^{k,x}(1)\|\|\eta^{y,x}(1)\|+\|\zeta^{l,y,x}(1)\|\|\eta^{k,x}(1)\|\\
&+\|\eta^{l,x}(1)\|\|\zeta^{k,y,x}(1)\|+\|\zeta^{l,k,x}(1)\|\|\eta^{y,x}(1)\|+\|\xi^{l,k,y,x}(1)\|\big)\\
\leq&~ C\|l\|_{-2\delta_2}\|k\|_{-2\delta_3}\|y\|_{-2\delta_1},
\end{align*}
where in the last step we use \eqref{xi2} and \eqref{u2}.
 This, together with \eqref{tleq1} implies
\begin{align*}
&~|D^3\phi_{\tau,\epsilon}(x)(l,k,y)|=\Big|\int_0^{\infty}e^{-\tau^{\frac{\beta}{2}-\epsilon}t}D^3u(t,x)(l,k,y)\mathrm{d}t\Big|\\
\leq&~ C\Big(\int_0^1e^{-\tau^{\frac{\beta}{2}-\epsilon}t}t^{-(\delta_1+\delta_2+\delta_3)}\mathrm{d}t+\int_1^{\infty}e^{-\tau^{\frac{\beta}{2}-\epsilon}t}\mathrm{d}t\Big)
\|l\|_{-2\delta_2}\|k\|_{-2\delta_3}\|y\|_{-2\delta_1}\\
\leq&~ C\tau^{-\frac{\beta}{2}+\epsilon}\|l\|_{-2\delta_2}\|k\|_{-2\delta_3}\|y\|_{-2\delta_1}.
\end{align*}
The proof is finished.

(\romannumeral2) 
 We deduce from \eqref{Dutge0} and the finiteness of algebraic moments for $\pi$ that for any $\epsilon\in(0,\frac{\beta}{2}),$
\begin{align*}
&~\int_H\|(-A)^{\delta_1}(D\phi_{\tau,\epsilon}(x)-D\phi(x))\|^2\pi(\mathrm{d}x)\\
\leq&~ \int_H\Big[\int_0^{\infty}\big|e^{-\tau^{\frac{\beta}{2}-\epsilon}t}-1\big|\|(-A)^{\delta_1}Du(t,x)\|\mathrm{d}t\Big]^2\pi(\mathrm{d}x)\\
\leq &~C\Big[\int_1^{\infty}|e^{-\tau^{\frac{\beta}{2}-\epsilon}t}-1\big|e^{-C(t-1)}\mathrm{d}t+\int_0^1|e^{-\tau^{\frac{\beta}{2}-\epsilon}t}-1\big|(1+t^{-\delta_1})\mathrm{d}t\Big]^2
\int_H(1+\|x\|^4)\pi(\mathrm{d}x)\end{align*}
converges to $0\text{ as }\tau\rightarrow0,
$
where in the last step we use Lebesgue's dominated convergence theorem.
The proof is finished.
\end{proof}

\section{Proofs of Propositions in Section \ref{mainresult}}\label{sec4}
In this section, we provide detailed proofs of Propositions \ref{prop2}, \ref{prop3} and  \ref{prop4}.

\subsection{Proof of Proposition \ref{prop2}}

In this subsection, we prove Proposition \ref{prop2} based on the martingale type CLT (see e.g. \cite{MDA74}) and the space-dependent regularity estimates of Poisson-type equations associated to \eqref{spde} in subsection \ref{poissontype}.
\begin{proof}[Proof of Proposition \ref{prop2}.]
Denote $Z_k:=\big\langle \int_{t_k}^{t_{k+1}} S^N(t-t_k)P^N\mathrm{d}W(t), D\phi_{\tau,\epsilon}(X^N_k)\big\rangle.$ Note that $\tau^{\frac{\beta}{2}} Z_k$ is $\mathcal{F}_{t_{k+1}}$-measurable, and $\mathbb{E}[\tau^{\frac{\beta}{2}} Z_k|\mathcal{F}_{t_k}]=0, $ a.s. Hence, $\{\tau^{\frac{\beta}{2}} Z_k,0\leq k\leq m-1\}$ is a martingale difference series. We deduce from \cite[Theorem $2.3$]{MDA74} that it suffices to verify that\\
$
\mathbb{E}\Big[\tau^{\beta}\max_{0\leq i\leq m-1}|Z_i|^2\Big]\rightarrow 0\text{ as  }\tau\rightarrow 0$ and 
$\tau^{\beta}\sum_{i=0}^{m-1}Z^2_i\overset{\mathbb{P}}{\longrightarrow} \pi(f) \text{ as }\tau\rightarrow 0.$

\textbf{Claim $1$:} $
\mathbb{E}\Big[\tau^{\beta}\max_{0\leq i\leq m-1}|Z_i|^2\Big]\rightarrow 0\text{ as  }\tau\rightarrow 0$.
  
\textbf{Proof of Claim $1$.} It is clear that
\begin{align*}
&\mathbb{E}\Big[\tau^{\beta}\max_{0\leq i\leq m-1}|Z_i|^2\Big]\\
\leq&~ \mathbb{E}\Big[\tau^{\beta}\max_{0\leq i\leq m-1}|Z_i|^2\mathbbm{1}_{\{|Z_i|^{2}\leq 1\}}\Big]+\mathbb{E}\Big[\tau^{\beta}\max_{0\leq i\leq m-1}|Z_i|^{2}\mathbbm{1}_{\{|Z_i|^2>1\}}\Big]\\
\leq&~ \tau^{\beta}+\sum_{i=0}^{m-1}\tau^{\beta}\mathbb{E}[|Z_i|^4].
\end{align*}
By \eqref{phi1}, we have
\begin{align}\label{z^4}
\mathbb{E}|Z_i|^4\leq&~ \mathbb{E}\Big[\|(-A)^{\frac{1-\beta}{2}}D\phi_{\tau,\epsilon}(X^N_i)\|^4\Big\|\int_{t_i}^{t_{i+1}} S^N(t-t_i)(-A)^{\frac{\beta-1}{2}}P^N\mathrm{d}W(t)\Big\|^4\Big]\nonumber\\
\leq&~ C\mathbb{E}\Big[(1+\|X^N_i\|^8)\|h\|^4_{1,\infty}\Big\|\int_{t_i}^{t_{i+1}}S^N(t-t_i)(-A)^{\frac{\beta-1}{2}}\mathrm{d}W(t)\Big\|^4\Big]\nonumber\\
\leq&~C\big(\mathbb{E}[1+\|X^N_i\|^{16}]\big)^{\frac{1}{2}}\Big(\mathbb{E}\Big\|\int_{t_i}^{t_{i+1}}S^N(t-t_i)(-A)^{\frac{\beta-1}{2}}\mathrm{d}W(t)\Big\|^8\Big)^{\frac{1}{2}},
\end{align}
where in the third step we use the H\"{o}lder inequality. 
Applying the Burkholder--Davis--Gundy inequality to \eqref{z^4}, we get
$\mathbb{E}|Z_i|^4\leq C\tau^2.$
Hence, we have
$
\mathbb{E}\big[\tau^{\beta}\max_{0\leq i\leq m-1}|Z_i|^2\big]
\leq C(\tau^{\beta}+\tau).
$

\textbf{Claim $2$:} $\tau^{\beta}\sum_{i=0}^{m-1}Z^2_i\overset{\mathbb{P}}{\longrightarrow} \pi(f) \text{ as }\tau\rightarrow 0.$

\textbf{Proof of Claim $2$.} Ch{\'e}byshev's inequality indicates that it suffices to prove the convergence in the mean square sense.
Note that
\begin{align*}
 &~\mathbb{E}\Big[\tau^{\beta}\sum_{i=0}^{m-1}\big(Z_i^2-\tau\pi(f)\big)\Big]^2\\
\leq\,&~ 3\mathbb{E}\Big[\tau^{\beta}\sum_{i=0}^{m-1}(Z_i^2-\pi^N_{\tau}(\bar{f}^{N,i}_{\tau,\epsilon}))\Big]^2+3\Big[\tau^{\beta}\sum_{i=0}^{m-1}(\pi^N_{\tau}(\bar{f}^{N,i}_{\tau,\epsilon})-\tau\pi(f_{\tau,\epsilon}))\Big]^2+3\Big[\pi(f_{\tau,\epsilon})-\pi(f)\Big]^2\\
=:&~3(I_1+I_2+I_3),
\end{align*}
where $\bar{f}^{N,i}_{\tau,\epsilon}(\cdot):=\int_{t_i}^{t_{i+1}}\| P^NQ^{\frac{1}{2}}S^N(t-t_i)D\phi_{\tau,\epsilon}(\cdot)\|^2\mathrm{d}t$ and $f_{\tau,\epsilon}(\cdot):=\| Q^{\frac{1}{2}}D\phi_{\tau,\epsilon}(\cdot)\|^2.$ It suffices to prove $I_i\rightarrow 0$ as $\tau\rightarrow 0$ for $i=1,2,3.$

The term $I_1$ can be divided as 
\begin{align*}
I_1=\,&~\tau^{2\beta}\sum_{i=0}^{m-1}\mathbb{E}\big[Z_i^2-\pi^N_{\tau}(\bar{f}^{N,i}_{\tau,\epsilon})\big]^2+2\tau^{2\beta}\sum_{0\leq i<j\leq m-1}\mathbb{E}\Big[(Z_i^2-\pi^N_{\tau}(\bar{f}^{N,i}_{\tau,\epsilon}))(Z_j^2-\pi^N_{\tau}(\bar{f}^{N,j}_{\tau,\epsilon}))\Big]\\
=:&~I_{1,1}+I_{1,2}.
\end{align*}
Similarly to \eqref{estimatef_1} and \eqref{estimatef_2}, it can be proved that $|\bar{f}^{N,i}_{\tau,\epsilon}(x)|\leq C\tau(1+\|x\|^4)$ and $|D\bar{f}^{N,i}_{\tau,\epsilon}(x)|\leq C\tau (1+\|x\|^4)$ for $x\in H$. In view of the finiteness of algebraic moments for $\pi^N_{\tau}$ (see Proposition \ref{propfull}), we have
\begin{align}\label{barf}\pi^N_{\tau}(\bar{f}^{N,i}_{\tau,\epsilon})\leq C\tau.
\end{align}
By \eqref{z^4} and \eqref{barf}, it can be verified that
\begin{align*}
I_{1,1}&\leq 2\tau^{2\beta}\sum_{i=0}^{m-1}\Big[\mathbb{E}Z_i^4+(\pi^N_{\tau}(\bar{f}^{N,i}_{\tau,\epsilon}))^2\Big]
\leq C\tau^{1+\beta}.
\end{align*}
Using the properties of the conditional expectation, the term $I_{1,2}$ can be estimated by
\begin{align*}
|I_{1,2}|&\leq2\tau^{2\beta}\sum_{0\leq i<j\leq m-1}\Big|\mathbb{E}\Big[(Z_i^2-\pi^N_{\tau}(\bar{f}^{N,i}_{\tau,\epsilon}))\mathbb{E}\big[Z_j^2-\pi^N_{\tau}(\bar{f}^{N,j}_{\tau,\epsilon})\big|\mathcal{F}_{t_{i+1}}\big]\Big]\Big|\\
&=2\tau^{2\beta}\sum_{0\leq i<j\leq m-1}\Big|\mathbb{E}\Big[(Z_i^2-\pi^N_{\tau}(\bar{f}^{N,i}_{\tau,\epsilon}))\mathbb{E}\big[\bar{f}^{N,j}_{\tau,\epsilon}(X^N_j)-\pi^N_{\tau}(\bar{f}^{N,j}_{\tau,\epsilon})\big|\mathcal{F}_{t_{i+1}}\big]\Big]\Big|.
\end{align*}
By Remark \ref{rem1}, \eqref{z^4} and \eqref{barf}, we obtain that 
\begin{align*}
|I_{1,2}|
&\leq C\tau^{2\beta+1}\sum_{0\leq i<j\leq m-1}\Big[\mathbb{E}(Z_i^2-\pi^N_{\tau}(\bar{f}^{N,i}_{\tau,\epsilon}))^2\Big]^{\frac{1}{2}}\Big[\sup_{k\ge 0}\mathbb{E}(1+\|X^N_k\|^{10})\Big]^{\frac{1}{2}}e^{-\frac{1}{2}(\lambda_1-K)(j-i-1)\tau}\\
&\leq C\tau^{2\beta+2}\sum_{0\leq i<j\leq m-1}e^{-\frac{1}{2}(\lambda_1-K)(j-i)\tau}\\
&\leq C\tau^{2\beta+2}\Big[\sum_{0<j-i\le \frac{\ln m}{\tau}}e^{-\frac{1}{2}(\lambda_1-K)(j-i)\tau}+\sum_{\frac{\ln m}{\tau}<j-i\leq m-1 }e^{-\frac{1}{2}(\lambda_1-K)(j-i)\tau}\times\mathbbm{1}_{\{\ln m<\tau (m-1)\}}\Big].
\end{align*}
Combining $m=\tau^{-1-\beta},$ we get
\begin{align*}
|I_{1,2}|\leq Cm^{-2}\big( m\ln m/\tau+e^{-\frac{1}{2}(\lambda_1-K)\ln m}m^2\big)\rightarrow 0\text{ as }\tau\rightarrow 0.
\end{align*}
 Let $f^N_{\tau,\epsilon}(\cdot):=\|P^NQ^{\frac{1}{2}}D\phi_{\tau,\epsilon}(\cdot)\|^2.$ Then $f^N_{\tau,\epsilon}$ has the same estimates as \eqref{estimatef_1} and \eqref{estimatef_2}, i.e., $|f^N_{\tau,\epsilon}(x)|\leq C\|h\|^2_{1,\infty}(1+\|x\|^4)$ and $f\in\mathcal{C}^1_4.$ We can divide the ingredient of the summation in $I_2$ as 
follows
\begin{align*}
\pi^N_{\tau}(\bar{f}^{N,i}_{\tau,\epsilon})-\pi(\tau f_{\tau,\epsilon})=~&~\pi^N_{\tau}(\bar{f}^{N,i}_{\tau,\epsilon}-\tau f^N_{\tau,\epsilon})+\tau[\pi^N_{\tau}(f^N_{\tau,\epsilon})-\pi^N( f^N_{\tau,\epsilon})]+\tau[\pi^N( f^N_{\tau,\epsilon})-\pi(f_{\tau,\epsilon})]\\
=:&~I^i_{2,1}+I_{2,2}+I_{2,3}.
\end{align*}
It is clear that
\begin{align*}
|I^i_{2,1}|&=\Big|\int_{H_N}\int_{t_i}^{t_{i+1}}\Big(\|P^N Q^{\frac{1}{2}}S^N(t-t_i)D\phi_{\tau,\epsilon}(x)\|^2-\| P^NQ^{\frac{1}{2}}D\phi_{\tau,\epsilon}(x)\|^2\Big)\mathrm{d}t\pi^N_{\tau}(\mathrm{d}x)\Big|\\
&=\Big|\int_{H_N}\int_{t_i}^{t_{i+1}}\big\langle P^NQ^{\frac{1}{2}}(S^N(t-t_i)-\mathrm{Id})D\phi_{\tau,\epsilon}(x),Q^{\frac{1}{2}}(S^N(t-t_i)+\mathrm{Id})D\phi_{\tau,\epsilon}(x)\big\rangle\mathrm{d}t\pi^N_{\tau}(\mathrm{d}x)\Big|\\
&\leq \int_{H_N}\int_{t_i}^{t_{i+1}}\| P^N(-A)^{\frac{\beta-1}{2}}Q^{\frac{1}{2}}(-A)^{-\frac{\beta+1}{2}+\gamma}(S^N(t-t_i)-\mathrm{Id})(-A)^{1-\gamma}D\phi_{\tau,\epsilon}(x)\|\times\\
&\qquad\qquad \qquad\|(-A)^{\frac{\beta-1}{2}}Q^{\frac{1}{2}}(S^N(t-t_i)+\mathrm{Id})(-A)^{\frac{1-\beta}{2}}D\phi_{\tau,\epsilon}(x)\|\mathrm{d}t\pi^N_{\tau}(\mathrm{d}x)\\
&\leq \int_{H_N}\int_{t_i}^{t_{i+1}}C\| (-A^N)^{-\frac{1+\beta}{2}+\gamma}(S^N(t-t_i)-\mathrm{Id})\|_{\mathcal{L}(H)}\|(-A)^{\frac{\beta-1}{2}}Q^{\frac{1}{2}}\|^2_{\mathcal{L}(H)}\times\\
&\qquad \qquad \qquad \|(-A)^{1-\gamma}D\phi_{\tau,\epsilon}(x)\|\|(-A)^{\frac{1-\beta}{2}}D\phi_{\tau,\epsilon}(x)\|\mathrm{d}t\pi^N_{\tau}(\mathrm{d}x),
\end{align*}
where in the last step we use the fact that $(S^N(t-t_i)+\mathrm{Id})\in \mathcal{L}(H)$.
Utilizing \eqref{semigroup2}, \eqref{phi1} and the finiteness of algebraic moments for $\pi^N_{\tau}$ (see Proposition \ref{propfull}), we obtain
\begin{align*}
|I^i_{2,1}| &\leq C\tau\int_{H_N}\tau^{\frac{1+\beta}{2}-\gamma}\|(-A)^{\frac{\beta-1}{2}}Q^{\frac{1}{2}}\|^2_{\mathcal{L}(H)}\|(-A)^{1-\gamma}D\phi_{\tau,\epsilon}(x)\|\|(-A)^{\frac{1-\beta}{2}}D\phi_{\tau,\epsilon}(x)\|\pi^N_{\tau}(\mathrm{d}x)\\
&\leq C\|(-A)^{\frac{\beta-1}{2}}Q^{\frac{1}{2}}\|^2_{\mathcal{L}(H)}\|h\|^2_{1,\infty}\tau^{\frac{1+\beta}{2}-\gamma+1},
\end{align*}
and hence $\tau^{\beta}\sum_{i=0}^{m-1}\pi^N_{\tau}(\bar{f}^{N,i}_{\tau,\epsilon}-\tau f^N_{\tau,\epsilon})\rightarrow 0$ as $\tau\rightarrow 0$.

For the term $I_{2,3},$
\begin{align*}
\pi^N(f^N_{\tau,\epsilon})-\pi(f_{\tau,\epsilon})=\big[\pi^N(f^N_{\tau,\epsilon})-\pi(f^N_{\tau,\epsilon})\big]+\big[\pi(f^N_{\tau,\epsilon})-\pi(f_{\tau,\epsilon})\big]=:I^1_{2,3}+I^2_{2,3}.
\end{align*}
Using \eqref{phi1} and the finiteness of the algebraic   moment of $\pi$ (see Proposition \ref{propexact}), the term $I^2_{2,3}$ can be estimated as
\begin{align*}
|I^2_{2,3}|=&~|\pi(f^N_{\tau,\epsilon}-f_{\tau,\epsilon})|=\pi\big(\|(P^N-\mathrm{Id})Q^{\frac{1}{2}}D\phi_{\tau,\epsilon}\|^2\big)\\
\leq&~
\|P^N-\mathrm{Id}\|^2_{\mathcal{L}(H)}\|(-A)^{\frac{\beta-1}{2}}Q^{\frac{1}{2}}\|^2_{\mathcal{L}(H)}\pi\big( \|(-A)^{\frac{1-\beta}{2}}D\phi_{\tau,\epsilon}(\cdot)\|^2\big)\\
\leq&~ C\|P^N-\mathrm{Id}\|^2_{\mathcal{L}(H)}\|(-A)^{\frac{\beta-1}{2}}Q^{\frac{1}{2}}\|^2_{\mathcal{L}(H)}(1+\pi(\|\cdot\|^4)),
\end{align*}
which goes to $0$ as $\tau\rightarrow 0$. Actually, $\tau\rightarrow0$ means $N\to\infty.$ 
The remaining terms $I_{2,2}$ and $I^1_{2,3}$ can be estimated according to the fact that $f^N_{\tau,\epsilon}\in \mathcal{C}^1_4$ and Remark \ref{remark2}.

It remains to prove $I_3\rightarrow 0$ as $\tau\rightarrow 0.$ We deduce from H\"older's inequality, \eqref{phi1} and Lemma \ref{lemmaDDphi} that
\begin{align*}
I_3\leq&~ \Big[\int_{H}\|Q^{\frac{1}{2}}D\phi_{\tau,\epsilon}(x)\|^2-\|Q^{\frac{1}{2}}D\phi(x)\|^2\pi(\mathrm{d}x)\Big]^2\\
=&~ \Big[\int_{H}\langle Q^{\frac{1}{2}}\big(D\phi_{\tau,\epsilon}(x)-D\phi(x)\big),Q^{\frac{1}{2}}\big(D\phi_{\tau,\epsilon}(x)+D\phi(x)\big)\rangle\pi(\mathrm{d}x)\Big]^2\\
\leq &~C\|Q\|^2_{\mathcal{L}(H)}\int_{H}\|D\phi_{\tau,\epsilon}(x)-D\phi(x)\|^2\pi(\mathrm{d}x)\int_{H}\|D\phi_{\tau,\epsilon}(x)+D\phi(x)\|^2\pi(\mathrm{d}x)
\rightarrow 0 \text{ as }\tau\rightarrow 0.
\end{align*}

The proof is finished.
\end{proof}

\subsection{Proof of Proposition \ref{prop3}}
Before giving the proof of Proposition \ref{prop3}, we present the following lemma which paves a way for proving Proposition \ref{prop3}.
\begin{lemma}\label{lemmaT}
Let $\Theta:H\times\mathbb{R}^{N}\rightarrow \mathbb{R}$ be a measurable function, and let
$\mathbb{Y}_k$ be $H$-valued $\mathcal{F}_{t_k}$-measurable random variable, and $\{r^1_{k},\ldots,r^N_k\}$ be real-valued $\mathcal{F}_{t_{k+1}}$-measurable and $\mathcal{F}_{t_k}$-independent random variables
 for $k=0,1,\ldots,m-1.$
If for $k=0,1,\ldots,m-1,$
$$\mathbb{E}[\Theta(\mathbb{Y}_k,r^1_{k},\ldots,r^N_k)\big|\mathcal{F}_{t_k}]=0$$ and \begin{align*}
\sum_{n=2}^{\infty}\frac{1}{n!}\mathbb{E}\Big[\big|\Theta(\mathbb{Y}_k,r^1_{k},\ldots,r^N_k)\big|^n\Big|\mathcal{F}_{t_k}\Big]\leq \frac{C}{m}
\end{align*}
 for some $C>0,$ 
then 
\begin{align*}
\mathbb{E}\Big[\exp\Big\{\sum_{k=0}^{m-1}\Theta(\mathbb{Y}_k,r^1_{k},\ldots,r^N_k)\Big\}\Big]\leq e^{C}.
\end{align*}
\end{lemma}
\begin{proof}
By Taylor's expansion, we have
\begin{align*} 
&\mathbb{E}[\exp\{\Theta(\mathbb{Y}_k,r^1_{k},\ldots,r^N_k)\}\big|\mathcal{F}_{t_k}]\\
\leq&~1+\sum_{n=2}^{\infty}\frac{1}{n!}\mathbb{E}\Big[\big|\Theta(\mathbb{Y}_k,r^1_{k},\ldots,r^N_k)\big|^n\Big|\mathcal{F}_{t_k}\Big]\leq 1+\frac{C}{m}.
\end{align*}
Hence, by iterating, we get
\begin{align*}
&\mathbb{E}\Big[\exp\Big\{\sum_{k=0}^{m-1}\Theta(\mathbb{Y}_k,r^1_{k},\ldots,r^N_k)\Big\}\Big]\\
=&~\mathbb{E}\Big[\exp\Big\{\sum_{k=0}^{m-2}\Theta(\mathbb{Y}_k,r^1_{k},\ldots,r^N_k)\Big\}\mathbb{E}\big[\exp\big\{\Theta(\mathbb{Y}_{m-1},r^1_{m-1},\ldots,r^N_{m-1})\big\}\big|\mathcal{F}_{t_{m-1}}\big]\Big]\\
\leq&~ \big(1+\frac{C}{m}\big)^m\leq e^{C}.
\end{align*}
The proof is finished.
\end{proof}
With this lemma in hand, we give the proof of Proposition \ref{prop3}.
\begin{proof}[Proof of Proposition \ref{prop3}.]
Below, we give estimates of $\mathcal{R}^N_{\tau, i},i=1,\ldots,5$ introduced in subsection \ref{proofmain}.

\textbf{Estimate of $\mathcal{R}^N_{\tau, 1}$.}
By Ch{\'e}byshev's inequality and \eqref{phi1}, for each fixed $a>0,$
\begin{align*}
\mathbb{P}(|\mathcal{R}^N_{\tau, 1}|\ge a)\leq a^{-1}\tau^{\frac{\beta}{2}}\mathbb{E}|\phi_{\tau,\epsilon}(X^N_m)-\phi_{\tau,\epsilon}(X^N_0)|\leq a^{-1}\tau^{\frac{\beta}{2}}C\|h\|_{\infty}(1+\sup_{k\in\mathbb{N}}\mathbb{E}\|X^N_k\|^2)\rightarrow 0 \text{ as }\tau\rightarrow0.
\end{align*} 

\textbf{Estimate of $\mathcal{R}^N_{\tau, 2}$.}
For $\gamma>0$ small enough, by the linear growth property of $F$ and \eqref{phi1},
\begin{align*}
\mathbb{P}(|\mathcal{R}^N_{\tau, 2}|\ge a)\leq &~a^{-1}\tau^{\frac{\beta}{2}}\sum_{k=0}^{m-1}\mathbb{E}\Big|\Big\langle \int_{t_k}^{t_{k+1}}(\mathrm{Id}-S^N(t-t_k))F(X^N_k)\mathrm{d}t,D\phi_{\tau,\epsilon}(X^N_k)\Big\rangle \Big|\\
\leq&~ C\tau^{1+\frac{\beta}{2}}\sum_{k=0}^{m-1}\tau^{1-\gamma}\mathbb{E}\Big[\|F(X^N_k)\|\|(-A)^{1-\gamma}D\phi_{\tau,\epsilon}(X^N_k)\|\Big]\\
\leq&~ C\tau^{1-\frac{\beta}{2}-\gamma}(1+\sup_{k\in\mathbb{N}}\mathbb{E}\|X^N_k\|^3)\rightarrow 0 \text{ as }\tau\rightarrow0.
\end{align*}

\textbf{Estimate of $\mathcal{R}^N_{\tau, 3}$.}
Noting that $D^2\phi_{\tau,\epsilon}(x)=(D^2\phi_{\tau,\epsilon}(x))^*,$ the term $\mathcal{R}^N_{\tau,3}$ can be divided into four terms denoted by $\mathcal{R}^{N,i}_{\tau,3},i=1,2,3,4,$
where 
\begin{align*}
\mathcal{R}^{N,1}_{\tau, 3}:=&\frac{1}{2}\tau^{1+\frac{\beta}{2}}\sum_{k=0}^{m-1}\mathrm{Tr}(QD^2\phi_{\tau,\epsilon}(X^N_k))-\frac{1}{2}\tau^{\frac{\beta}{2}}\sum_{k=0}^{m-1}\big\langle S^N(\tau)P^N\Delta W_k,D^2\phi_{\tau,\epsilon}(X^N_k)S^N(\tau)P^N\Delta W_k\big\rangle,\\
\mathcal{R}^{N,2}_{\tau,3}:=&-\frac{1}{2}\tau^{\frac{\beta}{2}}\sum_{k=0}^{m-1}\Big\langle (S^N(\tau)-\mathrm{Id})X^N_k+\tau S^N(\tau)F^N(X^N_k),\\
& \qquad \qquad \quad \;\;D^2\phi_{\tau,\epsilon}(X^N_k)\big((S^N(\tau)-\mathrm{Id})X^N_k+\tau S^N(\tau)F^N(X^N_k)\big)\Big\rangle,\\
\mathcal{R}^{N,3}_{\tau, 3}:=&-\sum_{k=0}^{m-1}\tau^{\frac{\beta}{2}}\big\langle D^2\phi_{\tau,\epsilon}(X^N_k)(S^N(\tau)-\mathrm{Id})X^N_k,S^N(\tau)P^N\Delta W_k\big\rangle, \\
\mathcal{R}^{N,4}_{\tau, 3}:=&-\sum_{k=0}^{m-1}\tau^{\frac{\beta}{2}}\big\langle S^N(\tau)P^N\Delta W_k,D^2\phi_{\tau,\epsilon}(X^N_k)\tau S^N(\tau)F^N(X^N_k)\big\rangle.
\end{align*}

\textit{Estimate of $\mathcal{R}^{N,1}_{\tau,3}.$} The term $\mathcal{R}^{N,1}_{\tau,3}$ can be further divided into four terms by inserting some terms, precisely, we have
\begin{align*}
\mathcal{R}^{N,1}_{\tau,3}=\,&~\frac{1}{2}\tau^{1+\frac{\beta}{2}}\sum_{k=0}^{m-1}\Big[\mathrm{Tr}(QD^2\phi_{\tau,\epsilon}(X^N_k))-\mathrm{Tr}(P^NQ^{\frac{1}{2}}D^2\phi_{\tau,\epsilon}(X^N_k)S^N(\tau)Q^{\frac{1}{2}})\Big]\\
&+\frac{1}{2}\tau^{1+\frac{\beta}{2}}\sum_{k=0}^{m-1}\mathrm{Tr}(P^NQ^{\frac{1}{2}}(\mathrm{Id}-S^N(\tau))D^2\phi_{\tau,\epsilon}(X^N_k)S^N(\tau)Q^{\frac{1}{2}})\\
&+\frac{1}{2}\tau^{1+\frac{\beta}{2}}\sum_{k=0}^{m-1}\sum_{j=1}^N\langle Q^{\frac{1}{2}}S^N(\tau)D^2\phi_{\tau,\epsilon}(X^N_k)S^N(\tau)Q^{\frac{1}{2}}e_j,e_j\rangle\Big[1-\big(\frac{\Delta_k\beta_j}{\sqrt{\tau}}\big)^2\Big]\\
&-\frac{1}{2}\tau^{1+\frac{\beta}{2}}\sum_{k=0}^{m-1}\sum_{i\neq j}\langle Q^{\frac{1}{2}}S^N(\tau)D^2\phi_{\tau,\epsilon}(X^N_k)S^N(\tau)Q^{\frac{1}{2}}e_i,e_j\rangle\frac{\Delta_k\beta_i}{\sqrt{\tau}}\frac{\Delta_k\beta_j}{\sqrt{\tau}}\\
=:&~\mathcal{R}^{N,1,1}_{\tau,3}+\mathcal{R}^{N,1,2}_{\tau,3}+\mathcal{R}^{N,1,3}_{\tau,3}+\mathcal{R}^{N,1,4}_{\tau,3}.
\end{align*}
Noting that
\begin{align*}
\mathbb{P}(\mathcal{R}^{N,1}_{\tau,3}\ge a)\leq \sum_{i=1}^4\mathbb{P}(\mathcal{R}^{N,1,i}_{\tau,3}\ge \frac{a}{4}),\quad \mathbb{P}(-\mathcal{R}^{N,1}_{\tau,3}\ge a)\leq \sum_{i=1}^4\mathbb{P}(-\mathcal{R}^{N,1,i}_{\tau,3}\ge \frac{a}{4}),
\end{align*}
it suffices to prove that $\mathbb{P}(\mathcal{R}^{N,1,i}_{\tau,3}\ge \frac{a}{4})$ and $\mathbb{P}(-\mathcal{R}^{N,1,i}_{\tau,3}\ge \frac{a}{4})$ converge to zero as $\tau\rightarrow0$ for $i=1,2,3,4.$

For the term $\mathcal{R}^{N,1,1}_{\tau,3}$, by using \cite[Lemma $16.20$]{MV} and the assumption $Q$ commutes with $A$, we get
 $$\mathrm{Tr}\big(P^NQ^{\frac{1}{2}}Q^{\frac{1}{2}}D^2\phi_{\tau,\epsilon}(x)\big)=\mathrm{Tr}\big(Q^{\frac{1}{2}}D^2\phi_{\tau,\epsilon}(x)P^NQ^{\frac{1}{2}}\big)$$ and 
 \begin{align*}
 &\mathrm{Tr}\big((\mathrm{Id}-P^N)Q^{\frac{1}{2}}Q^{\frac{1}{2}}D^2\phi_{\tau,\epsilon}(x)\big)\\
 =&~\mathrm{Tr}\big((\mathrm{Id}-P^N)(-A)^{-\beta+\gamma}(-A)^{\frac{\beta-1}{2}}Q^{\frac{1}{2}}(-A)^{\frac{\beta-1}{2}}Q^{\frac{1}{2}}(-A)^{1-\gamma}D^2\phi_{\tau,\epsilon}(x)\big)\\
 =&~\mathrm{Tr}\big((-A)^{\frac{\beta-1}{2}}Q^{\frac{1}{2}}(-A)^{1-\gamma}D^2\phi_{\tau,\epsilon}(x)(\mathrm{Id}-P^N)(-A)^{-\beta+\gamma}(-A)^{\frac{\beta-1}{2}}Q^{\frac{1}{2}}\big).
 \end{align*}
 Therefore, applying \eqref{phi2}, the assumption that $A$ commutes with $Q$, the inequality $\|(\mathrm{Id}-P^N)(-A)^{-\beta+\gamma}\|_{\mathcal{L}(H)}\leq N^{-2(\beta-\gamma)}$ with $\gamma\leq \beta$ and the fact that $P^Ne_j=e_j$ for $j=1,\ldots,N$, we have
\begin{align*}
&\quad \mathbb{E}|\mathcal{R}^{N,1,1}_{\tau,3}|\\
&=\frac{1}{2}\tau^{1+\frac{\beta}{2}}\mathbb{E}\Big|\Big[\sum_{k=0}^{m-1}\Big(\mathrm{Tr}(P^NQD^2\phi_{\tau,\epsilon}(X^N_k))-\mathrm{Tr}(P^N Q^{\frac{1}{2}}D^2\phi_{\tau,\epsilon}(X^N_k)S^N(\tau)Q^{\frac{1}{2}})\Big)\\
&\quad +\sum_{k=0}^{m-1}\mathrm{Tr}((\mathrm{Id}-P^N)QD^2\phi_{\tau,\epsilon}(X^N_k))\Big]\Big|\\
&=\frac{1}{2}\tau^{1+\frac{\beta}{2}}\mathbb{E}\Big|\Big[\sum_{k=0}^{m-1}
\sum_{j=1}^N\Big(\langle D^2\phi_{\tau,\epsilon}(X^N_k)P^NQ^{\frac{1}{2}}e_j,Q^{\frac{1}{2}}e_j\rangle-\langle D^2\phi_{\tau,\epsilon}(X^N_k)S^N(\tau)Q^{\frac{1}{2}}P^Ne_j,Q^{\frac{1}{2}}e_j\rangle\Big)\\
&\quad +\sum_{k=0}^{m-1}\mathrm{Tr}((\mathrm{Id}-P^N)QD^2\phi_{\tau,\epsilon}(X^N_k))\Big]\Big|\\
&\leq C\tau^{1+\frac{\beta}{2}}\sum_{k=0}^{m-1}\sum_{j=1}^N \tau^{\beta-\gamma}\mathbb{E}\|(-A^N)^{\frac{1+\beta}{2}-\gamma}D^2\phi_{\tau,\epsilon}(X^N_k)(-A^N)^{\frac{1-\beta}{2}}\|_{\mathcal{L}(H)}\|(-A)^{\frac{\beta-1}{2}}Q^{\frac{1}{2}}e_j\|^2\\
&\quad+C\tau^{1+\frac{\beta}{2}}\sum_{k=0}^{m-1}\sum_{j=1}^{\infty}\mathbb{E}\|(-A)^{1-\gamma}D^2\phi_{\tau,\epsilon}(X^N_k)\|_{\mathcal{L}(H)}\|(\mathrm{Id}-P^N)(-A)^{-\beta+\gamma}\|_{\mathcal{L}(H)}\|(-A)^{\frac{\beta-1}{2}}Q^{\frac{1}{2}}e_j\|^2\\
&\leq C(\tau^{\frac{\beta}{2}-\gamma}+\tau^{-\frac{\beta}{2}}N^{-2(\beta-\gamma)})(\sup_{k\in\mathbb{N}}\mathbb{E}\|X^N_k\|^2+1),
\end{align*}
which converges to 0 when we assume that $\alpha>\frac{1}{4}$ and take $0<\gamma\leq\min\big\{\frac{\beta}{2},\frac{\beta}{2\alpha}(\alpha-\frac{1}{4})\big\} $.
Similarly, we can prove that $\mathbb{E}|\mathcal{R}^{N,1,2}_{\tau,3}|\leq C\tau^{\frac{\beta}{2}-\gamma}.$ Applying the Ch\'ebyshev inequality yileds that for each fixed $a>0,$ $\mathbb{P}(|\mathcal{R}^{N,1,i}_{\tau,3}|>\frac{a}{4})\to 0$ as $\tau\to0$ for $i=1,2.$

It remains to prove that for each fixed $a>0,$ $\mathbb{P}(\mathcal{R}^{N,1,i}_{\tau,3}\ge \frac{a}{4})\rightarrow 0,\;\mathbb{P}(-\mathcal{R}^{N,1,i}_{\tau,3}\ge \frac{a}{4})\rightarrow 0$ as $\tau\rightarrow 0$ for $i=3,4.$
By Ch{\'e}byshev's inequality $\mathbb{P}(X>b)\leq e^{-b}\mathbb{E}[e^{X}]$ for $b>0$ and a random variable $X,$ we have 
\begin{align*}
&\mathbb{P}\Big(\frac{1}{2}\tau^{1+\frac{\beta}{2}-\gamma}\sum_{k=0}^{m-1}\sum_{j=1}^N\big\langle Q^{\frac{1}{2}}S^N(\tau)D^2\phi_{\tau,\epsilon}(X^N_k)S^N(\tau)Q^{\frac{1}{2}}e_j,e_j\big\rangle\big(1-(\frac{\Delta_k\beta_j}{\sqrt{\tau}})^2\big)\ge\frac{a}{4}\tau^{-\gamma}\Big)\\
\leq&~ e^{-4^{-1}a\tau^{-\gamma}}     \mathbb{E}\Big[\exp\Big\{\frac{1}{2}\tau^{1+\frac{\beta}{2}-\gamma}\sum_{k=0}^{m-1}\sum_{j=1}^N\big\langle Q^{\frac{1}{2}}S^N(\tau)D^2\phi_{\tau,\epsilon}(X^N_k)S^N(\tau)Q^{\frac{1}{2}}e_j,e_j\big\rangle\big(1-(\frac{\Delta_k\beta_j}{\sqrt{\tau}})^2\big)\Big\}\Big].
\end{align*}
Let $$\Theta(X^N_k,\Delta_{k}\beta_1,\ldots,\Delta_{k}\beta_N)=\frac{1}{2}\tau^{1+\frac{\beta}{2}-\gamma}\sum_{j=1}^N\big\langle Q^{\frac{1}{2}}S^N(\tau)D^2\phi_{\tau,\epsilon}(X^N_{k})S^N(\tau)Q^{\frac{1}{2}}e_j,e_j\big\rangle\big(1-(\frac{\Delta_{k}\beta_j}{\sqrt{\tau}})^2\big).$$
 On one hand, we can see that $\mathbb{E}[\Theta(X^N_k,\Delta_{k}\beta_1,\ldots,\Delta_{k}\beta_N)\big|\mathcal{F}_{t_k}]=0.$
On the other hand, by using the assumption that $A$ commutes with $Q$, we have for $k=0,1,\ldots,m-1,$
\begin{align*}
&\mathcal{J}_1:=\mathbb{E}\Big[\sum_{n=2}^{\infty}\frac{1}{n!}|\Theta(X^N_{k},\Delta_{k}\beta_1,\ldots,\Delta_{k}\beta_N)|^n\Big|\mathcal{F}_{t_{k}}\Big]\\
=&~\mathbb{E}\Big[\sum_{n=2}^{\infty}\frac{1}{n!}\Big|\frac{1}{2}\tau^{1+\frac{\beta}{2}-\gamma}\sum_{j=1}^N\big\langle Q^{\frac{1}{2}}S^N(\tau)D^2\phi_{\tau,\epsilon}(X^N_{k})S^N(\tau)Q^{\frac{1}{2}}e_j,e_j\big\rangle\big(1-(\frac{\Delta_{k}\beta_j}{\sqrt{\tau}})^2\big)\Big|^n\Big|\mathcal{F}_{t_{k}}\Big]\\
\leq &~\mathbb{E}\Big[\sum_{n=2}^{\infty}\frac{1}{n!}\Big(\frac{1}{2}\tau^{1+\frac{\beta}{2}-\gamma}\sum_{j=1}^N\|(-A)^{-\frac{\beta-\gamma}{2}}(-A)^{\frac{\beta-1}{2}}Q^{\frac{1}{2}}e_j\|^2\|(-A)^{\frac{1-\gamma}{2}}D^2\phi_{\tau,\epsilon}(X^N_{k})(-A)^{\frac{1-\gamma}{2}}\|_{\mathcal{L}(H)}\times\\
&\quad\big(1+(\frac{\Delta_{k}\beta_j}{\sqrt{\tau}})^2\big)\Big)^n\Big|\mathcal{F}_{t_{k}}\Big]\\
\leq &~\mathbb{E}\Big[\sum_{n=2}^{\infty}\frac{1}{n!}\Big(C\tau^{1+\frac{\beta}{2}-\gamma}\sum_{j=1}^Nj^{-2(\beta-\gamma)}\|(-A)^{\frac{\beta-1}{2}}Q^{\frac{1}{2}}e_j\|^2\tau^{-\frac{\beta}{2}+\epsilon}\big(1+(\frac{\Delta_{k}\beta_j}{\sqrt{\tau}})^2\big)\Big)^n\Big|\mathcal{F}_{t_{k}}\Big].
\end{align*}

When $\beta>\frac{1}{2},$ by the fact that $\mathcal{N}(0,1)(|\cdot|^{2n})=(2n-1)!!$ for standard normal distribution $\mathcal{N}(0,1),$ we get that 
\begin{align*}
\mathcal{J}_1\leq
&~\mathbb{E}\Big[\sum_{n=2}^{\infty}\frac{1}{n!}\Big(C\|(-A)^{\frac{\beta-1}{2}}Q^{\frac{1}{2}}\|^2_{\mathcal{L}_2(H)}\tau^{1+\epsilon-\gamma} \big(1+\sup_{1\leq j\leq N}(\frac{\Delta_{k}\beta_j}{\sqrt{\tau}})^2\big)\Big)^n\Big]\\
\leq &~\sum_{n=2}^{\infty}\frac{1}{n!}(C\|(-A)^{\frac{\beta-1}{2}}Q^{\frac{1}{2}}\|^2_{\mathcal{L}_2(H)}\tau^{1+\epsilon-\gamma})^n2^{n-1}\Big(1+\mathbb{E}\Big[\sup_{1\leq j\leq N}(\frac{\Delta_{k}\beta_j}{\sqrt{\tau}})^{2n}\Big]\Big)\\
\leq&~ \sum_{n=2}^{\infty}\frac{1}{n!}(C\|(-A)^{\frac{\beta-1}{2}}Q^{\frac{1}{2}}\|^2_{\mathcal{L}_2(H)}\tau^{1+\epsilon-\gamma})^n2^{n-1}(1+N\times(2n-1)!!)\\
\leq&~ \sum_{n=2}^{\infty}\frac{(2n)!!}{n!}(2C\|(-A)^{\frac{\beta-1}{2}}Q^{\frac{1}{2}}\|^2_{\mathcal{L}_2(H)}\tau^{1+\epsilon-\gamma})^nN\\
\leq&~ \sum_{n=2}^{\infty}(4C\|(-A)^{\frac{\beta-1}{2}}Q^{\frac{1}{2}}\|^2_{\mathcal{L}_2(H)}\tau^{1+\epsilon-\gamma} N^{\frac{1}{n}})^n.
\end{align*}
Since $\epsilon>\frac{\alpha-1+\beta}{2}$, we have $\epsilon\ge\frac{\alpha-1+\beta}{2}+\gamma$ for $\gamma>0$ sufficiently small, i.e., $1+\epsilon-\gamma-\frac{1}{2}\alpha\ge \frac{1+\beta}{2}.$  Taking $\tau$ small enough 
such that 
$4C\|(-A)^{\frac{\beta-1}{2}}Q^{\frac{1}{2}}\|^2_{\mathcal{L}_2(H)}\tau^{\frac{1+\beta}{2}}< \frac{1}{2},$ 
we obtain $$\mathcal{J}_1\leq \sum_{n=2}^{\infty}\big[4C\|(-A)^{\frac{\beta-1}{2}}Q^{\frac{1}{2}}\|^2_{\mathcal{L}_2(H)}\tau^{\frac{1+\beta}{2}}(\tau^{\alpha}N)^{\frac{1}{2}}\big]^n\leq C\tau^{1+\beta}=\frac{C}{m}.$$

When $0<\beta\leq \frac{1}{2},$ applying H\"{o}lder's inequality gives that 
\begin{align*}
\mathcal{J}_1\leq
&~\mathbb{E}\Big[\sum_{n=2}^{\infty}\frac{1}{n!}\big[C\tau^{1+\epsilon-\gamma}\big(\sum_{j=1}^Nj^{-2(\beta-\gamma)\times \frac{1}{2\beta-3\gamma}}\big)^{2\beta-3\gamma}\big(\sum_{j=1}^{N}(1+(\frac{\Delta_{k}\beta_j}{\sqrt{\tau}})^2)^{ \frac{1}{1-2\beta+3\gamma}}\big)^{1-2\beta+3\gamma}\big]^n\Big]\\
\leq&~ \sum_{n=2}^{\infty}\frac{1}{n!}(C\tau^{1+\epsilon-\gamma})^nN^{(1-2\beta+3\gamma)n-1}\sum_{j=1}^N\mathbb{E}\big(1+(\frac{\Delta _{k}\beta_j}{\sqrt{\tau}})^2\big)^{n}\\
\leq&~ \sum_{n=2}^{\infty}(C\tau^{1+\epsilon-\gamma}N^{1-2\beta+3\gamma})^n.
\end{align*}
Since $\epsilon>0$ and $\alpha<\frac{1}{2}$, we have
$\epsilon\ge\frac{5\gamma}{2}> \frac{2\alpha(1-\beta+3\gamma)-1+\beta+2\gamma}{2}$ for $\gamma>0$ sufficiently small, i.e.,
$1+\epsilon-\gamma-\alpha(1-2\beta+3\gamma)\ge \frac{1+\beta}{2}.$ 
Taking $\tau$ small enough such that $C\tau^{\frac{1+\beta}{2}}\leq\frac{1}{2},$
we get $\mathcal{J}_1\leq C\tau^{1+\beta}=\frac{C}{m}$ for $\beta\leq\frac{1}{2}.$

Applying Lemma \ref{lemmaT}, we get $\mathbb{P}(\mathcal{R}^{N,1,3}_{\tau,3}\ge \frac{a}{4})\leq Ce^{-4^{-1}a\tau^{-\gamma}}\rightarrow 0$ as $\tau\rightarrow0.$
By the similar procedure, we can prove that 
$\mathbb{P}(-\mathcal{R}^{N,1,3}_{\tau,3}\ge \frac{a}{4})\leq Ce^{-4^{-1}a\tau^{-\gamma}}\rightarrow0$ as $\tau\rightarrow 0,$ for each fixed $a>0.$ These imply that $\mathbb{P}(|\mathcal{R}^{N,1,3}_{\tau,3}|\ge \frac{a}{4})\rightarrow 0$ as $\tau\rightarrow0.$

Similarly, for the term $\mathcal{R}^{N,1,4}_{\tau,3},$ note that for $k=0,1,\ldots,m-1,$
$$\mathbb{E}\Big[\sum_{i\neq j}\big\langle Q^{\frac{1}{2}}S^N(\tau)D^2\phi_{\tau,\epsilon}(X^N_k)S^N(\tau)Q^{\frac{1}{2}}e_i,e_j\big\rangle\Delta_k\beta_i\Delta_k\beta_j\Big|\mathcal{F}_{t_k}\Big]=0$$ and
\begin{align*}
\mathcal{J}_2:=
&~\mathbb{E}\Big[\sum_{n=2}^{\infty}\frac{1}{n!}\Big(-\frac{1}{2}\tau^{1+\frac{\beta}{2}-\gamma}\sum_{i\neq j}\big\langle Q^{\frac{1}{2}}S^N(\tau)D^2\phi_{\tau,\epsilon}(X^N_{k})S^N(\tau)Q^{\frac{1}{2}}e_i,e_j\big\rangle\frac{\Delta_{k}\beta_i}{\sqrt{\tau}}\frac{\Delta_{k}\beta_j}{\sqrt{\tau}}\Big)^n\Big|\mathcal{F}_{t_{k}}\Big]\\
\leq&~ \mathbb{E}\Big[\sum_{n=2}^{\infty}\frac{1}{n!}\Big(C\tau^{1+\epsilon-\gamma}\sum_{i\neq j}\|(-A)^{-\frac{\beta-\gamma}{2}}(-A)^{\frac{\beta-1}{2}}Q^{\frac{1}{2}}e_i\|\|(-A)^{-\frac{\beta-\gamma}{2}}(-A)^{\frac{\beta-1}{2}}Q^{\frac{1}{2}}e_j\|\times\\&\qquad \Big|\frac{\Delta_{k}\beta_i}{\sqrt{\tau}}\frac{\Delta_{k}\beta_j}{\sqrt{\tau}}\Big|\Big)^n\Big|\mathcal{F}_{t_{k}}\Big]\\
\leq&~ \mathbb{E}\Big[\sum_{n=2}^{\infty}\frac{1}{n!}\Big(C\tau^{1+\epsilon-\gamma}\Big(\sum_{i=1}^Ni^{-(\beta-\gamma)}\|(-A)^{\frac{\beta-1}{2}}Q^{\frac{1}{2}}e_i\|
\big|\frac{\Delta_{k}\beta_i}{\sqrt{\tau}}\big|\Big)^2\Big)^n\Big|\mathcal{F}_{t_{k}}\Big]\\
\leq&~ \mathbb{E}\Big[\sum_{n=2}^{\infty}\frac{1}{n!}\big(C\tau^{1+\epsilon-\gamma}\|(-A)^{\frac{\beta-1}{2}}Q^{\frac{1}{2}}\|^2_{\mathcal{L}_2(H)}\sum_{i=1}^Ni^{-2(\beta-\gamma)}(\frac{\Delta_{k}\beta_i}{\sqrt{\tau}})^2\big)^n\Big|\mathcal{F}_{t_{k}}\Big],
\end{align*}
where we use H\"{o}lder's inequality in the last step.
When $\beta>\frac{1}{2},$ we have
\begin{align*}
\mathcal{J}_2\leq&~\mathbb{E}\Big[\sum_{n=2}^{\infty}\frac{1}{n!}\Big(C\tau^{1+\epsilon-\gamma}\sum_{i=1}^Ni^{-2(\beta-\gamma)}\sup_{1\leq i\leq N}\big(\frac{\Delta_{k}\beta_i}{\sqrt{\tau}}\big)^2\Big)^n\Big]
\leq \sum_{n=2}^{\infty}(C\tau^{1+\epsilon-\gamma}N^{\frac{1}{n}})^n.
\end{align*}
When $\beta<\frac{1}{2},$ we have 
\begin{align*}
\mathcal{J}_2\leq &~\mathbb{E}\Big[\sum_{n=2}^{\infty}\frac{1}{n!}\Big[C\tau^{1+\epsilon-\gamma}\big(\sum_{i=1}^Ni^{-2(\beta-\gamma)\times \frac{1}{2\beta-3\gamma}}\big)^{2\beta-3\gamma}\big(\sum_{i=1}^{N}(\frac{\Delta_{k}\beta_i}{\sqrt{\tau}})^{2\times \frac{1}{1-2\beta+3\gamma}}\big)^{1-2\beta+3\gamma}\Big]^n\Big]\\
\leq&~ \sum_{n=2}^{\infty}(C\tau^{1+\epsilon-\gamma}N^{1-2\beta+3\gamma})^n.
\end{align*}
The remaining proof is similar to that of $\mathcal{J}_1$. We omit it.
Hence, we get $\mathcal{R}^{N,1}_{\tau,3}\overset{\mathbb{P}}{\longrightarrow}0\text{ as }\tau\rightarrow0.$

\textit{Estimate of $\mathcal{R}^{N,2}_{\tau, 3}.$}
The term $\mathcal{R}^{N,2}_{\tau, 3}$ can be divided into four terms, we only present the estimate of the low order term, i.e.,
\begin{align*}
&\mathbb{E}\Big|\sum_{k=0}^{m-1}\tau^{\frac{\beta}{2}}\big\langle (S^N(\tau)-\mathrm{Id})X^N_k,D^2\phi_{\tau,\epsilon}(X^N_k)(S^N(\tau)-\mathrm{Id})X^N_k\big\rangle\Big|\\
\leq&~ C\tau^{\frac{\beta}{2}}\sum_{k=0}^{m-1}\tau^{1-\gamma}\tau^{\beta}\mathbb{E}\Big[\|(-A)^{1-\gamma}D^2\phi_{\tau,\epsilon}(X^N_k)\|_{\mathcal{L}(H)}\|(-A)^{\frac{\beta}{2}}X^N_k\|^2\Big]\\
\leq&~ C\tau^{\frac{\beta}{2}-\gamma}\sup_{k\in\mathbb{N}}\mathbb{E}\big(\|X^N_k\|^4+\|X^N_k\|^4_{\beta}+1\big)
\end{align*}
goes to 0 as $\tau\rightarrow0$ for $\gamma<\frac{\beta}{2}.$ The other three terms are high order terms that can be proved similarly, we omit the proofs.

\textit{Estimate of $\mathcal{R}^{N,3}_{\tau,3}$.} By Proposition \ref{propfull} and \eqref{phi2}, 
\begin{align*}
\mathbb{E}|\mathcal{R}^{N,3}_{\tau,3}|=&~\mathbb{E}\Big|\sum_{k=0}^{m-1}\tau^{\frac{\beta}{2}}\big\langle D^2\phi_{\tau,\epsilon}(X^N_k)(S^N(\tau)-\mathrm{Id})X^N_k,S^N(\tau)P^N\Delta W_k\big\rangle\Big|\\
\leq&~ \tau^{\frac{1+3\beta}{2}-\gamma}\sum_{k=0}^{m-1}\mathbb{E}\Big[\|(-A)^{\frac{1-\beta}{2}}D^2\phi_{\tau,\epsilon}(X^N_k)(-A)^{\frac{1+\beta}{2}-\gamma}\|_{\mathcal{L}(H)}\|X^N_k\|_{\beta}\|S^N(\tau)(-A)^{\frac{\beta-1}{2}}\Delta W_k\|\Big]\\
\leq&~ \tau^{\frac{1+3\beta}{2}-\gamma}\sum_{k=0}^{m-1}\Big(\mathbb{E}\Big[\|(-A)^{\frac{1-\beta}{2}}D^2\phi_{\tau,\epsilon}(X^N_k)(-A)^{\frac{1+\beta}{2}-\gamma}\|^2_{\mathcal{L}(H)}\|X^N_k\|^2_{\beta}\Big]\Big)^{\frac{1}{2}}\times\\
&\qquad \qquad \qquad \Big(\mathbb{E}\Big\|\int_{t_k}^{t_{k+1}}S^N(\tau)(-A)^{\frac{\beta-1}{2}}\mathrm{d}W(s)\Big\|^2\Big)^{\frac{1}{2}}\\
\leq&~ C\tau^{\frac{\beta}{2}-\gamma}\Big(\sup_{k\in\mathbb{N}}\mathbb{E}\big[(1+\|X^N_k\|^4)\|X^N_k\|^2_{\beta}\big]\Big)^{\frac{1}{2}}\to0\text{ as }\tau\to0.
\end{align*}
Applying Ch\'ebyshev's inequality leads to that $\mathcal{R}^{N,3}_{\tau,3}\overset{\mathbb{P}}{\longrightarrow}0$ as $\tau\to0$.

\textit{Estimate of $\mathcal{R}^{N,4}_{\tau,3}.$}
Denote \begin{align*}
\tilde{J}_k:=&~2\mu\tau^{1+\frac{\beta}{2}-\epsilon}\Big\langle D^2\phi_{\tau,\epsilon}(X^N_k)S^N(\tau)F^N(X^N_k),S^N(\tau)P^N\frac{\Delta W_k}{\sqrt{\tau}}\Big\rangle\\
&\qquad \quad-2\mu^2\tau^{2+\beta-2\epsilon}\sum_{j=1}^N\big\langle D^2\phi_{\tau,\epsilon}(X^N_k)S^N(\tau)F^N(X^N_k),S^N(\tau)Q^{\frac{1}{2}}e_j\big\rangle^2,
\end{align*}
where $\mu$ is a parameter to be determined.
Using the property of the exponential martingale, we know that
$
\mathbb{E}[\exp\{\tilde{J}_k\}|\mathcal{F}_{t_k}]=1.$
Hence, by iterating, we get
\begin{align*}
\mathbb{E}\Big[\exp\Big\{\sum_{k=0}^{m-1}\tilde{J}_k\Big\}\Big]=\mathbb{E}\Big[\exp\Big\{\sum_{k=0}^{m-2}\tilde{J}_k\Big\}\mathbb{E}\big[\exp\{\tilde{J}_{m-1}\}\big|\mathcal{F}_{m-1}\big]\Big]=1.
\end{align*}
Then H{\"o}lder's inequality yields that
\begin{align*}
&~\mathbb{E}\Big[\exp\Big\{\sum_{k=0}^{m-1}\mu\tau^{1+\frac{\beta}{2}-\epsilon}\Big\langle D^2\phi_{\tau,\epsilon}(X^N_k)S^N(\tau)F^N(X^N_k),S^N(\tau)P^N\frac{\Delta W_k}{\sqrt{\tau}}\Big\rangle\Big\}\Big]\\
\leq&~ \Big(\mathbb{E}\Big[\exp\Big\{\sum_{k=0}^{m-1}\tilde{J}_k\Big\}\Big]\Big)^{\frac{1}{2}}
\Big(\mathbb{E}\exp\Big\{\sum_{k=0}^{m-1}2\mu^2\tau^{2+\beta-2\epsilon}\sum_{j=1}^N\langle D^2\phi_{\tau,\epsilon}(X^N_k)S^N(\tau)F^N(X^N_k),S^N(\tau)Q^{\frac{1}{2}}e_j\rangle^2\Big\}\Big)^{\frac{1}{2}}.
\end{align*}
 These, together with Ch{\'e}byshev's inequality lead to
 \begin{align*}
&\mathbb{P}(-\mathcal{R}^{N,4}_{\tau,3}>a)\\
\leq&~ e^{-a\mu\tau^{-\frac{1}{2}-\epsilon}}\mathbb{E}\exp\Big\{\sum_{k=0}^{m-1}\mu\tau^{1+\frac{\beta}{2}-\epsilon}\Big\langle D^2\phi_{\tau,\epsilon}(X^N_k)S^N(\tau)F^N(X^N_k),S^N(\tau)P^N\frac{\Delta W_k}{\sqrt{\tau}}\Big\rangle\Big\}\\
\leq&~ e^{-a\mu\tau^{-\frac{1}{2}-\epsilon}}\Big(\mathbb{E}\exp\Big\{\sum_{k=0}^{m-1}2\mu^2\tau^{2+\beta-2\epsilon}\sum_{j=1}^{N}\big\langle D^2\phi_{\tau,\epsilon}(X^N_k)S^N(\tau)F^N(X^N_k),S^N(\tau)Q^{\frac{1}{2}}e_j\big\rangle^2\Big\}\Big)^{\frac{1}{2}}.
\end{align*}
By using the fact that $\|(-A)^{\frac{1-\beta}{2}}D^2\phi_{\tau,\epsilon}(x)\|_{\mathcal{L}(H)}\leq C\tau^{-\frac{\beta}{2}+\epsilon}$ and the Lipschitz continuity of $F$, we have
\begin{align}\label{tildeC}
&\mathbb{P}(-\mathcal{R}^{N,4}_{\tau,3}>a)\notag\\
\leq&~ e^{-a\mu\tau^{-\frac{1}{2}-\epsilon}}\Big(\mathbb{E}\Big[\exp\Big\{\sum_{k=0}^{m-1}C\mu^2\tau^{2+\beta-2\epsilon}\tau^{-\beta+2\epsilon}\|F^N(X^N_k)\|^2\|(-A)^{\frac{\beta-1}{2}}Q^{\frac{1}{2}}\|^2_{\mathcal{L}_2(H)}\Big\}\Big]\Big)^{\frac{1}{2}}\notag\\
\leq&~ e^{-a\mu\tau^{-\frac{1}{2}-\epsilon}}\Big(\mathbb{E}\Big[\exp\Big\{\sum_{k=0}^{m-1}C\mu^2\tau^2\|F^N(X^N_k)\|^2\Big\}\Big]\Big)^{\frac{1}{2}}\notag\\
\leq&~ e^{C\|F(0)\|^2}\times e^{-a\mu\tau^{-\frac{1}{2}-\epsilon}}\Big(\mathbb{E}\Big[\exp\Big\{\sum_{k=0}^{m-1}\tilde{C}\mu^2\tau^2\|X^N_k\|^2\Big\}\Big]\Big)^{\frac{1}{2}}.
\end{align}

We claim that $\mathbb{E}\Big[\exp\Big\{\sum_{k=0}^{m-1}\tilde{C}\mu^2\tau^2\|X^N_k\|^2\Big\}\Big]\leq C.$
In fact, denote $$V^N_k:=\int_0^{t_k}S^N(t_k-\lfloor s\rfloor_{\tau})P^N\mathrm{d}W(s),\quad Z^N_k:=X^N_k-V^N_k.$$
Then by \cite[Lemma 4.1]{czh20}, we have 
\begin{align*}
\|Z^N_{k+1}\|^2\leq &~e^{-(\lambda_1-K)\tau}\|Z^N_k\|^2+2(\tau^2+\frac{\tau}{\lambda_1-K})e^{-2\lambda_1\tau}\|F^N(V^N_k)\|^2\\
\leq &~e^{-(\lambda_1-K)\tau}\|Z^N_k\|^2+C\tau(\|V^N_k\|^2+1).
\end{align*}
By Young's inequality $bc\leq \frac{b^p}{p}+\frac{c^q}{q}$ with $p=e^{(\lambda_1-K)\tau}$, the Fernique Theorem (see e.g. \cite[Theorem $2.7$]{Daprato}) and Assumption \ref{assumpX_0}, we obtain that
\begin{align*}
&\quad\mathbb{E}\Big[\exp\{\mu\|Z^N_{k+1}\|^2\}\Big]\\
&\leq e^{-(\lambda_1-K)\tau}\mathbb{E}\Big[\exp\{\mu\|Z^N_k\|^2\}\Big]+(1-e^{-(\lambda_1-K)\tau})\mathbb{E}\Big[\exp\Big\{\frac{C\mu\tau(\|V^N_k\|^2+1)}{1-e^{-(\lambda_1-K)\tau}}\Big\}\Big]\\
&\leq e^{-(\lambda_1-K)\tau (k+1)}\mathbb{E}\Big[\exp\{\mu \|Z^N_0\|^2\}\Big]\\
&\quad+\sum_{j=0}^k e^{-(\lambda_1-K)\tau (k-j)}(1-e^{-(\lambda_1-K)\tau})\mathbb{E}\Big[\exp\Big\{\frac{C\mu\tau(\|V^N_j\|^2+1)}{1-e^{-(\lambda_1-K)\tau}}\Big\}\Big]\\
&\leq \mathbb{E}\Big[\exp\{\mu \|Z^N_0\|^2\}\Big]+C\sum_{j=0}^k e^{-(\lambda_1-K)\tau j}(1-e^{-(\lambda_1-K)\tau})\leq C
\end{align*}
and
\begin{align}\label{estexp}
\mathbb{E}\big[e^{\frac{\mu}{4}\|X^N_{k+1}\|^2}\big]\leq \mathbb{E}\big[e^{\frac{\mu}{2}\|Z^N_{k+1}\|^2+\frac{\mu}{2}\|V^N_{k+1}\|^2}\big]\leq \Big(\mathbb{E}\big[e^{\mu\|Z^N_{k+1}\|^2}\big]\mathbb{E}\big[e^{\mu\|V^N_{k+1}\|^2}\big]\Big)^{\frac{1}{2}}\leq C
\end{align}
hold for all $\mu\leq\mu_1$ independent of $\tau$ with some $\mu_1>0.$
Taking $\mu\leq\min\{\mu_1,(4\tilde{C})^{-1}\}$ with $\tilde{C}$ being the same as that in \eqref{tildeC}, and using \eqref{estexp} and H{\"o}lder's inequality, we get
\begin{align*}
\mathbb{E}\Big[\exp\Big\{\sum_{k=0}^{m-1}\tilde{C}\mu^2\tau^2\|X^N_k\|^2\Big\}\Big]\leq \Pi_{k=0}^{m-1}\Big(\mathbb{E}\Big[\exp\{m\tilde{C}\mu^2\tau^2\|X^N_k\|^2\}\Big]\Big)^{\frac{1}{m}}\leq \Pi_{k=0}^{m-1}C^{\frac{1}{m}}\leq C.
\end{align*} 
The proof of the claim is completed.

By the same argument, we obtain the same bound for $\mathbb{P}(\mathcal{R}^{N,4}_{\tau,3}>a)$. Thus, 
\begin{align*}
\mathbb{P}(|\mathcal{R}^{N,4}_{\tau, 3}|>a)\leq Ce^{-a\mu\tau^{-\frac{1}{2}-\epsilon}}\rightarrow 0\text{ as }\tau\rightarrow 0.
\end{align*}

\textbf{Estimate of $\mathcal{R}^N_{\tau, 4}$.}
Noticing that
\begin{align*}
&\mathcal{R}^N_{\tau,4}=-\tau^{\frac{\beta}{2}}\sum_{k=0}^{m-1}\mathcal{R}^N_{k,\phi_{\tau,\epsilon}}\\
=&-\tau^{\frac{\beta}{2}}\sum_{k=0}^{m-1}\frac{1}{2}\int_0^1(1-t)^2D^3\phi_{\tau,\epsilon}\big(X^N_k+t(X^N_{k+1}-X^N_k)\big)\big(X^N_{k+1}-X^N_k,X^N_{k+1}-X^N_k,X^N_{k+1}-X^N_k\big)\mathrm{d}t
\end{align*}
and 
\begin{align*}
X^N_{k+1}-X^N_k=(S^N(\tau)-\mathrm{Id})X^N_k+\tau S^N(\tau)F^N(X^N_k)+S^N(\tau)P^N\Delta W_k,
\end{align*}
it is natural to split $\mathcal{R}^N_{\tau,4}$ into two terms denoted by $\mathcal{R}^{N,1}_{\tau,4},\mathcal{R}^{N,2}_{\tau,4}$ with
\begin{align*}
\mathcal{R}^{N,1}_{\tau,4}:=&~-\tau^{\frac{\beta}{2}}\sum_{k=0}^{m-1}\frac{1}{2}\int_0^1(1-t)^2D^3\phi_{\tau,\epsilon}\big(X^N_k+t(X^N_{k+1}-X^N_k)\big)\big(S^N(\tau)P^N\Delta W_k,\\
&S^N(\tau)P^N\Delta W_k,S^N(\tau)P^N\Delta W_k\big)\mathrm{d}t
\end{align*}
and $\mathcal{R}^{N,2}_{\tau,4}:=\mathcal{R}^{N}_{\tau,4}-\mathcal{R}^{N,1}_{\tau,4}.$

The estimate of the term $\mathcal{R}^{N,2}_{\tau,4}$ is similar to that of $\mathcal{R}^{N,3}_{\tau,3},$ and we take the estimate of 
\begin{align*}
\mathcal{G}_1:=&~
\sum_{k=0}^{m-1}\frac{\tau^{\frac{\beta}{2}}}{2}\int_0^1(1-t)^2D^3\phi_{\tau,\epsilon}\big(X^N_k+t(X^N_{k+1}-X^N_k)\big)\\
&\qquad \qquad \qquad \big((S^N(\tau)-\mathrm{Id})X^N_k,S^N(\tau)\Delta W_k,S^N(\tau)\Delta W_k\big)\mathrm{d}t
\end{align*}
as an example since the others are high order terms. Precisely, applying \eqref{phi3} with $\delta_1=\beta-\gamma, \,\delta_2=\delta_3=\frac{1-\beta}{2}$ and \cite[Theorem $1.1.9$]{ZGQ} yields that
\begin{align*}
\mathbb{E}|\mathcal{G}_1|
=&~ \frac{\tau^{1+\frac{\beta}{2}}}{2}\mathbb{E}\Big|\sum_{k=0}^{m-1}\int_0^1(1-t)^2D^3\phi_{\tau,\epsilon}\big(X^N_k+t(X^N_{k+1}-X^N_k)\big)\Big((-A)^{\frac{1-\beta}{2}}S^N(\tau)(-A)^{\frac{\beta-1}{2}}\frac{\Delta W_k}{\sqrt{\tau}},\\
 &\qquad \quad (-A)^{\frac{1-\beta}{2}}S^N(\tau)(-A)^{\frac{\beta-1}{2}}\frac{\Delta W_k}{\sqrt{\tau}},(-A)^{\beta-\gamma}(-A)^{-\frac{3\beta}{2}+\gamma}(S^N(\tau)-\mathrm{Id})(-A)^{\frac{\beta}{2}}X^N_k\Big)\mathrm{d}t\Big|\\
\leq&~ C\tau^{1+2\beta-\gamma}\sum_{k=0}^{m-1}\mathbb{E}\Big[(1+\|X^N_k\|^2+\|X^N_{k+1}\|^2)\|X^N_k\|_{\beta}\|S^N(\tau)(-A)^{\frac{\beta-1}{2}}\frac{\Delta W_k}{\sqrt{\tau}}\|^2\Big]\\
\leq&~ C\tau^{\beta-\gamma}\big(\sup_{k\ge 0}\mathbb{E}(1+\|X^N_k\|^6)\sup_{k\ge 0}\mathbb{E}\|X^N_k\|^3_{\beta}\big)^{\frac{1}{3}}\|(-A)^{\frac{\beta-1}{2}}Q^{\frac{1}{2}}\|^2_{\mathcal{L}_2(H)}\rightarrow 0\text{ as }\tau\rightarrow 0.
\end{align*}
By the similar argument, we can get that for each fixed $a>0,$
$$\mathbb{P}(|\mathcal{R}^{N,2}_{\tau,4}|>\frac{a}{2})\leq \frac{2}{a}\mathbb{E}|\mathcal{R}^{N,2}_{\tau,4}|\rightarrow 0\;\text{as }\tau\rightarrow 0.$$

Now we are in a position to prove $\mathbb{P}(|\mathcal{R}^{N,1}_{\tau,4}|>\frac{a}{2})\rightarrow0$ as $\tau\rightarrow0.$
Let $\mathcal{R}^{N,1}_{\tau,4}=\mathcal{R}^{N,1,1}_{\tau,4}+\mathcal{R}^{N,1,2}_{\tau,4},$ where
\begin{align*}
\mathcal{R}^{N,1,1}_{\tau,4}:=&-\tau^{\frac{\beta}{2}}\sum_{k=0}^{m-1}\frac{1}{2}\int_0^1(1-t)^2\Big[D^3\phi_{\tau,\epsilon}(X^N_k+t(X^N_{k+1}-X^N_k))-D^3\phi_{\tau,\epsilon}(X^N_k)\Big]\big(S^N(\tau)P^N\Delta W_k,\\
&\qquad\qquad\qquad S^N(\tau)P^N\Delta W_k,S^N(\tau)P^N\Delta W_k\big)\mathrm{d}t\\
=&-\tau^{\frac{\beta}{2}}\sum_{k=0}^{m-1}\frac{1}{2}\int_0^1(1-t)^2\int_0^1D^4\phi_{\tau,\epsilon}\big(X^N_k+\bar{t}t(X^N_{k+1}-X^N_k)\big)\mathrm{d}\bar{t}\Big(t(X^N_{k+1}-X^N_k),\\
&\qquad\qquad\qquad S^N(\tau)P^N\Delta W_k,S^N(\tau)P^N\Delta W_k,S^N(\tau)P^N\Delta W_k\Big)\mathrm{d}t,
\end{align*}
and 
\begin{align*}
\mathcal{R}^{N,1,2}_{\tau,4}&:=-\tau^{\frac{\beta}{2}}\sum_{k=0}^{m-1}\frac{1}{2}\int_0^1(1-t)^2D^3\phi_{\tau,\epsilon}(X^N_k)\big(S^N(\tau)P^N\Delta W_k,S^N(\tau)P^N\Delta W_k,S^N(\tau)P^N\Delta W_k\big)\mathrm{d}t\\
&=:-\tau^{\frac{\beta}{2}}\sum_{k=0}^{m-1}\mathcal{R}^{N,1,2}_{\tau,4,k}.
\end{align*}
Note that 
$
\mathbb{P}(\mathcal{R}^{N,1,2}_{\tau,4}>\frac{a}{4})\leq e^{-\frac{a}{4}\tau^{-\epsilon}}\mathbb{E}\Big[\exp\Big\{-\tau^{\frac{\beta}{2}-\epsilon}\sum_{k=0}^{m-1}\mathcal{R}^{N,1,2}_{\tau,4,k}\Big\}\Big]$
and 
$\mathbb{E}[\mathcal{R}^{N,1,2}_{\tau,4,k}\big|\mathcal{F}_{t_k}]=0.$
Utilizing \eqref{phi7} and the Burkholder--Davis--Gundy inequality, we get for $k=0,1,\ldots,m-1,$
\begin{align*}
 &~\sum_{n=2}^{\infty}\frac{1}{n!}\mathbb{E}\Big[|\tau^{\frac{\beta}{2}-\epsilon}\mathcal{R}^{N,1,2}_{\tau,4,k}|^n\Big|\mathcal{F}_{t_{k}}\Big]\\
=&~\sum_{n=2}^{\infty}\frac{1}{n!}\mathbb{E}\Big[\Big|\tau^{\frac{\beta}{2}-\epsilon}\frac{1}{2}\int_0^1(1-t)^2D^3\phi_{\tau,\epsilon}(X^N_{k})
\Big((-A)^{\frac{1-\beta}{2}}S^N(\tau)(-A)^{\frac{\beta-1}{2}}P^N\Delta W_{k},\\
&\quad (-A)^{\frac{1-\beta}{2}}S^N(\tau)(-A)^{\frac{\beta-1}{2}}P^N\Delta W_{k},(-A)^{\frac{1-\beta}{2}}S^N(\tau)(-A)^{\frac{\beta-1}{2}}P^N\Delta W_{k}\Big)\mathrm{d}t\Big|^n\Big|\mathcal{F}_{t_{k}}\Big]\\
\leq&~ \sum_{n=2}^{\infty}\frac{1}{n!}\mathbb{E}\Big[\tau^{\frac{\beta}{2}-\epsilon}C\tau^{-\frac{\beta}{2}+\epsilon}\tau^{\frac{3}{2}}\big\|S^N(\tau)(-A)^{\frac{\beta-1}{2}}P^N\frac{\Delta W_{k}}{\sqrt{\tau}}\big\|^3\Big]^n\\
\leq&~ \sum_{n=2}^{\infty}\frac{1}{n!}(C\tau^{\frac{3}{2}}\|(-A)^{\frac{\beta-1}{2}}Q^{\frac{1}{2}}\|^3_{\mathcal{L}_2(H)})^n,
\end{align*}
Take $\tau< \frac{1}{2}$ so that $C\tau^{\frac{1}{2}}\|(-A)^{\frac{\beta-1}{2}}Q^{\frac{1}{2}}\|^3_{\mathcal{L}_2(H)}\leq 1$, then
\begin{align*}
\sum_{n=2}^{\infty}\frac{1}{n!}\mathbb{E}\Big[|\tau^{\frac{\beta}{2}-\epsilon}\mathcal{R}^{N,1,2}_{\tau,4,k}|^n\Big|\mathcal{F}_{t_{k}}\Big]
\leq \sum_{n=2}^{\infty}\frac{\tau^n}{n!}\leq \frac{\tau^2}{2(1-\tau)}
\leq \tau^2\leq \frac{1}{m}.
\end{align*}
Applying Lemma \ref{lemmaT} leads to $\mathbb{P}(\mathcal{R}^{N,1,2}_{\tau,4}>\frac{a}{4})\leq Ce^{-\frac{a}{4}\tau^{-\epsilon}},\;\epsilon\in (\max\{0,\frac{\alpha-1+\beta}{2}\},\frac{\beta}{2}).$ Similar argument yields the same bound for $\mathbb{P}(-\mathcal{R}^{N,1,2}_{\tau,4}>\frac{a}{4})$.

The term $\mathcal{R}^{N,1,1}_{\tau,4}$ can be divided into three subterms.  Because the other subterms can be proved similarly, we only prove the subterm
\begin{align*}
&\mathcal{G}_2:=\tau^{\frac{\beta}{2}}\sum_{k=0}^{m-1}\int_0^1(1-t)^2\int_0^1D^4\phi_{\tau,\epsilon}\big(X^N_k+\bar{t}t(X^N_{k+1}-X^N_k)\big)\mathrm{d}\bar{t}\\
&\qquad\qquad\qquad\quad\Big(tS^N(\tau)P^N\Delta W_k,S^N(\tau)P^N\Delta W_k,S^N(\tau)P^N\Delta W_k,
S^N(\tau)P^N\Delta W_k\Big)\mathrm{d}t.
\end{align*}

 For $\beta>\frac{1}{2},$
 applying \eqref{phi4} with $\delta_i=\frac{1-\beta}{2},i=1,\ldots,4,$ we obtain
\begin{align*}
\mathbb{E}|\mathcal{G}_2|
\leq&~ C\tau^{\frac{\beta}{2}}\sum_{k=0}^{m-1}\mathbb{E}\Big[(1+\|X^N_k\|^2+\|X^N_{k+1}\|^2)\Big\|\int_{t_k}^{t_{k+1}}S^N(\tau)(-A)^{\frac{\beta-1}{2}}\mathrm{d}W(s)\Big\|^4\Big]
\leq C\tau^{1-\frac{\beta}{2}}
\end{align*}
converges to $0$ as $\tau\to0.$
For $\beta\in(0,\frac{1}{2}],$ applying \eqref{phi4} with $\delta_1=\delta_2=\frac{1-\beta}{2},\,\delta_3=\delta_4=0$, we get the upper bound $C\tau^{1-\frac{\beta}{2}}N^{2(1-\beta)}$ for $\mathbb{E}|\mathcal{G}_2|$. Since $\alpha<\frac{1}{2},$ we have $\tau^{1-\frac{\beta}{2}}N^{2(1-\beta)}\leq 2\tau^{1-\frac{\beta}{2}-2\alpha(1-\beta)}\leq 2\tau^{\frac{\beta}{2}}.$
Hence, $\mathbb{E}|\mathcal{G}_2|\leq C\tau^{\frac{\beta}{2}}\to0\text{ as }\tau\to0$ for $\beta\in(0,\frac{1}{2}].$

\textbf{Estimate of $\mathcal{R}^N_{\tau, 5}$.}
The mild form of $X^{N,\tau}_t$ implies that for $t\in(t_k,t_{k+1}],$ we have
\begin{align*}
X^{N,\tau}_t-X^N_k=(S^N(t-t_k)-\mathrm{Id})X^N_k+(t-t_k)S^N(t-t_k)F^N(X^N_k)+S^N(t-t_k)(W(t)-W(t_k)).
\end{align*}
Therefore, $\mathcal{R}^{N}_{\tau,5}=\sum_{i=1}^3\mathcal{R}^{N,i}_{\tau,5}$
with
\begin{align*}
&\mathcal{R}^{N,1}_{\tau,5}:=\sum_{k=0}^{m-1}\Big\langle \int_{t_k}^{t_{k+1}}A^N(S^N(t-t_k)-\mathrm{Id})X^N_k\mathrm{d}t,D\phi_{\tau,\epsilon}(X^N_k)\Big\rangle\\
&\mathcal{R}^{N,2}_{\tau,5}:=\sum_{k=0}^{m-1}\Big\langle \int_{t_k}^{t_{k+1}}A^N(t-t_k)S^N(t-t_k)F^N(X^N_k)\mathrm{d}t,D\phi_{\tau,\epsilon}(X^N_k)\Big\rangle
\end{align*}
and
\begin{align*}
\mathcal{R}^{N,3}_{\tau,5}:=\sum_{k=0}^{m-1}\Big\langle \int_{t_k}^{t_{k+1}}A^NS^N(t-t_k)(W(t)-W(t_k))\mathrm{d}t,D\phi_{\tau,\epsilon}(X^N_k)\Big\rangle.
\end{align*}

 Taking $\nu=\frac{2\gamma}{1-2\alpha}>0,$ where we use $\alpha<\frac{1}{2},$ we obtain
\begin{align*}
\mathbb{E}|\mathcal{R}^{N,1}_{\tau,5}|=&~\tau^{\frac{\beta}{2}}\mathbb{E}\Big|\sum_{k=0}^{m-1}\Big\langle \int_{t_k}^{t_{k+1}}A^N(S^N(t-t_k)-\mathrm{Id})X^N_k\mathrm{d}t,D\phi_{\tau,\epsilon}(X^N_k)\Big\rangle\Big|\\
\leq&~ \tau^{\frac{\beta}{2}}\mathbb{E}\Big[\sum_{k=0}^{m-1}\int_{t_k}^{t_{k+1}}\|(S^N(t-t_k)-\mathrm{Id})(-A)^{-\frac{\beta}{2}-\nu+\gamma}\|\|X^N_k\|_{\beta+2\nu}\|(-A)^{1-\gamma}D\phi_{\tau,\epsilon}(X^N_k)\|\mathrm{d}t\Big]\\
\leq&~ C\tau^{\nu-\gamma}N^{2\nu}\sup_{k\ge 0}\mathbb{E}\Big[\|X^N_k\|_{\beta}(1+\|X^N_k\|^2)\Big]
\leq C\tau^{\nu-\gamma-2\nu\alpha}=C\tau^{\gamma}\to 0\text{ as }\tau\to0.
\end{align*}
By \eqref{phi1} with $\delta_1=1-\gamma$, we have
\begin{align*}
 \mathbb{E}|\mathcal{R}^{N,2}_{\tau,5}|&\leq \tau^{\frac{\beta}{2}}\mathbb{E}\Big[\sum_{k=0}^{m-1}\int_{t_k}^{t_{k+1}}(t-t_k)\|(-A)^{\gamma}S^N(t-t_k)\|_{\mathcal{L}(H)}\|F^N(X^N_k)\|\mathrm{d}t \|(-A)^{1-\gamma}D\phi_{\tau,\epsilon}(X^N_k)\| \Big]\\
 &\leq C\tau^{\frac{\beta}{2}}\mathbb{E}\Big[\sum_{k=0}^{m-1}\int_{t_k}^{t_{k+1}}(t-t_k)^{1-\gamma}\mathrm{d}t\|F^N(X^N_k)\|\|(-A)^{1-\gamma}D\phi_{\tau,\epsilon}(X^N_k)\|\Big]\\
 &\leq C\tau^{1-\frac{\beta}{2}-\gamma}\sup_{k\ge 0}\mathbb{E}(1+\|X^N_k\|^3)\to 0\text{ as }\tau\to0.
 \end{align*}
As for the term $\mathcal{R}^{N,3}_{\tau,5},$
 \begin{align*}
 \mathbb{P}(\mathcal{R}^{N,3}_{\tau,5}>a)&\leq e^{-a\tau^{-\epsilon}}\mathbb{E}\Big[\exp\Big\{\sum_{k=0}^{m-1}\Big\langle \tau^{\frac{\beta}{2}-\epsilon}\int_{t_k}^{t_{k+1}}A^NS^N(t-t_k)(W(t)-W(t_k))\mathrm{d}t,
 D\phi_{\tau,\epsilon}(X^N_k)\Big\rangle\Big\}\Big]\\
 &=:e^{-a\tau^{-\epsilon}}\mathbb{E}\Big[\exp\Big\{\sum_{k=0}^{m-1}\mathcal{R}^{N,3}_{\tau,5,k}\Big\}\Big].
  \end{align*}
Note $\mathbb{E}[\mathcal{R}^{N,3}_{\tau,5,k}\big|\mathcal{F}_{t_k}]=0$.
Applying H\"{o}lder's inequality, the Burkholder--Davis--Gundy inequality and \eqref{phi1}, we get for $k=0,1,\ldots,m-1,$
 \begin{align*}
&\sum_{n=2}^{\infty}\frac{1}{n!}\mathbb{E}\Big[|\mathcal{R}^{N,3}_{\tau,5,k}|^n\Big|\mathcal{F}_{t_{k}}\Big]\\
\leq&~ \sum_{n=2}^{\infty}\frac{1}{n!}\mathbb{E}\Big[\Big(\tau^{\frac{\beta}{2}-\epsilon}\Big\|\int_{t_{k}}^{t_{k+1}}(-A)^{\gamma}S^N(t-t_{k})(W(t)-W(t_{k}))\mathrm{d}t\Big\|
\|(-A)^{1-\gamma}D\phi_{\tau,\epsilon}(X^N_{k})\|\Big)^n\Big|\mathcal{F}_{t_{k}}\Big]\\
\leq&~ \sum_{n=2}^{\infty}\frac{1}{n!}\mathbb{E}\Big[C^n\Big\|\int_{t_{k}}^{t_{k+1}}\int_{t_{k}}^t(-A)^{\frac{1-\beta}{2}+\gamma}S^N(t-t_{k})(-A)^{\frac{\beta-1}{2}}\mathrm{d}W(s)\mathrm{d}t\Big\|^n\Big]\\
\leq&~ \sum_{n=2}^{\infty}\frac{1}{n!}C^n\tau^{n-1}\int_{t_{k}}^{t_{k+1}}\Big(\int_{t_{k}}^t\|(-A)^{\frac{1-\beta}{2}+\gamma}S^N(t-t_{k})\|^2_{\mathcal{L}(H)}\|(-A)^{\frac{\beta-1}{2}}Q^{\frac{1}{2}}\|^2_{\mathcal{L}_2(H)}\mathrm{d}s\Big)^{\frac{n}{2}}\mathrm{d}t\\
\leq&~ \sum_{n=2}^{\infty}\frac{1}{n!}C^n\|(-A)^{\frac{\beta-1}{2}}Q^{\frac{1}{2}}\|^n_{\mathcal{L}_2(H)}\tau^{n-1}\int_{t_{k}}^{t_{k+1}}(t-t_{k})^{\frac{n(\beta-2\gamma)}{2}}\mathrm{d}t\\
\leq&~ \sum_{n=2}^{\infty}\frac{1}{n!}(C\|(-A)^{\frac{\beta-1}{2}}Q^{\frac{1}{2}}\|_{\mathcal{L}_2(H)}\tau^{\frac{\beta}{2}-\gamma+1}\big)^n\leq  \frac{1}{m},
 \end{align*}
 where we take $0<\tau<\frac{1}{2}$ such that $C\|(-A)^{\frac{\beta-1}{2}}Q^{\frac{1}{2}}\|_{\mathcal{L}_2(H)}\tau^{\frac{\beta}{2}-\gamma}\leq 1$.
 Hence, by Lemma \ref{lemmaT}, we have $\mathbb{P}(\mathcal{R}^{N,3}_{\tau,5}>a)\rightarrow 0$ as $\tau\rightarrow 0.$ Similar argument leads to $\mathbb{P}(-\mathcal{R}^{N,3}_{\tau,5}>a)\rightarrow 0$ as $\tau\rightarrow 0.$
 
Combining the above estimates, we finish the proof.
\end{proof}

\subsection{Proof of Proposition \ref{prop4}}
In this subsection, we present the proof of Proposition \ref{prop4}, based on the regularity estimates of Poisson's equation \eqref{GammaNtau} associated to \eqref{semidiscrete} that are given in subsection \ref{poissontype}.
\begin{proof}[Proof of Proposition \ref{prop4}.]
By noting that $\pi(\phi_{\tau,\epsilon})=0,$ we get
\begin{align*}
&~\mathbb{E}\Big|\frac{1}{m}\sum_{k=0}^{m-1}\phi_{\tau,\epsilon}(X^N_k)\Big|=\mathbb{E}\Big|\frac{1}{m}\sum_{k=0}^{m-1}\Big(\phi_{\tau,\epsilon}(X^N_k)-\pi (\phi_{\tau,\epsilon})\Big)\Big|\\
\leq&~ \mathbb{E}\Big|\frac{1}{m}\sum_{k=0}^{m-1}\Big(\phi_{\tau,\epsilon}(X^N_k)-\pi^N(\phi_{\tau,\epsilon})\Big)\Big|+\big|\pi^N(\phi_{\tau,\epsilon})-\pi (\phi_{\tau,\epsilon})\big|.
\end{align*}
For $\varphi\in\mathcal{C}^2_b(H_N;\mathbb{R}),$ define $\mathcal{L}^{N,\tau}_t$ as  
\begin{align*}
&\mathcal{L}^{N,\tau}_t\varphi(x):=\big\langle A^Nx+S^N(t-\lfloor t\rfloor_{\tau})F^N(X^{N,\tau}_{\lfloor t\rfloor_{\tau}}),D\varphi(x)\big\rangle+\frac{1}{2}\mathrm{Tr}(P^NQ^{\frac{1}{2}}S^N(2(t-\lfloor t\rfloor_{\tau}))Q^{\frac{1}{2}}D^2\varphi(x)).
\end{align*}
Applying It$\mathrm{\hat{o}}$'s formula to $\Gamma^N_{\tau,\epsilon}$ yileds that
\begin{align*}
&~\quad\Gamma^N_{\tau,\epsilon}(X^N_{k+1})-\Gamma^N_{\tau,\epsilon}(X^N_k)\\
&=\int_{t_k}^{t_{k+1}}\langle A^NX^{N,\tau}_t+S^N(t-\lfloor t\rfloor_{\tau})F^N(X^{N,\tau}_{\lfloor t\rfloor_{\tau}}),D\Gamma^N_{\tau,\epsilon}(X^{N,\tau}_t)\rangle\mathrm{d}t\\
&\quad+\int_{t_k}^{t_{k+1}}\big\langle D\Gamma^N_{\tau,\epsilon}(X^{N,\tau}_t),S^N(t-\lfloor t\rfloor_{\tau})\mathrm{d}W(t)\big\rangle+\int_{t_k}^{t_{k+1}}\frac{1}{2}\mathrm{Tr}(P^NQ^{\frac{1}{2}}S^N(2(t-\lfloor t\rfloor_{\tau}))Q^{\frac{1}{2}}D^2\Gamma^N_{\tau,\epsilon}(X^{N,\tau}_t))\mathrm{d}t\\
&=\int_{t_k}^{t_{k+1}}(\mathcal{L}^{N,\tau}_t-\mathcal{L}^N)\Gamma^N_{\tau,\epsilon}(X^{N,\tau}_t)\mathrm{d}t+\int_{t_k}^{t_{k+1}}\mathcal{L}^N\Gamma^N_{\tau,\epsilon}(X^{N,\tau}_t)\mathrm{d}t+\int_{t_k}^{t_{k+1}}\big\langle D\Gamma^N_{\tau,\epsilon}(X^{N,\tau}_t),S^N(t-\lfloor t\rfloor_{\tau})\mathrm{d}W(t)\big\rangle\\
&=\int_{t_k}^{t_{k+1}}(\mathcal{L}^{N,\tau}_t-\mathcal{L}^N)\Gamma^N_{\tau,\epsilon}(X^{N,\tau}_t)\mathrm{d}t+\tau(\phi_{\tau,\epsilon}(X^N_k)-\pi^N(\phi_{\tau,\epsilon}))+\int_{t_k}^{t_{k+1}}\phi_{\tau,\epsilon}(X^{N,\tau}_t)-\phi_{\tau,\epsilon}(X^N_k)\mathrm{d}t\\
&\quad+\int_{t_k}^{t_{k+1}}\big\langle D\Gamma^N_{\tau,\epsilon}(X^{N,\tau}_t),S^N(t-\lfloor t\rfloor_{\tau})\mathrm{d}W(t)\big\rangle,
\end{align*}
where in the last step we use the fact that $\Gamma^N_{\tau,\epsilon}$ satisfies \eqref{GammaNtau}.

Summing over $k=0,\ldots,m-1$ and multiplying by $\tau^{\beta}$ on both sides of the above equation, we obtain
\begin{align*}
&m^{-1}\sum_{k=0}^{m-1}\Big(\phi_{\tau,\epsilon}(X^N_k)-\pi^N(\phi_{\tau,\epsilon})\Big)\\
=~&~\tau^{\beta}\big(\Gamma^N_{\tau,\epsilon}(X^N_m)-\Gamma^N_{\tau,\epsilon}(X^N_0)\big)-\tau^{\beta}\sum_{k=0}^{m-1}\int_{t_k}^{t_{k+1}}\phi_{\tau,\epsilon}(X^{N,\tau}_t)
-\phi_{\tau,\epsilon}(X^N_k)\mathrm{d}t\\
&-\tau^{\beta}\sum_{k=0}^{m-1}\int_{t_k}^{t_{k+1}}(\mathcal{L}^{N,\tau}_t-\mathcal{L}^N)\Gamma^N_{\tau,\epsilon}(X^{N,\tau}_t)\mathrm{d}t-\tau^{\beta}\sum_{k=0}^{m-1}\int_{t_k}^{t_{k+1}}\big\langle D\Gamma^N_{\tau,\epsilon}(X^{N,\tau}_t),S^N(t-\lfloor t\rfloor_{\tau})\mathrm{d}W(t)\big\rangle\\
=:&~\sum_{i=1}^{4}\mathcal{I}_i.
\end{align*}

By Lemma \ref{lemmaphi} (\romannumeral3), the term $\mathcal{I}_1$
 can be estimated as
$$\mathbb{E}|\mathcal{I}_1|\leq C\tau^{\beta}(1+\sup_{k\ge 0}\mathbb{E}\|X^N_k\|^3)\leq C\tau^{\beta}\to0\text{ as }\tau\to0.$$
For the term $\mathcal{I}_2,$ by the mean value theorem, \eqref{phi1} and H\"{o}lder's continuity of $X^{N,\tau}_t$, we obtain
\begin{align*}
\mathbb{E}|\mathcal{I}_2|&\leq C\tau^{\beta}\sum_{k=0}^{m-1}\int_{t_k}^{t_{k+1}}\mathbb{E}\Big[\Big\|\int_0^1D\phi_{\tau,\epsilon}(X^{N,\tau}_t+(1-\theta)(X^N_k-X^{N,\tau}_t))\mathrm{d}\theta\Big\|\|X^{N,\tau}_t-X^N_k\|\Big]\mathrm{d}t\\
&\leq C\tau^{\beta}\sum_{k=0}^{m-1}\int_{t_k}^{t_{k+1}}\big(\sup_{t\ge 0}\mathbb{E}(1+\|X^{N,\tau}_t\|^4)\big)^{\frac{1}{2}}(\mathbb{E}\|X^{N,\tau}_t-X^N_k\|^2)^{\frac{1}{2}}\mathrm{d}t\leq C\tau^{\frac{\beta}{2}}\to0\text{ as }\tau\to0.
\end{align*}
For the term $\mathcal{I}_4,$
\begin{align*}
\mathbb{E}|\mathcal{I}_4|^2&\leq \tau^{2\beta}\int_{0}^{t_m}\mathbb{E}\Big[\sum_{j=1}^N\langle D\Gamma^N_{\tau,\epsilon}(X^{N,\tau}_t),S^N(t-\lfloor t\rfloor_{\tau})Q^{\frac{1}{2}}e_j\rangle ^2\Big]\mathrm{d}t\\
&\leq C\tau^{2\beta}\int_{0}^{t_m}\mathbb{E}\|(-A)^{\frac{1-\beta}{2}}D\Gamma^N_{\tau,\epsilon}(X^{N,\tau}_t)\|^2\sum_{j=1}^N\|(-A)^{\frac{\beta-1}{2}}Q^{\frac{1}{2}}e_j\|^2\mathrm{d}t\leq C\tau^{\beta}\to0\text{ as }\tau\to0.
\end{align*}
For the term $\mathcal{I}_3,$
\begin{align*}
\mathbb{E}|\mathcal{I}_3|&\leq \tau^{\beta}\sum_{k=0}^{m-1}\int_{t_k}^{t_{k+1}}\mathbb{E}\Big[\Big|\big\langle S^N(t-\lfloor t\rfloor_{\tau})F^N(X^{N,\tau}_{\lfloor t\rfloor_{\tau}})-F^N(X^{N,\tau}_t),D\Gamma^N_{\tau,\epsilon}(X^{N,\tau}_t)\big\rangle\Big|\\
&\quad+\frac{1}{2}\Big|\mathrm{Tr}(P^NQ^{\frac{1}{2}}(S^N(2(t-\lfloor t\rfloor_{\tau}))-\mathrm{Id})Q^{\frac{1}{2}}D^2\Gamma^N_{\tau,\epsilon}(X^{N,\tau}_t))\Big|\Big]\mathrm{d}t=:\mathcal{I}^1_3+\mathcal{I}^2_3.
\end{align*}
H\"{o}lder's continuity of $X^{N,\tau}_t$ and the regularity estimate of $D\Gamma^N_{\tau,\epsilon}$ imply that $\mathcal{I}^1_3\leq C\tau^{\frac{\beta}{2}}\to0\text{ as }\tau\to0.$ 
 Note that the assumption that $A$ commutes with $Q$ leads to
\begin{align*}
&\quad\mathrm{Tr}((P^NQ^{\frac{1}{2}}(S^N(2(t-\lfloor t\rfloor_{\tau}))-\mathrm{Id})Q^{\frac{1}{2}}D^2\Gamma^N_{\tau,\epsilon}(X^{N,\tau}_t)))\\
&=\mathrm{Tr}((S^N(2(t-\lfloor t\rfloor_{\tau}))-\mathrm{Id})Q^{\frac{1}{2}}D^2\Gamma^N_{\tau,\epsilon}(X^{N,\tau}_t)P^NQ^{\frac{1}{2}})\\
&=\mathrm{Tr}((-A)^{-\beta+\gamma}(S^N(2(t-\lfloor t\rfloor_{\tau}))-\mathrm{Id})(-A)^{\frac{\beta-1}{2}}Q^{\frac{1}{2}}(-A)^{\frac{1+\beta}{2}-\gamma}D^2\Gamma^N_{\tau,\epsilon}(X^{N,\tau}_t)(-A)^{\frac{1-\beta}{2}}P^N(-A)^{\frac{\beta-1}{2}}Q^{\frac{1}{2}}).
\end{align*}
By \eqref{phi2} and \cite[(2.2), (2.3) and (2.4)]{czh20},
the term $\mathcal{I}^2_3$ can be estimated as
\begin{align*}
\mathcal{I}^2_3\leq&~ \frac{1}{2}\tau^{\beta}\sum_{k=0}^{m-1}\int_{t_k}^{t_{k+1}}\|(-A)^{-\beta+\gamma}(S^N(2(t-\lfloor t\rfloor _{\tau}))-\mathrm{Id})\|_{\mathcal{L}(H)}\|(-A)^{\frac{\beta-1}{2}}Q^{\frac{1}{2}}\|^2_{\mathcal{L}_2(H)}\times\\
&\qquad\qquad\qquad\mathbb{E}\|(-A)^{\frac{1+\beta}{2}-\gamma}D^2\phi_{\tau,\epsilon}(X^{N,\tau}_t)(-A)^{\frac{1-\beta}{2}}\|_{\mathcal{L}(H)}\mathrm{d}t\\
\leq&~ C\tau^{\beta-\gamma}\to0\text{ as }\tau\to0.
\end{align*}

Moreover, by Remark \ref{remark2} and \eqref{phi1}, we have $|\pi^N(\phi_{\tau,\epsilon})-\pi(\phi_{\tau,\epsilon})|\leq CN^{-\beta}.$ Therefore, we get
\begin{align*}
\mathbb{E}|\tilde{\mathcal{R}}^N_{\tau}|=\mathbb{E}\Big|\tau^{-\epsilon}m^{-1}\sum_{k=0}^{m-1}\phi_{\tau,\epsilon}(X^N_k)\Big|\leq C\tau^{-\epsilon}(N^{-\beta}+\tau^{\frac{\beta}{2}}).
\end{align*}
Since $\epsilon$ satisfies that $\epsilon<\beta\alpha,$ we get the desired result.
The proof is completed.
\end{proof}

\section{Appendix}
In this section, we list proofs of Propositions \ref{propfull} and \ref{propcon}.
\subsection{Proof of Proposition \ref{propfull}}
\begin{proof}
For the existence and uniqueness of the invariant measure $\pi^N_{\tau}$, we refer to \cite[Theorem 4.7]{czh20}.
Similarly to the proof of \cite[Lemma $4.1$]{czh20},
$$\sup_{k\in\mathbb{N}}\mathbb{E}\|X^N_k\|^2\leq C(\|(-A)^{\frac{\beta-1}{2}}Q^{\frac{1}{2}}\|_{\mathcal{L}_2(H)},K)(1+\mathbb{E}\|X^N_0\|^2)$$
can be proved by applying \eqref{semigroup3}. 
We show by induction that the $p$th moment of the numerical solution of the full discretization is bounded for $p\ge 2$ being an integer, since the non-integer case can be obtained by H\"{o}lder's inequality.
Assume that $$\sup_{k\in\mathbb{N}}\mathbb{E}\|X^N_{k}\|^{2(p-1)}\leq C(\|(-A)^{\frac{\beta-1}{2}}Q^{\frac{1}{2}}\|_{\mathcal{L}_2(H)},K,p-1)(1+\mathbb{E}\|X^N_0\|^{2(p-1)}).$$ Now we prove $$\sup_{k\in\mathbb{N}}\mathbb{E}\|X^N_{k}\|^{2p}\leq C(\|(-A)^{\frac{\beta-1}{2}}Q^{\frac{1}{2}}\|_{\mathcal{L}_2(H)},K,p)(1+\mathbb{E}\|X^N_0\|^{2p}).$$

Let $V^N_k:=\int_0^{t_k}S^N(t_k-\lfloor s\rfloor_{\tau})P^N\mathrm{d}W(s), Z^N_k:=X^N_k-V^N_k.$ Then
\begin{align*}
Z^N_{k+1}=S^N(\tau)Z^N_k+S^N(\tau)F^N(Z^N_k+V^N_k)\tau,\quad Z^N_0=X^N_0.
\end{align*} 
By the Burkholder--Davis--Gundy inequality and \eqref{semigroup1}, we get
\begin{align*}
\mathbb{E}\|V^N_k\|^{2p}&\leq \Big(C(p)\int_0^{t_k}\mathbb{E}\|S^N(t_k-\lfloor s\rfloor_{\tau})Q^{\frac{1}{2}}\|^2_{\mathcal{L}_2(H)}\mathrm{d}s\Big)^p\\
&\leq \Big(C(p)\int_0^{t_k}\|(-A)^{\frac{1-\beta}{2}}S^N(t_k-\lfloor s\rfloor_{\tau})\|^2_{\mathcal{L}(H)}\|(-A)^{\frac{\beta-1}{2}}Q^{\frac{1}{2}}\|^2_{\mathcal{L}_2(H)}\mathrm{d}s\Big)^p\\
&\leq \Big(C(p)\int_0^{t_k}(t_k-\lfloor s\rfloor_{\tau})^{\beta-1}e^{-\lambda_1(t_k-\lfloor s\rfloor_{\tau})}\mathrm{d}s\Big)^p\leq C(p).
\end{align*} 
Hence, by Young's inequality, we obtain
\begin{align*}
&\mathbb{E}\|Z^N_{k+1}\|^{2p}\\
\leq&~ e^{-2\lambda_1\tau p}\mathbb{E}\Big(\|Z^N_k\|^2+2\tau\langle Z^N_k,F^N(Z^N_k+V^N_k)-F^N(V^N_k)\rangle+2\tau\langle Z^N_k,F^N(V^N_k) \rangle\\
&+2\tau^2\|F^N(Z^N_k+V^N_k)-F^N(V^N_k)\|^2+2\tau^2\|F^N(V^N_k)\|^2\Big)^{p}\\
\leq&~ e^{-2\lambda_1\tau p}\mathbb{E}\Big((1+2\tau K+2\tau\gamma_1+2\tau^2L_F^2)\|Z^N_k\|^2+2(\tau^2+C(\gamma_1)\tau)\|F^N(V^N_k)\|^2\Big)^p\\
\leq&~ e^{-2\lambda_1\tau p}(1+2\tau K+2\tau \gamma_1+2\tau^2L^2_F)^p\mathbb{E}\|Z^N_k\|^{2p}\\
&+e^{-2\lambda_1\tau p}\sum_{j=0}^{p-1}C^p_j\mathbb{E}\Big[\big((1+2\tau K+2\tau \gamma_1+2\tau^2L^2_F)\|Z^N_k\|^{2}\big)^{j}\big(2\tau(\tau+C(\gamma_1))\|F(V^N_k)\|^2\big)^{p-j}\Big],
\end{align*}
where $\gamma_1>0$ is a parameter to be determined and $C^p_j$ is the combinatorial number. By taking $\gamma_1=\frac{\lambda_1-K}{4}$ and noticing 
$
2\tau^2L^2_F\leq \frac{\lambda_1-K}{2}\tau,
$
we get $1+2\tau K+2\tau\gamma_1+2\tau^2L_F^2\leq e^{(\lambda_1+K)\tau}.$ 
Let $\bar{\gamma}_1>0$ be a constant to be determined.
Applying Young's inequality again leads to
\begin{align*}
&~\mathbb{E}\|Z^N_{k+1}\|^{2p}\\
\leq&~ e^{-2\lambda_1\tau p}\Big[e^{(\lambda_1+K)\tau p}\mathbb{E}\|Z^N_k\|^{2p}+\bar{\gamma}_1\tau e^{(\lambda_1+K)\tau p}\mathbb{E}\|Z^N_k\|^{2p}
+C(\bar{\gamma}_1)\tau (\tau+C(\gamma_1))^p\mathbb{E}\|F(V^N_k)\|^{2p}\\
&
\qquad\qquad +\sum_{j=0}^{p-2}C^p_j\mathbb{E}\big[\big(e^{(\lambda_1+K)\tau}\|Z^N_k\|^{2}\big)^{j}
\big(2\tau(\tau+C(\gamma_1))\|F(V^N_k)\|^2\big)^{p-j}\big]\Big].
\end{align*}
Taking $\bar{\gamma}_1=\frac{(\lambda_1-K)p}{2}$, by the induction assumption and Assumption \ref{assumpX_0}, we have for $p\ge 2,$
\begin{align*}
&\quad \mathbb{E}\|Z^N_{k+1}\|^{2p}\\
&\leq e^{-\frac{\lambda_1-K}{2}\tau p}\mathbb{E}\|Z^N_k\|^{2p}+C(\bar{\gamma}_1)\tau e^{-2\lambda_1\tau p}(\tau+C(\gamma_1))^p\mathbb{E}\|F(V^N_k)\|^{2p}\\
&\quad+C(p)\tau e^{-2\lambda_1\tau p}\sum_{j=0}^{p-2}e^{(\lambda_1+K)\tau j}\mathbb{E}\big[\|Z^N_k\|^{2j}(\|V^N_k\|^{2(p-j)}+1)\big]\\
&\leq e^{-\frac{\lambda_1-K}{2}\tau p}\mathbb{E}\|Z^N_k\|^{2p}+C(\|(-A)^{\frac{\beta-1}{2}}Q^{\frac{1}{2}}\|_{\mathcal{L}_2(H)},K,p)(1+\mathbb{E}\|X^N_0\|^{2p})\tau e^{-(\lambda_1-K)\tau p}\\
&\leq e^{-\frac{\lambda_1-K}{2}\tau p(k+1)}\mathbb{E}\|Z^N_0\|^{2p}+C(\|(-A)^{\frac{\beta-1}{2}}Q^{\frac{1}{2}}\|_{\mathcal{L}_2(H)},K,p)(1+\mathbb{E}\|X^N_0\|^{2p})\tau\sum_{j=0}^{k}e^{-\frac{\lambda_1-K}{2}\tau p j}\\
&\leq C(\|(-A)^{\frac{\beta-1}{2}}Q^{\frac{1}{2}}\|_{\mathcal{L}_2(H)},K,p)(1+\mathbb{E}\|X^N_0\|^{2p}).
\end{align*}
Hence, $\mathbb{E}\|X^N_{k+1}\|^{2p}\leq C(\|(-A)^{\frac{\beta-1}{2}}Q^{\frac{1}{2}}\|_{\mathcal{L}_2(H)},K,p)(1+\mathbb{E}\|X^N_0\|^{2p}).$
By the mild form of the full discretization, Minkowski's inequality, the Burkholder--Davis--Gundy inequality and \eqref{semigroup3}, we have
\begin{align*}
&~\|(-A)^{\frac{\beta}{2}}X^{N,\tau}_t\|_{L^{2p}(\Omega;H)}\\\leq&~ \|(-A)^{\frac{\beta}{2}}S^N(t)X^N_0\|_{L^{2p}(\Omega;H)}+\int_0^{t}\|(-A)^{\frac{\beta}{2}}S^N(t-\lfloor s\rfloor_{\tau})F^N(X^{N,\tau}_{\lfloor s\rfloor_{\tau}})\|_{L^{2p}(\Omega;H)}\mathrm{d}s\\
&+\big\|\int_0^{t}(-A)^{\frac{\beta}{2}}S^N(t-\lfloor s\rfloor_{\tau})\mathrm{d}W(s)\big\|_{L^{2p}(\Omega;H)}\\
\leq&~ \|X^N_0\|_{L^{2p}(\Omega;\dot{H}^{\beta})}+\int_0^{t}\|(-A)^{\frac{\beta}{2}}S^N(t-\lfloor s\rfloor_{\tau})\|_{\mathcal{L}(H)}\|F^N(X^{N,\tau}_{\lfloor s\rfloor_{\tau}})\|_{L^{2p}(\Omega;H)}\mathrm{d}s\\
&+C(p)\Big(\int_0^{t}\|(-A)^{\frac{1}{2}}S^N(t-\lfloor s\rfloor_{\tau})\|^2_{\mathcal{L}(H)}\|(-A)^{\frac{\beta-1}{2}}Q^{\frac{1}{2}}\|^2_{\mathcal{L}_2(H)}\mathrm{d}s\Big)^{\frac{1}{2}}\\\leq&~ C(\|(-A)^{\frac{\beta-1}{2}}Q^{\frac{1}{2}}\|_{\mathcal{L}_2(H)},K,p)(1+\|X^N_0\|_{L^{2p}(\Omega;\dot{H}^{\beta})}).
\end{align*}

Now we prove the H\"{o}lder continuity of $X^{N,\tau}_t.$ By \eqref{contiversion}, 
we have
\begin{align*}
\mathbb{E}\|X^{N,\tau}_t-X^{N,\tau}_s\|^2&\leq C(t-s)^{\beta}\mathbb{E}\|X^{N,\tau}_s\|^2_{\beta}+C(t-s)\int_s^t\mathbb{E}(1+L^2_F\|X^{N,\tau}_{\lfloor r\rfloor_{\tau}}\|^2)\mathrm{d}r\\
&\quad+C\int_s^t\|(-A)^{\frac{1-\beta}{2}}S^N(t-\lfloor r\rfloor_{\tau})(-A)^{\frac{\beta-1}{2}}Q^{\frac{1}{2}}\|^2_{\mathcal{L}_2(H)}\mathrm{d}r\\
&\leq C(t-s)^{\beta}+C(t-s)^2+C\int_s^t(t-r)^{-1+\beta}\mathrm{d}r\\
&\leq C(t-s)^{\beta},\quad 0<t-s\leq 1.
\end{align*}
For the case of $t-s>1,$ the above H\"older continuity can be obtained by the finiteness of the second moment of $X^{N,\tau}_t.$ 

Noting 
\begin{align*}
X^N_{k+1}-\bar{X}^N_{k+1}=S^N(\tau)(X^N_k-\bar{X}^N_k)+\tau S^N(\tau)(F^N(X^N_k)-F^N(\bar{X}^N_k))
\end{align*}
and Assumption \ref{assumpF_1},
we have
\begin{align*}
&\|X^N_{k+1}-\bar{X}^N_{k+1}\|^2\\
\leq&~ e^{-2\lambda_1\tau} \big(\|X^N_k-\bar{X}^N_k\|^2+2\tau\langle X^N_k-\bar{X}^N_k,F^N(X^N_k)-F^N(\bar{X}^N_k)\rangle
+\tau^2\|F^N(X^N_k)-F^N(\bar{X}^N_k)\|^2\big)\\
\leq&~ e^{-2\lambda_1\tau}(1+2K\tau+L_F^2\tau^2)\|X^N_k-\bar{X}^N_k\|^2.
\end{align*}
Taking the $p$th power and the expectation on both sides of the above equation leads to
\begin{align*}
\mathbb{E}\|X^N_{k+1}-\bar{X}^N_{k+1}\|^{2p}&\leq e^{-2\lambda_1p\tau}(1+2K\tau+L_F^2\tau^2)^p\mathbb{E}\|X^N_k-\bar{X}^N_k\|^{2p}\\
&\leq (1+2K\tau+L_F^2\tau^2)^{(k+1)p}e^{-2\lambda_1\tau (k+1)p}\mathbb{E}\|X^N_0-\bar{X}^N_0\|^{2p}\\
&\leq (e^{-(\lambda_1-K)\tau})^{(k+1)p}\mathbb{E}\|X^N_0-\bar{X}^N_0\|^{2p},
\end{align*}
where in the last step we use the assumption $\tau\leq \frac{\lambda_1-K}{4L_F^2}.$
Finally, \eqref{full1} can be proved as
\begin{align*}
|\mathbb{E}\psi(X^N_k)-\pi^N_{\tau}(\psi)|&=|\mathbb{E}\psi(X^N_k)-\psi(\tilde{X}^N_k)|\leq \|\psi\|_{1,\infty}(\mathbb{E}\|X^N_k-\tilde{X}^N_k\|^2)^{\frac{1}{2}}\\
&\leq C\|\psi\|_{1,\infty}(1+\mathbb{E}\|X^N_0\|^2)e^{-\frac{1}{2}(\lambda_1-K)t_k},
\end{align*}
where $\tilde{X}^N_k$ is the solution of \eqref{fulldiscrete} with the law of the initial datum $\tilde{X}^N_0$ being the invariant measure $\pi^N_{\tau}.$
The proof is finished.
\end{proof}

\subsection{Proof of Proposition \ref{propcon}}
\begin{proof}
We introduce an auxiliary process $Y^N_t$ which satisfies the following SPDE
\begin{align*}
\mathrm{d}Y^N_t=A^NY^N_t\mathrm{d}t+S^N(t-\lfloor t\rfloor_{\tau})F^N(X^{N,\tau}_{\lfloor t\rfloor_{\tau}})\mathrm{d}t+P^N\mathrm{d}W(t),\quad Y^N_0=P^NX_0.
\end{align*}
Then \eqref{semigroup2} and \eqref{semigroup3} imply that
\begin{align*}
\mathbb{E}\|X^{N,\tau}_t-Y^N_t\|^2&=\mathbb{E}\big\|\int_0^tS^{N}(t-s)(S^N(s-\lfloor s\rfloor_{\tau})-\mathrm{Id})\mathrm{d}W(s)\big\|^2\\
&\leq\int_0^t\mathbb{E}\|S^N(t-s)(S^N(s-\lfloor s\rfloor_{\tau})-\mathrm{Id})Q^{\frac{1}{2}}\|_{\mathcal{L}_2(H)}^2\mathrm{d}s\\
&\leq \|(-A)^{\frac{\beta-1}{2}}Q^{\frac{1}{2}}\|^2_{\mathcal{L}_2(H)}\int_0^t\|(-A^N)^{\frac{1}{2}}S^N(t-s)(-A^N)^{-\frac{\beta}{2}}(S^N(s-\lfloor s\rfloor_{\tau})-\mathrm{Id})\|^2_{\mathcal{L}(H)}\mathrm{d}s\\
&\leq C\tau^{\beta}.
\end{align*}
It\^{o}'s formula, together with Poincar\'e's inequality and Assumption \ref{assumpF_1} yields that
\begin{align*}
&\frac{1}{2}\mathrm{d}\|{Y}^N_t-X^N(t)\|^2\\
=&~\langle A^N({Y}^N_t-X^N(t)),{Y}^N_t-X^N(t)\rangle\mathrm{d}t
+\big\langle S^N(t-\lfloor t\rfloor_{\tau})F^N({X}^{N,\tau}_{\lfloor t\rfloor_{\tau}})-F^N(X^N(t)),{Y}^N_t-X^N(t)\big\rangle \mathrm{d}t\\
\leq& -(\lambda_1-K)\|{Y}^N_t-X^N(t)\|^2\mathrm{d}t+\big\langle S^N(t-\lfloor t\rfloor_{\tau})F^N({X}^{N,\tau}_{\lfloor t\rfloor_{\tau}})-F^N({X}^{N,\tau}_{\lfloor t\rfloor_{\tau}}),{Y}^N_t-X^N(t)\big\rangle \mathrm{d}t\\
&+\big\langle F^N(X^{N,\tau}_{\lfloor t\rfloor_{\tau}})-F^N(X^{N,\tau}_t),Y^N_t-X^N(t)\big\rangle\mathrm{d}t+\big\langle F^N(X^{N,\tau}_t)-F^N(Y^N_t),
 Y^N_t-X^N(t)\big\rangle\mathrm{d}t\\
\leq& -(\lambda_1-K)\|{Y}^N_t-X^N(t)\|^2\mathrm{d}t+\big\langle (S^N(t-\lfloor t\rfloor_{\tau})-\mathrm{Id})F^N({X}^{N,\tau}_{\lfloor t\rfloor_{\tau}}),{Y}^N_t-X^N(t)\big\rangle \mathrm{d}t\\
&+L_F\|{Y}^N_t-X^N(t)\|(\|{X}^{N,\tau}_t-{X}^{N,\tau}_{\lfloor t\rfloor_{\tau}}\|+\|X^{N,\tau}_t-Y^N_t\|)\mathrm{d}t.
\end{align*}
Hence, 
\begin{align}\label{integral}
\|Y^N_t-X^N(t)\|^2&\leq -2\int_0^t(\lambda_1-K)\|Y^N_s-X^N(s)\|^2\mathrm{d}s+2J_1\notag\\
&\quad+\int_0^t2L_F\|Y^N_s-X^N(s)\|
 (\|X^{N,\tau}_s-X^{N,\tau}_{\lfloor s\rfloor_{\tau}}\|+\|X^{N,\tau}_s-Y^N_s\|)\mathrm{d}s,
\end{align}
where
\begin{align*}
&J_1:=\int_0^t\langle (S^N(s-\lfloor s\rfloor_{\tau})-\mathrm{Id})F^N(X^{N,\tau}_{\lfloor s\rfloor_{\tau}}),Y^N_s-X^N(s)\rangle\mathrm{d}s.
\end{align*}
By mild forms of $Y^N_s$ and $X^N(s)$, we have
\begin{align*}
Y^N_s-X^N(s)&=\int_0^sS^N(s-\lfloor r\rfloor_{\tau})F^N(X^{N,\tau}_{\lfloor r\rfloor_{\tau}})-S^N(s-r)F^N(X^N(r))\mathrm{d}r\\
&=\int_0^sS^N(s-r)(F^N(X^{N,\tau}_r)-F^N(X^N(r)))\mathrm{d}r\\
&\quad+\int_0^sS^N(s-r)(F^N(X^{N,\tau}_{\lfloor r\rfloor_{\tau}})-F^N(X^{N,\tau}_r))\mathrm{d}r\\
&\quad+\int_0^s(S^N(s-{\lfloor r\rfloor_{\tau}})-S^N(s-r))F^N(X^{N,\tau}_{\lfloor r\rfloor_{\tau}})\mathrm{d}r.
\end{align*}
Hence, we split $J_1$ into three parts correspondingly, denoted by $J_{1,1},J_{1,2},J_{1,3}.$ 

For the term $J_{1,1},$
\begin{align*}
J_{1,1}&=\int_0^t\Big\langle (-A^N)^{-\frac{\beta}{2}}(S^N(s-\lfloor s\rfloor_{\tau})-\mathrm{Id})F^N(X^{N,\tau}_{\lfloor s\rfloor_{\tau}}),\\&\qquad \quad\int_0^s(-A^N)^{\frac{\beta}{2}}S^N(s-r)(F^N(X^{N,\tau}_r)-F^N(X^N(r)))\mathrm{d}r\Big\rangle\mathrm{d}s\\
&\leq \gamma\int_0^t\int_0^s(s-r)^{-\frac{\beta}{2}}e^{-\frac{\lambda_1(s-r)}{2}}\|X^{N,\tau}_r-X^N(r)\|^2\mathrm{d}r\mathrm{d}s\\
&\quad+\frac{1}{4\gamma}\int_0^t\int_0^s(s-r)^{-\frac{\beta}{2}}e^{-\frac{\lambda_1(s-r)}{2}}\tau^{\beta}\|F^N(X^{N,\tau}_{\lfloor s\rfloor_{\tau}})\|^2\mathrm{d}r\mathrm{d}s\\
&\leq C(\beta,\lambda_1)\gamma\int_0^t\|X^{N,\tau}_r-X^N(r)\|^2\mathrm{d}r
+\frac{1}{4\gamma}C(\beta,\lambda_1)\tau^{\beta}\int_0^tL_F^2(1+\|X^{N,\tau}_{\lfloor s\rfloor_{\tau}}\|^2)\mathrm{d}s,
\end{align*}
where $C(\beta,\lambda_1)=\int_0^{\infty}s^{-\frac{\beta}{2}}e^{-\frac{\lambda_1s}{2}}\mathrm{d}s<\infty$ and in the last step we have used the Fubini theorem. Take $\gamma$ small enough such that $ C(\beta,\lambda_1)\gamma\leq \frac{\lambda_1-K}{8}$, then 
\begin{align*}
\mathbb{E}J_{1,1}\leq &~\frac{\lambda_1-K}{4}\int_0^t\mathbb{E}\|Y^N_r-X^N(r)\|^2\mathrm{d}r+\frac{\lambda_1-K}{4}\int_0^t\mathbb{E}\|X^{N,\tau}_r-Y^N_r\|^2\mathrm{d}r\\
&+C\tau^{\beta}\int_0^tL_F^2(1+\mathbb{E}\|X^{N,\tau}_{\lfloor s\rfloor_{\tau}}\|^2)\mathrm{d}s\\
\leq& ~\frac{\lambda_1-K}{4}\int_0^t\mathbb{E}\|Y^N_r-X^N(r)\|^2\mathrm{d}r+C\tau^{\beta}t.
\end{align*}

Terms $J_{1,2}$ and $J_{1,3}$ can be proved similarly. To be specific,
for the term $J_{1,2},$ 
\begin{align*}
J_{1,2}&=\int_0^t\Big\langle (-A^N)^{-\frac{\beta}{2}}(S^N(s-\lfloor s\rfloor_{\tau})-\mathrm{Id})F^N(X^{N,\tau}_{\lfloor s\rfloor_{\tau}}),\\&\qquad \quad\int_0^s(-A^N)^{\frac{\beta}{2}}S^N(s-r)(F^N(X^{N,\tau}_{\lfloor r\rfloor_{\tau}})-F^N(X^{N,\tau}_r))\mathrm{d}r\Big\rangle\mathrm{d}s\\
&\leq C(\beta,\lambda_1)L^2_F\int_0^t\|X^{N,\tau}_r-X^{N,\tau}_{\lfloor r\rfloor_{\tau}}\|^2\mathrm{d}r+C(\beta,\lambda_1)\tau^{\beta}\int_0^tL^2_F(1+\|X^{N,\tau}_{\lfloor s\rfloor_{\tau}}\|^2)\mathrm{d}s.
\end{align*}
Hence, by the H\"older's continuity of $X^{N,\tau}_r,$ we get
$
\mathbb{E}J_{1,2}\leq C\tau^{\beta}t.
$
For the term $J_{1,3},$ when $\beta<1,$ we have
\begin{align*}
J_{1,3}&=\int_0^t\Big\langle (-A^N)^{-\frac{\beta}{2}}(S^N(s-\lfloor s\rfloor_{\tau})-\mathrm{Id})F^N(X^{N,\tau}_{\lfloor s\rfloor_{\tau}}),\\
&\qquad\quad\int_0^s(-A^N)^{\beta}S^N(s-r)(-A^N)^{-\frac{\beta}{2}}(S^N(r-\lfloor r\rfloor_{\tau})-\mathrm{Id})F^N(X^{N,\tau}_{\lfloor r\rfloor_{\tau}})\mathrm{d}r\Big\rangle\mathrm{d}s\\
&\leq 2C(\beta,\lambda_1)\tau^{\beta}\int_0^t
\|F^N(X^{N,\tau}_{\lfloor s\rfloor_{\tau}})\|^2\mathrm{d}s,
\end{align*}
and when $\beta=1,$ we deduce from chain's rule that
\begin{align*}
J_{1,3}&=\int_0^t\Big\langle (-A^N)^{-\frac{1}{2}}(S^N(s-\lfloor s\rfloor_{\tau})-\mathrm{Id})(-A^N)^{\frac{1}{2}}F^N(X^{N,\tau}_{\lfloor s\rfloor_{\tau}}),\\
&\qquad \quad\int_0^sS^N(s-r)(-A^N)^{-\frac{1}{2}}(S^N(r-\lfloor r\rfloor_{\tau})-\mathrm{Id})(-A^N)^{\frac{1}{2}}F^N(X^{N,\tau}_{\lfloor r\rfloor_{\tau}})\mathrm{d}r\Big\rangle\mathrm{d}s\\
&\leq C\tau\int_0^t\int_0^se^{-\lambda_1(s-r)}\|(-A^N)^{\frac{1}{2}}F^N(X^{N,\tau}_{\lfloor s\rfloor_{\tau}})\|^2+\|(-A^N)^{\frac{1}{2}}F^N(X^{N,\tau}_{\lfloor r\rfloor_{\tau}})\|^2)\mathrm{d}r\mathrm{d}s\\
&\leq C\tau\int_0^t\int_0^se^{-\lambda_1(s-r)}L^2_F(\|X^{N,\tau}_{\lfloor s\rfloor_{\tau}}\|^2_1+\|X^{N,\tau}_{\lfloor r\rfloor_{\tau}}\|^2_1)\mathrm{d}r\mathrm{d}s.
\end{align*}
Hence, we have
$\mathbb{E}J_{1,3}\leq C\tau^{\beta}t.$

Taking the expectation on both sides of \eqref{integral} yields
\begin{align*}
\mathbb{E}\|Y^N_t-X^N(t)\|^2\leq -(\lambda_1-K)\int_0^t\mathbb{E}\|Y^N_s-X^N(s)\|^2\mathrm{d}s+C\tau^{\beta}t.
\end{align*}
Applying Gr{\"o}nwall's inequality leads to
\begin{align*}
\sup_{t\ge 0}\mathbb{E}\|Y^N_t-X^N(t)\|^2\leq C\tau^{\beta}.
\end{align*}
Hence, the triangle inequality gives
\begin{align*}
\sup_{k\in\mathbb{N}}\mathbb{E}\|{X}^N_{t_k}-X^N(t_k)\|^2\leq C(\|(-A)^{\frac{\beta-1}{2}}Q^{\frac{1}{2}}\|_{\mathcal{L}_2(H)},K)(1+\mathbb{E}\|X_0\|_{\beta}^2)\tau^{\beta}.
\end{align*}
The proof is finished.
\end{proof}

\bibliographystyle{plain}
\bibliography{clt.bib}

\begin{thebibliography}{10}

\bibitem{highorder}
A.~Abdulle, G.~Vilmart, and K.~C. Zygalakis.
\newblock High order numerical approximation of the invariant measure of
  ergodic {SDE}s.
\newblock {\em SIAM J. Numer. Anal.}, 52(4):1600--1622, 2014.

\bibitem{measure}
K.~B. Athreya and S.~N. Lahiri.
\newblock {\em Measure Theory and Probability Theory}.
\newblock Springer Texts in Statistics. Springer, New York, 2006.

\bibitem{B12}
C.~Br\'{e}hier.
\newblock Strong and weak orders in averaging for {SPDE}s.
\newblock {\em Stochastic Process. Appl.}, 122(7):2553--2593, 2012.

\bibitem{B14}
C.~Br\'{e}hier.
\newblock Approximation of the invariant measure with an {E}uler scheme for
  stochastic {PDE}s driven by space-time white noise.
\newblock {\em Potential Anal.}, 40(1):1--40, 2014.

\bibitem{BK17}
C.~Br\'{e}hier and M.~Kopec.
\newblock Approximation of the invariant law of {SPDE}s: error analysis using a
  {P}oisson equation for a full-discretization scheme.
\newblock {\em IMA J. Numer. Anal.}, 37(3):1375--1410, 2017.

\bibitem{pde01}
S.~Cerrai.
\newblock {\em Second Order {PDE}'s in Finite and Infinite Dimension: A
  Probabilistic Approach}, volume 1762 of {\em Lecture Notes in Mathematics}.
\newblock Springer-Verlag, Berlin, 2001.

\bibitem{ZGQ}
K.~Chang.
\newblock {\em Methods in Nonlinear Analysis}.
\newblock Springer Monographs in Mathematics. Springer-Verlag, Berlin, 2005.

\bibitem{CHW17}
C.~Chen, J.~Hong, and X.~Wang.
\newblock Approximation of invariant measure for damped stochastic nonlinear
  {S}chr\"{o}dinger equation via an ergodic numerical scheme.
\newblock {\em Potential Anal.}, 46(2):323--367, 2017.

\bibitem{czh20}
Z.~Chen, S.~Gan, and X.~Wang.
\newblock A full-discrete exponential {E}uler approximation of the invariant
  measure for parabolic stochastic partial differential equations.
\newblock {\em Appl. Numer. Math.}, 157:135--158, 2020.

\bibitem{CHS21}
J.~Cui, J.~Hong, and L.~Sun.
\newblock Weak convergence and invariant measure of a full discretization for
  parabolic {SPDE}s with non-globally {L}ipschitz coefficients.
\newblock {\em Stochastic Process. Appl.}, 134:55--93, 2021.

\bibitem{Daprato}
G.~Da~Prato and J.~Zabczyk.
\newblock {\em Stochastic Equations in Infinite Dimensions}, volume 152 of {\em
  Encyclopedia of Mathematics and its Applications}.
\newblock Cambridge University Press, Cambridge, second edition, 2014.

\bibitem{hwbook}
J.~Hong and X.~Wang.
\newblock {\em Invariant Measures for Stochastic Nonlinear {S}chr\"{o}dinger
  Equations: Numerical Approximations and Symplectic Structures}, volume 2251
  of {\em Lecture Notes in Mathematics}.
\newblock Springer, Singapore, 2019.

\bibitem{HWZ17}
J.~Hong, X.~Wang, and L.~Zhang.
\newblock Numerical analysis on ergodic limit of approximations for stochastic
  {NLS} equation via multi-symplectic scheme.
\newblock {\em SIAM J. Numer. Anal.}, 55(1):305--327, 2017.

\bibitem{clt12}
T.~Komorowski and A.~Walczuk.
\newblock Central limit theorem for {M}arkov processes with spectral gap in the
  {W}asserstein metric.
\newblock {\em Stochastic Process. Appl.}, 122(5):2155--2184, 2012.

\bibitem{cltsde}
J.~Lu, Y.~Tan, and L.~Xu.
\newblock Central limit theorem and self-normalized {C}ram\'er-type moderate
  deviation for {E}uler-{M}aruyama scheme.
\newblock arXiv: 2012.04328, 2021.

\bibitem{siam10}
J.~C. Mattingly, A.~M. Stuart, and M.~V. Tretyakov.
\newblock Convergence of numerical time-averaging and stationary measures via
  {P}oisson equations.
\newblock {\em SIAM J. Numer. Anal.}, 48(2):552--577, 2010.

\bibitem{MDA74}
D.~L. McLeish.
\newblock Dependent central limit theorems and invariance principles.
\newblock {\em Ann. Probability}, 2:620--628, 1974.

\bibitem{MV}
R.~Meise and D.~Vogt.
\newblock {\em Introduction to functional analysis}, volume~2 of {\em Oxford
  Graduate Texts in Mathematics}.
\newblock The Clarendon Press, Oxford University Press, New York, 1997.
\newblock Translated from the German by M. S. Ramanujan and revised by the
  authors.

\bibitem{AAP12}
G.~Pag\`es and F.~Panloup.
\newblock Ergodic approximation of the distribution of a stationary diffusion:
  rate of convergence.
\newblock {\em Ann. Appl. Probab.}, 22(3):1059--1100, 2012.

\bibitem{wxj}
X.~Wang.
\newblock An efficient explicit full-discrete scheme for strong approximation
  of stochastic {A}llen-{C}ahn equation.
\newblock {\em Stochastic Process. Appl.}, 130(10):6271--6299, 2020.

\end{thebibliography}

\end{document}